\documentclass[11pt,twoside]{amsart}

\usepackage[left=.9in,top=.9in,right=.9in,bottom=.7in]{geometry}

\usepackage{graphicx}
\usepackage{amsmath}
\usepackage{amsthm}
\usepackage{amssymb}
\usepackage[all]{xy}
\usepackage{cmap}
\usepackage{mathrsfs} 
\usepackage{bm}       

\usepackage[OT2,T1]{fontenc}
\DeclareSymbolFont{cyrletters}{OT2}{wncyr}{m}{n}
\DeclareMathSymbol{\Sha}{\mathalpha}{cyrletters}{"58}

\numberwithin{equation}{section}
\newtheorem{theorem}{Theorem}[section]
\newtheorem{proposition}[theorem]{Proposition}
\newtheorem{lemma}[theorem]{Lemma}
\newtheorem{corollary}[theorem]{Corollary}
\newtheorem{conjecture}[theorem]{Conjecture}

\newtheorem{theoremintro}{Theorem}
\newtheorem{conjectureintro}{Conjecture}

\theoremstyle{definition}
\newtheorem{definition}[theorem]{Definition}

\newtheorem{question}[theorem]{Question}

\theoremstyle{remark}
\newtheorem{remark}[theorem]{Remark}
\newtheorem*{remark*}{Remark}

\newcommand{\PIC}{\mathscr{P}\!\mathit{ic}}
\newcommand{\supp}{\operatorname{Supp}}

\newcommand{\Db}{{\cat{D}^{\mathrm{b}}}}

\newcommand{\Br}{{\rm Br}}

\newcommand{\Pic}{{\rm Pic}}

\newcommand{\rk}{{\rm rk\,}}

\newcommand{\Hom}{{\rm Hom}}
\newcommand{\Spec}{{\rm Spec}\,}

\renewcommand{\dim}{{\rm dim}}
\renewcommand{\supp}{{\rm Supp}}

\newcommand{\Cl}{\mathrm{Cl}}
\newcommand{\bslash}{\smallsetminus}
\newcommand{\linedef}[1]{\emph{#1}}
\newcommand{\sheaf}[1]{\mathscr{#1}}
\newcommand{\CliffAlg}{\sheaf{C}}
\newcommand{\CliffB}{\sheaf{B}}
\newcommand{\pullback}{{}^*}
\newcommand{\pushforward}{{}_*}
\newcommand{\category}[1]{\mathsf{#1}}
\newcommand{\Coh}{\category{Coh}}
\newcommand{\COH}{\category{COH}}
\newcommand{\maxideal}{\mathfrak{m}}
\newcommand{\mult}{{}^{\times}}

\newcommand{\disc}{\text{disc}}
\newcommand{\stack}[1]{\mathcal{#1}}
\newcommand{\gerbe}[1]{\mathcal{#1}}
\newcommand{\tr}{\mathrm{tr}}

\newcommand{\coker}{{\rm coker}}

\newcommand{\id}{{\rm id}}

\newcommand{\dual}{^{\vee}}

\newcommand{\cat}[1]{\mathsf{#1}}

\newcommand{\ra}{\rightarrow}

\newcommand{\OO}{\mathscr{O}}

\newcommand{\PProj}{\mathbf{Proj}\,}
\newcommand{\Proj}{\mathrm{Proj}\,}
\newcommand{\mapto}[1]{\xrightarrow{#1}}

\newcommand{\tensor}{\otimes}
\newcommand{\isom}{\cong}
\newcommand{\wt}[1]{\widetilde{#1}}
\DeclareMathOperator{\EEnd}{\mathscr{E}\!\mathit{nd}}
\DeclareMathOperator{\HHom}{\mathscr{H}\!\mathit{om}}

\DeclareMathOperator{\QQuad}{\mathscr{Q}\!\mathit{uad}}
\DeclareMathOperator{\BBil}{\mathscr{B}\!\mathit{il}}
\newcommand{\op}{^{\mathrm{op}}}
\newcommand{\LG}{\Lambda\!\mathrm{G}}
\newcommand{\ol}[1]{\overline{#1}}
\newcommand{\inv}{^{-1}}

\newcommand{\exterior}{{\textstyle \bigwedge}}

\newcommand{\Group}[1]{\mathbf{#1}}

\newcommand{\Gm}{\Group{G}_{\mathrm{m}}}
\newcommand{\CliffZ}{\kz}

\newcommand{\SSpec}{\mathbf{Spec}\,}

\newcommand{\cali}{\mathcal}
\newcommand{\cal}{\mathscr}
\newcommand{\ka}{{\cal A}}

\newcommand{\kc}{{\cal C}}

\newcommand{\cE}{{\cali E}}
\newcommand{\kf}{{\cal F}}

\newcommand{\ki}{{\cal I}}

\newcommand{\ko}{{\cal O}}
\newcommand{\kp}{{\cal P}}

\newcommand{\cQ}{{\cali Q}}

\newcommand{\kv}{{\cal V}}

\newcommand{\cX}{{\cali X}}

\newcommand{\cW}{{\cali W}}
\newcommand{\kz}{{\cal Z}}

\newcommand{\ZZ}{\mathbb{Z}}
\newcommand{\Z}{\mathbb{Z}}

\newcommand{\CC}{\mathbb{C}}
\newcommand{\FF}{\mathbb{F}}

\newcommand{\PP}{\mathbb{P}}

\newcommand{\quadform}[1]{<\! #1 \!>}
\newcommand{\sheafsharp}{^{\scriptscriptstyle\#}}
\newcommand{\CliffKuz}{\mathfrak{B}}
\newcommand{\CoordKuz}{\mathfrak{A}}

\def\vp{\varphi}

\usepackage[backref=page]{hyperref}
\hypersetup{pdftitle={Fibrations in complete intersections of quadrics}}
\hypersetup{pdfauthor={Asher Auel, Marcello Bernardara, Michele Bolognesi}} 
\hypersetup{colorlinks=true,linkcolor=blue,anchorcolor=blue,citecolor=blue}

\begin{document}

\title[Fibrations in complete intersections of quadrics]{Fibrations in
complete intersections of quadrics, \\ Clifford algebras, derived
categories, \\ and rationality problems}

\author{Asher Auel}
\address{Department of Mathematics\\%
Yale University\\%
10 Hillhouse Avenue\\%
New Haven, CT 06511\\%
USA}
\email{asher.auel@yale.edu}

\author{Marcello Bernardara}
\address{Institut de Math\'ematiques de Toulouse \\ %
Universit\'e Paul Sabatier \\ %
118 route de Narbonne \\ %
31062 Toulouse Cedex 9\\ %
France}
\email{marcello.bernardara@math.univ-toulouse.fr}

\author{Michele Bolognesi}
\address{Institut de Recherche Math\'ematique de Rennes \\ %
Universit\'e de Rennes 1 \\ %
263 Avenue du G\'en\'eral Leclerc, CS 74205 \\ %
35042 Rennes Cedex \\ %
France}
\email{michele.bolognesi@univ-rennes1.fr}




\begin{abstract}
Let $X \to Y$ be a fibration whose fibers are complete intersections
of $r$ quadrics.  We develop new categorical and algebraic tools---a
theory of relative homological projective duality and the Morita
invariance of the even Clifford algebra under quadric reduction by
hyperbolic splitting---to study semiorthogonal decompositions of the
bounded derived category $\Db(X)$.  Together with results in the
theory of quadratic forms, we apply these tools in the case
where $r=2$ and $X \to Y$ has relative dimension 1, 2, or 3, in which
case the fibers are curves of genus one, Del Pezzo surfaces of degree
4, or Fano threefolds, respectively.  In the latter two cases, if $Y =
\PP^1$ over an algebraically closed field of characteristic zero, we
relate rationality questions to categorical representability of $X$.
\end{abstract}

\maketitle

\vspace{-0.7cm}

\setcounter{tocdepth}{1}
\tableofcontents

\setcounter{section}{-1}

\vspace{-1cm}

\setcounter{section}{0}
\section*{Introduction}

One of the numerous applications of the study of triangulated
categories in algebraic geometry is understanding how to extract, from
the bounded derived category of coherent sheaves $\Db(X)$, information
about the birational geometry of a given smooth projective variety
$X$.

Since the seminal work of Bondal--Orlov
\cite{bondal_orlov:semiorthogonal}, it has become understood that such
information should be encoded in semiorthogonal decompositions
$$
\Db(X) = \langle \cat{A}_1, \dotsc, \cat{A}_n \rangle
$$
by admissible triangulated subcategories:\ purely homological
properties of the components of such a decomposition often reflect
geometric properties of $X$.  For example, if for each $i > 1$, the
component $\cat{A}_i$ is ``zero dimensional'' (i.e., equivalent to the
bounded derived category of the base field), then $\cat{A}_1$ should
contain nontrivial information about the birational geometry of $X$.
When $X$ is a Fano threefold, many examples support this idea
\cite{berna_macri_mehro_stella},
\cite{bolognesi_bernardara:conic_bundles},
\cite{bondal_orlov:semiorthogonal}, \cite{kuznetsov:v14},
\cite{kuznetsov:v12}.

In particular, in the case that $X \to S$ is a conic bundle over a
rational complex surface, a semiorthogonal decomposition by derived
categories of points and smooth projective curves allows one to
reconstruct the intermediate jacobian $J(X)$ as the sum of the
jacobians of the curves. This can determine the rationality of $X$
when $S$ is minimal \cite{bolognesi_bernardara:conic_bundles}.  More
generally, this works if $X$ is a complex threefold with negative
Kodaira dimension (e.g., a Fano threefold) whose codimension 2 cycles
are universally described by a principally polarized abelian variety
\cite[\S3.2]{bolognesi_bernardara:representability}.  In such
cases, homological properties of semiorthogonal decompositions are
related to classical notions of representability of cycles on $X$.

Attempting to trace the link between derived categories and algebraic
cycles, the second and third named authors defined in
\cite{bolognesi_bernardara:representability} the notion of
\emph{categorical representability} in a given dimension $m$ (or
codimension $\dim(X)-m$) of a smooth projective variety $X$, by
requiring the existence of a semiorthogonal decomposition whose
components can be fully and faithfully embedded in derived categories
of smooth projective varieties of dimension at most $m$.
Categorical representability in dimension one is equivalent to the
existence of a semiorthogonal decomposition by copies of the derived category of a point and
derived categories of smooth projective curves.  One might wonder if
categorical representability in codimension 2 is a necessary condition
for rationality.  The work of Kuznetsov on cubic fourfolds
\cite{kuznetsov:cubic_fourfolds}, developed before the definition of
categorical representability, shows how this philosophy persists as a
tool to conjecturally understand rationality problems in dimension
larger than three, where one cannot appeal to more classical methods,
such as the study of the intermediate jacobian.

In this paper, we provide two new instances where categorical
representability is strictly related to birational properties.  These
arise as fibrations $X \to \PP^1$ whose fibers are complete
intersections of two quadrics.  We impose a genericity hypothesis on
such fibrations (see Definition~\ref{def:generic-fibration}) so that
the associated pencil of quadrics has simple degeneration along a
smooth divisor.

\smallskip

In \S\ref{section:delpezzo}, we consider fibrations $X \to \PP^1$
whose fibers are Del Pezzo surfaces of degree four.  Such threefolds
have negative Kodaira dimension and their rationality (over the
complex numbers) is completely classified \cite{alekseev:dP4},
\cite{shramov:dp4}.  We provide a purely categorical criterion for
rationality of $X$ based on \cite{bolognesi_bernardara:conic_bundles}.

\begin{theoremintro}[\S\ref{section:delpezzo}]
\label{thm:1}
Let $X \to \PP^1$ be a generic Del Pezzo fibration of degree
four over the complex numbers.  Then $X$ is rational if and only if it
is categorically representable in codimension 2.  Moreover, there is a
semiorthogonal decomposition
$$
\Db(X)=\langle \Db(\Gamma_1), \dotsc, \Db(\Gamma_k), E_1, \dotsc , E_l
\rangle,
$$
with $\Gamma_i$ smooth projective curves and $E_i$ exceptional objects
if and only if $J(X) = \oplus J(\Gamma_i)$ as principally polarized
abelian varieties.
\end{theoremintro}

In \S\ref{section:4fold}, we consider fibrations $X \to \PP^1$ whose
fibers are complete intersections of two four-dimensional
quadrics. Such fourfolds have a semiorthogonal decomposition
$$
\Db(X) = \langle \cat{A}_X, E_1, \dotsc, E_4 \rangle,
$$
where $E_i$ are exceptional objects.  Moreover, we construct a
fibration $T \to \PP^1$ in hyperelliptic curves and a Brauer class
$\beta \in \Br(T)$ such that $\cat{A}_X \simeq \Db(T,\beta)$.  We
state a conjecture in the same spirit as Kuznetsov's conjecture for
cubic fourfolds \cite[Conj.~1.1]{kuznetsov:cubic_fourfolds}.

\begin{conjectureintro}
\label{conj:1}
Let $X \to \PP^1$ be a fibration whose fibers are
intersections of two four-dimensional quadrics over the complex
numbers.  Then $X$ is rational if and only if it is categorically
representable in codimension 2.
\end{conjectureintro}

The main evidence for this conjecture is provided in two cases of rational fibrations that are categorically representable in codimension 2.  Recall from
\cite[Thm.~2.2]{colliot-thelene_sansuc_swinnerton-dyer:quadrics_I}
that if $X$ contains a surface generically ruled over $\PP^1$, then
$X$ is rational.

\begin{theoremintro}[\S\ref{section:4fold}]
\label{thm:2}
Let $X \to \PP^1$ be a generic fibration whose fibers are
intersections of two four-dimensional quadrics over a field $k$.  Let
$T \to \PP^1$ and $\beta \in \Br(T)$ be the associated fibration in
hyperelliptic curves and Brauer class.  If $\beta=0$, then $X$ is
rational and categorically representable in codimension 2. In
particular, this is the case if $X$ contains a surface generically
ruled over $\PP^1$.
\end{theoremintro}

These results are obtained from new general constructions involving
the derived category of a smooth projective variety $X$ admitting a
fibration $X \to Y$ in complete intersection of quadrics.  Over an
algebraically closed field, the derived category of an intersection of
two quadrics was studied by
Kapranov~\cite{kapranov:intersection_quadrics} and
Bondal--Orlov~\cite[\S2]{bondal_orlov:semiorthogonal}.  Along with
results in the algebraic theory of quadratic forms, we utilize three
main tools:\ homological projective duality \cite{kuznetsov:hpd}
extended to a relative context for quadric fibrations (see
\S\ref{subsec:relative-hpd}), Kuznetsov's~\cite{kuznetsov:quadrics}
description of the derived category of a quadric fibration via the
even Clifford algebra, and a Morita invariance result for the even
Clifford algebra under ``quadric reduction'' by hyperbolic splitting
(see \S\ref{subsec:hyperbolic_algebraic}).  We now give an overview of
these tools.

\medskip

Let $Q \to S$ be a flat quadric fibration of relative dimension $n-2$
over a scheme with associated even Clifford algebra $\kc_0$.
Kuznetsov \cite{kuznetsov:quadrics}, extending the seminal work of
Kapranov \cite{kapranov:quadric}, \cite[\S4]{kapranov:derived} on
derived categories of quadrics, exhibits a semiorthogonal
decomposition
\begin{equation}\label{decomposition-of-quadric}
\Db(Q)= \langle \Db(S,\kc_0), \Db(S)_1, \dotsc, \Db(S)_{n-2} \rangle,
\end{equation}
where $\Db(S)_i \simeq \Db(S)$ for all $1 \leq i \leq n-2$. The category
$\Db(S,\kc_0)$ is the component of $\Db(Q)$ encoding nontrivial
information about the quadric.  One of the main results of this paper
is the derived invariance of this category under the geometric process
of ``quadric reduction'' by hyperbolic splitting, see
\S\ref{subsec:hyperbolic_algebraic}.  This allows us to pass to
smaller dimensional quadric fibrations.

``Quadric reduction'' by hyperbolic splitting is a classical
construction which, starting from a quadric $Q \subset \PP^{n+1}$ with
a smooth rational point $x$, describes a quadric $Q' \subset
\PP^{n-1}$ with the same degeneration type. Roughly, $Q'$ is the base
of the cone obtained by intersecting $T_x Q$ and $Q$.  This is the
analogue, in the language of quadratic forms, of ``splitting off a
hyperbolic plane.''
This construction can be performed relatively over a base scheme $S$,
given quadric fibration $Q \to S$ with possibly singular fibers and a
smooth section (i.e., a section $S \to Q$ avoiding singular points of
fibers), see \S\ref{subsec:hyperbolic_geometric}.

\begin{theoremintro}[Corollary~\ref{morita}]
\label{thm:3}
Let $S$ be a regular integral scheme, $Q \to S$ a quadric fibration
with simple degeneration along a regular divisor, $Q' \to S$ the
quadric fibration obtained by quadric reduction along a smooth
section, and $\kc_0$ and $\kc_0'$ the respective even Clifford
algebras. If the relative dimension is odd, assume that 2 is
invertible on $S$.  Then there is an equivalence $\Db(S,\kc_0) \simeq
\Db(S,\kc_0')$.
\end{theoremintro}

This result is classical if $S$ is the spectrum of a field and the
quadric $Q$ is smooth, see \cite[Lemma~14.2]{elman_karpenko_merkurjev}.
Along the way to proving Theorem~\ref{thm:3}, we provide new
generalizations, to a degenerate setting, of fundamental results in
the theory of quadratic forms (see
\S\ref{subsec:hyperbolic_algebraic}). These include an isotropic
splitting principle (Theorem~\ref{thm:isotropic_splitting_principle})
and an orthogonal sum formula for the even Clifford algebra (Lemma
\ref{lem:even_clifford_perp}).  In general, algebraic results of this
type are quite limited in the literature (see e.g.,
\cite{bayer_fainsilber} and
\cite[IV.4.8]{knus:quadratic_hermitian_forms}).  Some stack theoretic
considerations are required in the case of odd rank.  In
Appendices~\ref{appendix_clifford} and \ref{subsec:grassmannian}, we
prove the equality of differing constructions and interpretations of
the even Clifford algebra in the literature.

Such results on degenerate forms may prove useful in their own right.
There has been recent focus on such forms from various number
theoretic perspectives.  An approach to Bhargava's~\cite{bhargava:ICM}
construction of moduli spaces of ``rings of low rank'' over arbitrary
base schemes is developed by Wood~\cite{wood:binary},
who must deal with degenerate forms (of higher degree).  In related
developments, building on the work of
Delone--Faddeev~\cite{delone_faddeev} over $\Z$ and Gross--Lucianovic
\cite{gross_lucianovic} over local rings, V.\
Balaji~\cite{balaji_ternary}, and independently
Voight~\cite{voight:quaternion_rings}, used Clifford algebras of
degenerate quadratic forms of rank 3 to classify degenerations of
quaternion algebras over arbitrary bases.

\medskip

Homological projective duality was introduced by Kuznetsov
\cite{kuznetsov:hpd} to study semiorthogonal decompositions of
hyperplane sections (see also \cite{kuznetsov:hyp-sections}).  For
example, consider a finite set of quadric hypersurfaces
$\{Q_i\}_{i=0}^r$ in a projective space $\PP^{n-1} = \PP(V)$ and their
complete intersection $X$.  Let $Q$ be the linear system of quadrics
generated by the $Q_i$, which is a quadric fibration $Q \to \PP^{r}$
of relative dimension $n-2$, and let $\kc_0$ be its associated even
Clifford algebra. The derived categories $\Db(X)$ and
$\Db(\PP^r,\kc_0)$ are strongly related \cite{kuznetsov:quadrics}.  In
particular, if $X$ is Fano or Calabi--Yau, then there is a fully
faithful functor $\Db(\PP^r,\kc_0) \to \Db(X)$.  In
\S\ref{section:tools}, we describe a relative version of this
construction, replacing $V$ by a vector bundle $E$ over a smooth
variety $Y$, the quadrics by flat quadric fibrations $Q_i \to Y$
contained in $\PP(E)$, the intersection by the relative complete
intersection $X \to Y$ of the quadric fibrations, the linear system of
quadrics by the \linedef{linear span quadric fibration} $Q \to S$
(which is a flat quadric fibration over a projective bundle $S \to Y$,
see Definition~\ref{linearspan}), and $\kc_0$ by the even Clifford
algebra of $Q \to S$.  Then $\Db(X)$ and $\Db(S,\kc_0)$ are similarly
related:\ if the generic fiber of $X \to Y$ is Fano or Calabi--Yau,
then there is a fully faithful functor $\Db(S,\kc_0) \to \Db(X)$.

Considering the relationship between intersections of quadrics and
linear spans also has arithmetic roots.  The Amer--Brumer theorem
(Theorem~\ref{amer-brumer}), which is indispensable for our work,
states that the intersection of two quadrics has a rational point if
and only if the linear system of quadrics (or pencil of quadrics) has
a rational section over $\PP^1$.  Versions of this theorem also hold
for 0-cycles of degree 1 on intersections of more quadratic forms
\cite{colliot-thelene_levine}.  The general subject is also considered
in \cite{colliot-thelene_sansuc_swinnerton-dyer:quadrics_I}.

Using the above described categorical and algebraic tools, we proceed
in the relative setting described above.  Assuming $X$ is a smooth
projective variety and the generic fiber of $X \to Y$ is Fano (resp.\
Calabi--Yau)---a condition that is satisfied when $2r+2<n$ (resp.\
$2r+2=n$)---we obtain a semiorthogonal decomposition
$$
\Db(X) = \langle \Db(S,\kc_0), \Db(Y)_{1}, \dotsc, \Db(Y)_{n-2r}
\rangle,
$$
where $\Db(Y)_i \simeq \Db(Y)$ for $1 \leq i \leq n-2r$ (resp.\ an
equivalence $\Db(S,\kc_0) \simeq \Db(X)$).  The category
$\Db(S,\kc_0)$ is then the important component $\cat{A}_X$ of $\Db(X)$
remarked on earlier.  Moreover, if $Q \to S$ admits a smooth section,
we can more closely examine the relationship between $\Db(X)$ and
$\Db(Q)$ by performing quadric reduction.  In the cases we consider,
such sections are guaranteed by the work of Alexeev
\cite{alekseev:dP4} and by the theory of quadratic forms (see
Lemma~\ref{lem:C_2}), or by the general results of
\cite{campana_peternell_pukhlikov} and \cite{graber_harris_starr} on
fibrations over curves with rationally connected fibers.

\smallskip

This article is organized as follows. In \S\ref{sec:Quadratic_forms},
we adapt many standard results in the theory of regular quadratic
forms and even Clifford algebras to possibly degenerate forms with
values in line bundles, eventually leading to a proof of the Morita
invariance of the even Clifford algebra under quadric reduction by
hyperbolic splitting.  In \S\ref{section:tools}, we introduce the
basic notions and relevant results in derived category theory and then
detail the relative version of homological projective duality for
quadric fibrations. Finally, we study the explicit case of varieties
$X \to Y$ which are intersection of two quadric fibrations:\ in
\S\ref{elliptic}, we consider the case of genus 1 fibrations, where
our results are analogous to (but different form) well-known ones; in
\S\ref{section:delpezzo}, we consider the case of quartic Del Pezzo
fibrations over curves; and in \S\ref{section:4fold}, we consider
fourfolds $X \to \PP^1$ fibered in complete intersections of two
quadrics.  The last two applications require working over an
algebraically closed field (assumed to be of characteristic zero for
certain statements).  In Appendix~\ref{appendix_clifford}, we compare
various constructions of the even Clifford algebra in the literature
and recall some of its functorial properties.  In
Appendix~\ref{subsec:grassmannian}, we provide a proof that for
quadric fibrations of relative dimension 2 and 4, the Brauer class of
the even Clifford algebra over the discriminant cover coincides with
that arising from the Stein factorization of the lagrangian
grassmannian considered in \cite{hassett_varilly:K3} and
\cite{kuznetsov:cubic_fourfolds}.

\subsection*{Notations}
In general, we work over an arbitrary base field $k$.
As we indicate in \S\ref{basechange}, the basic results in derived
categories and semiorthogonal decompositions that we appeal to hold
over any field.  
Any triangulated category is assumed to be essentially small and
support a $k$-linear structure.  Whenever a functor between derived
categories is given, it will be denoted as if underived. For example,
if $f: X \to Y$ is a morphism, then denote by $f^*$ and $f_*$ the
derived pull-back and push-forward, respectively.  
By a scheme, we mean a noetherian separated
scheme.  By a variety, we mean a reduced scheme of finite type over a
field.
 If $X$ is a scheme,
$\ko_X(1)$ a line bundle, and $\cat{A} \subset \Db(X)$ a subcategory,
we denote by $\cat{A}(i)$ the image of $\cat{A}$ under the
autoequivalence $\otimes \ko_X(i)$ of $\Db(X)$. In the relative
setting $X \to Y$, we use the same notation for a given line bundle
$\ko_{X/Y}(1)$ on $X$.  
If $\ka$ is a $\OO_X$-algebra, which is coherent as an $\OO_S$-module
and right (or left) noetherian, then we denote by $\Coh(X,\ka)$ the
abelian category of finitely presented right (or left) $\ka$-modules
and by $\Db(X,\ka)$ its derived category. 
By a vector bundle on $X$ we mean a locally free $\ko_X$-module of
finite constant rank.  By a projective bundle on $Y$ we mean $\PP(E) =
\PProj S^{\bullet}(E\dual) \to Y$ for a vector bundle $E$ on $X$.

\medskip

{\small\noindent{\bf Acknowledgments.}  %
Much of this work has been developed during visits (both long and
short) of the authors at the Max Planck Institut f\"ur Mathematik in
Bonn, Universit\"at Duisburg--Essen, Universit\'e Rennes 1, ETH
Z\"urich, Emory University, and the Courant Institute of Mathematical
Sciences at New York University.  The hospitality of each institute is warmly acknowledged.
The first author is partially supported by a National Science Foundation
grant MSPRF DMS-0903039 and by an NSA Young Investigator Grant. The
second author was partially supported by the SFB/TR 45 `Periods,
moduli spaces, and arithmetic of algebraic varieties'.  The authors
would specifically like to thank A.\ Chambert-Loir, J.-L.\
Colliot-Th\'el\`ene, B.\ Hassett, M.-A.\ Knus, E.\ Macr\`\i, L.\
Moret-Bailly, R.\ Parimala, P.\ Stellari, V.\ Suresh and A.\ Vistoli for many
helpful discussions.  Extensive and incisive comments from an
anonymous referee, who is heartily thanked, helped to dramatically
improve an earlier version of the manuscript.}

\section{Quadratic forms and even Clifford algebras}
\label{sec:Quadratic_forms}

This section is mostly devoted to generalizing some well-known results
in the theory of regular quadratic forms (or ``quadratic spaces'') to
quadratic forms with degeneration over schemes.  While our treatment
proceeds in the ``algebraic theory of quadratic forms'' style, we
nevertheless keep the geometric perspective of quadric fibrations in
mind.

\subsection{Quadratic forms}
\label{subsec:Quadratic_forms_and_quadric_fibrations}

Let $Y$ be a scheme.  Fix vector bundles $E$ and $L$ on $Y$.  Denote
by $T^2 E$, $S^2 E$, and $S_2 E$ the tensor square, symmetric square,
and submodule of symmetric tensor squares of $E$, respectively.  We
first collect together various notions of quadratic form with values
in a $L$, and show that they are all equivalent.

\begin{lemma} 
\label{lem:quad}
The following sets of objects are in natural bijection.
\begin{enumerate}
\item \label{lem:quad.1} Morphisms of sheaves $q : E \to L$
satisfying:\ $q(a\, v)=a^2\, q(v)$ on sections $a$ of $\OO_Y$ and $v$
of $E$; and that the morphism of sheaves $b_q : E \times E \to L$, defined
by $b_q(v,w) = q(v+w) - q(v) - q(w)$ on sections $v$ and $w$ of $E$,
is $\OO_Y$-bilinear.

\item \label{lem:quad.2} Morphisms of $\OO_Y$-modules $S_2 E \to L$, i.e., global
sections $\Gamma(Y,\HHom(S_2 E, L))$.

\item \label{lem:quad.3} Morphisms of $\OO_Y$-modules $L\dual \to
S^2(E\dual)$, i.e., global sections $\Gamma(Y,\HHom(L\dual,
S^2(E\dual))$.

\item \label{lem:quad.4} Global sections $\Gamma(Y,S^2(E\dual)\tensor L)$.

\item \label{lem:quad.5} Global sections $s_q \in \Gamma(\PP(E),\OO_{\PP(E)/Y}(2)\tensor
p\pullback L)$, where $p : \PP(E) \to Y$ is the projection.
\end{enumerate}
\end{lemma}
\begin{proof}
Let $\QQuad(E,L)$ be the Zariski presheaf (which is actually a sheaf,
and moreover, naturally an $\OO_Y$-module), whose sections over $U$
are the morphisms of sheaves $q : E|_U \to L|_U$ on $U$, satisfying
the conditions in \eqref{lem:quad.1}.  Let $\BBil(E,L)$ be the Zariski
presheaf (which is an $\OO_Y$-module), whose sections over
$U$ are the $\OO_U$-bilinear morphisms $b : E|_U \times E|_U \to
L|_U$.  Then there is a commutative diagram of morphisms of
$\OO_Y$-modules:
\begin{equation}
\label{eq:quad}
\begin{split}
\xymatrix@R=10pt{
T^2(E\dual)\tensor L \ar[dd]_{\text{\rotatebox{90}{$\sim$}}} \ar@{->>}[rr] \ar@{->>}[rd]&& \HHom(S_2 E,L)\ar@<2mm>[dd]_{\text{\rotatebox{90}{$\sim$}}}\\
& S^2(E\dual)\tensor L \ar[ru]_{\text{\rotatebox{15}{$\sim$}}} \ar[rd]^{\text{\rotatebox{-15}{$\sim$}}} & \\
\BBil(E,L) \ar@{->>}[rr] & & \QQuad(E,L) \ar@<2mm>[uu]
}
\end{split}
\end{equation}
The left vertical map is the canonical isomorphism $f \otimes g
\otimes l \mapsto \bigl( (v,w) \mapsto f(v)g(w)l \bigr)$.  The bottom
horizontal map is $b \mapsto \bigl( v \mapsto b(v,v) \bigr)$.  The
composition of these maps clearly factors through $S^2(E\dual) \tensor
L$, hence the diagonal maps from upper left to lower right.  The
vertical downward map at right is $\vp \mapsto \bigl(v \mapsto
\vp(v\tensor v)\bigr)$, while the upward map is $q \mapsto \bigl(
v\tensor v \mapsto q(v)\bigr)$.  The top horizontal map is $f \tensor
g \tensor l \mapsto \bigl(v \tensor v \mapsto f(v)g(v)l \bigr)$, this
clearly factor through $S^2(E\dual)\tensor L$, hence the diagonal map
from the center to upper right.  (In particular, this gives a
canonical isomorphism $S^2(E\dual) \to S_2(E)\dual$.)
In defining maps to $\HHom(S_2 E,L)$, we observe that 
$S_2 E$ is the sheafification of the $\OO_Y$-submodule of
$E \tensor E$ generated by all sections of the form $v\tensor v$, thus
by the universal property of sheafification, to define an
$\OO_Y$-module morphism $S_2 E \to L$ it suffices to specify it on
sections of the form $v\tensor v$.

The fact that the triangle at right consists of isomorphisms is
checked locally, see \cite[Lemma~2.1]{swan:quadric_hypersurfaces} or
\cite[Prop.~6.1]{wood:binary}.  Taking global sections in this
triangle establishes the bijection between \eqref{lem:quad.1},
\eqref{lem:quad.2}, and \eqref{lem:quad.4}.  The evident
$\OO_Y$-isomorphisms $S^2(E\dual) \tensor L \to (L\dual)\dual \tensor
S^2(E\dual) \to \HHom(L\dual, S^2(E\dual))$ prove the bijection between
\eqref{lem:quad.4} and \eqref{lem:quad.3}.  Finally, the isomorphism
$\Gamma(Y,S^2(E\dual)\tensor L) \isom \Gamma(\PP(E),\OO(2)\tensor
p\pullback L)$ establishes the bijection between \eqref{lem:quad.4} and \eqref{lem:quad.5}.
\end{proof}

By a \linedef{(line bundle-valued) quadratic form} on $Y$, we mean a
triple $(E,q,L)$ as in Lemma~\ref{lem:quad}(1), with $L$ a line bundle
on $Y$.  We will mostly dispense with the title ``line
bundle-valued.''  
The \linedef{rank} of $(E,q,L)$ is
the rank of $E$.

A quadratic form $(E,q,L)$ is \linedef{primitive} if the associated
$\OO_Y$-module morphism $S_2E \to L$ is an epimorphism.  This is
equivalent to $q$ being nonzero over the residue field of any point of
$Y$ or to the associated morphism $L\dual \to S^2(E\dual)$ being
a monomorphism with locally free cokernel.

There is an evident \linedef{polar} $\OO_Y$-module morphism $\psi_{q} : E \to
\HHom(E,L)$ associated to $b_{q}$. A quadratic form $(E,q,L)$ is
\linedef{regular} if $\psi_{q}$ is an isomorphism.  A quadratic form
of odd rank can never be regular over a characteristic 2 point, and
the notion of \linedef{semiregularity} is needed, see
\cite[IV.3]{knus:quadratic_hermitian_forms}.  We say that a quadratic
form is \linedef{generically (semi)regular} if it is (semi)regular at
every maximal generic point of $Y$ (this is called
\linedef{non-degenerate} in \cite{knus:quadratic_hermitian_forms}).

If $(E,q,L)$ has rank $n$, then the signed determinant
$(-1)^{n(n-1)/2}\det \psi_q : \det E \to \det E\dual \tensor
L^{\tensor n}$ gives rise to a global section $\disc(q) \in
\Gamma(Y,(\det E\dual)^{\tensor 2} \tensor L^{\tensor n})$ called the
\linedef{discriminant} (in odd rank, one must take the
\linedef{half-discriminant}, cf.\
\cite[IV.3.1]{knus:quadratic_hermitian_forms}).  The zero scheme $D
\subset Y$ of the (half-)discriminant is the \linedef{discriminant
divisor}.  Its reduced subscheme is the locus of points of $Y$ where
$q$ is not (semi)regular. If $q$ is generically (semi)regular, then
the discriminant divisor is the subscheme associated to an effective
Cartier divisor.

Assume that $(E,q,L)$ is generically (semi)regular, let $j : D \to
S$ denote the discriminant divisor and $i : U \to S$ its open
complement.
The restriction of the polar morphism $\psi_q{}|_U : E|_U \to
\HHom(E,L)|_U$ is then an $\OO_U$-module isomorphism and we have an exact
sequence of $\OO_D$-modules
$$
0 \to R \to j\pullback E \mapto{j\pullback\psi_q}
j\pullback\HHom(E,L) \to C \to 1 
$$
where we define $R$ and $C$ to be the sheaf kernel and cokernel of
$j\pullback\psi_q$, respectively.  We call $R$ the \linedef{bilinear
radical} of $j\pullback q$.  The \linedef{quadratic radical} is the
$\OO_D$-submodule $R_q \subset R$ where $q|_D$ vanishes.  The bilinear
and quadratic radical coincide over any point where 2 is invertible.
Note that as $E$ is locally free and $q$ is generically (semi)regular,
$\psi_q$ is injective and $R$ and $R_q$ are torsion-free
$\OO_D$-modules.

A quadratic form $(E,q,L)$ of rank $n$ has \linedef{(at most) simple
degeneration} on $Y$ if it has a (semi)regular quadratic subform of
rank at least $n-1$ over the local ring of every point of $Y$.
Equivalently, the quadratic radical $R_q$ has rank $\leq 1$.  For an
(nonempty) effective Cartier divisor $D \subset S$, we say that $q$
has simple degeneration \linedef{along} $D$ if $(E,q,L)$ has simple
degeneration and is generically (semi)regular with discriminant
divisor $D$.

A \linedef{similarity} between quadratic forms $(E,q,L)$ and
$(E',q',L')$ is a pair $(\vp,\lambda)$ consisting of $\OO_Y$-module
isomorphisms $\vp : E \to E'$ and $\lambda : L \to L'$ such that
either of the following equivalent diagrams,
\begin{equation}
\label{similarity}
\begin{split}
\xymatrix{
E \ar[d]_{\vp} \ar[r]^{q}  & L \ar[d]^{\lambda} \\
E' \ar[r]^{q'} & L' \\
}
\qquad
\xymatrix{
S_2E \ar[d]_{S_2\vp} \ar[r]  & L \ar[d]^{\lambda} \\
S_2E' \ar[r] & L' \\
}
\end{split}
\end{equation}
commute, which happens if and only if we have $q'(\vp(v)) = \lambda
(q(v))$ on sections.

Given quadratic forms $(E_1,q_1,L)$ and $(E_2,q_2,L)$ with values in
the same line bundle, their \linedef{orthogonal sum} $(E_1,q_1,L)
\perp (E_2,q_2,L) = (E_1 \oplus E_2, q_1 \perp q_2,L)$ is defined by
$(q_1 \perp q_2)(v) = q_1(v)+q_2(v)$.

\subsection{Quadric fibrations}
\label{subsec:quadric_fibrations}

The \linedef{quadric fibration} $\pi : Q \to Y$ associated to a
nonzero quadratic form $(E,q,L)$ of rank $n \geq 2$ is the restriction
of the projection $p : \PP(E) \to Y$ via the closed embedding $j : Q
\to \PP(E)$ defined by the vanishing of the global section $s_q \in
\Gamma(\PP(E),\OO_{\PP(E)/Y}(2)\tensor p\pullback L)$.  Write
$\OO_{Q/Y}(1) = j\pullback \OO_{\PP(E)/Y}(1)$.  
The form $(E,q,L)$ is primitive if and only if $\pi : Q \to Y$ is flat
of relative dimension $n-2$, see
\cite[8~Thm.~22.5]{matsumura:commutative_ring_theory}.
The fiber $Q_y$ is a smooth projective quadric (resp.\ a quadric cone
with isolated singularity) over any point $y$ where $(E,q,L)$ is
(semi)regular (resp.\ has simple degeneration).

Define the \linedef{projective similarity} class of a quadratic form
$(E,q,L)$ to be the set of similarity classes of quadratic forms
$(N,q_N,N^{\tensor 2}) \tensor (E,q,L) = (N\tensor E, q_N \tensor q,
N^{\tensor 2}\tensor L)$ ranging over all line bundles
$N$, where $q_N : N \to N^{\tensor 2}$ is the squaring form.  In
\cite{balmer_calmes:lax}, this is referred to as a
\linedef{lax-similarity} class.  Though the following should be
well-known, we could not find a proof in the literature.

\begin{proposition}
\label{prop:proj_sim_quadric}
Let $Y$ be an integral locally factorial scheme.  Let $\pi : Q \to Y$
and $\pi' : Q' \to Y$ be quadric fibrations associated to primitive
generically (semi)regular quadratic forms $(E,q,L)$ and $(E',q',L')$.
Then $(E,q,L)$ and $(E',q',L')$ are in the same projective similarity
class if and only if $Q$ and $Q'$ are $Y$-isomorphic.
\end{proposition}
\begin{proof}
Let $\eta$ be the generic point of $Y$ and $\pi : Q \to Y$ a flat
quadric bundle of relative dimension $\geq 1$ (the case of relative
dimension 0 is easy).  Restriction to the generic fiber of $\pi$ gives
rise to a complex
\begin{equation}
\label{eq:minimal}
0 \to \Pic(Y) \mapto{\pi\pullback} \Pic(Q) \to \Pic(Q_{\eta}) \to 0.
\end{equation}
We claim that since $Y$ is locally factorial, \eqref{eq:minimal} is
exact in the middle.  First, note that flat pullback and restriction
to the generic fiber give rise to an exact sequence of Weil divisor
groups
\begin{equation}
\label{eq:div}
0 \to Z^1(Y) \mapto{\pi\pullback} Z^1(Q) \to Z^1(Q_{\eta}) \to 0.
\end{equation}
Indeed, since $Z^1(Q_\eta) = \varinjlim Z^1(Q_U)$ (where the limit
is taken over all dense open sets $U \subset Y$ and we write $Q_U = Q
\times_Y U$) the exactness at right and center of \eqref{eq:div} then
follows from the exactness of the excision sequence
$$
Z^0(\pi\inv(Y\bslash U)) \to Z^1(Q) \to Z^1(Q_U) \to 0
$$
associated to the closed subscheme $\pi\inv(Y\bslash U) \subset Q$.
Note that the exactness at right of the excision sequence follows by
taking the closure of a prime divisor.
The exactness at left of \eqref{eq:div} follows since $\pi$ is
surjective on codimension 1 points.

Since $\pi$ is dominant and the function fields of $Q$ and $Q_\eta$
coincide, the sequence \eqref{eq:div} of Weil divisor groups induces a
sequence
$$
\Cl(Y) \mapto{\pi\pullback} \Cl(Q) \to \Cl(Q_{\eta}) \to 0
$$
of Weil divisor class groups, which is exact by a diagram chase.
We then have the following commutative diagram
$$
\xymatrix@R=14pt{
\Pic(Y) \ar@{^{(}->}[d] \ar[r]^{\pi\pullback} & \Pic(Q) \ar@{^{(}->}[d] \ar[r]
&\Pic(Q_{\eta}) \ar@{^{(}->}[d] \ar[r] & 0\\
\Cl(Y) \ar[r]^{\pi\pullback} & \Cl(Q) \ar[r] & \Cl(Q_{\eta}) \ar[r] & 0
}
$$
of abelian groups.  The vertical inclusions are equalities since $Y$
is locally factorial, cf.\ \cite[Tome~4~Cor.~21.6.10]{EGA4}.  Finally, a
diagram chase shows that \eqref{eq:minimal} is exact in the middle.

Let $(E,q,L)$ and $(E',q',L')$ be projectively similar with respect to
a line bundle $N$ and $\OO_Y$-module isomorphisms $\vp : E' \to N
\tensor E$ and $\lambda : L' \to N^{\tensor 2} \tensor L$ preserving
the quadratic forms.  Write $h = \PProj (\vp\dual) : \PP(E') \to
\PP(N\tensor E)$ for the associated $Y$-isomorphism.  There is a
natural $Y$-isomorphism $g : \PP(N \tensor E) \to \PP(E)$ satisfying
$g\pullback \OO_{\PP(E)/Y}(1) \isom \OO_{\PP(E\tensor N)/Y}(1)\tensor
g\pullback p\pullback N$, see
\cite[II~Lemma~7.9]{hartshorne:algebraic_geometry}.  Denote by $f = g
\circ h : \PP(E') \to \PP(E)$ the composition.  Then via the
isomorphism
$$
\Gamma\bigl( \PP(E'),f\pullback(
\OO_{\PP(E)/Y}(2) \tensor p\pullback L) \bigr) \to
\Gamma\bigl( \PP(E'),\OO_{\PP(E')/Y}(2)\tensor p'\pullback L' \bigr)
$$
induced by $f\pullback \OO_{\PP(E)/Y}(2) \isom \OO_{\PP(E')}(2)\tensor
(p'{}\pullback N)^{\tensor 2}$ and $p'{}\pullback \lambda\inv :
(p'{}\pullback N)^{\tensor 2}\tensor p'{}\pullback L \to p'{}\pullback
L'$, the global section $f\pullback s_q$ is taken to $s_{q'}$, hence
$f$ restricts to a $Y$-isomorphism $Q' \to Q$.

Conversely, let $f : Q' \to Q$ be a $Y$-isomorphism.  First, we will
prove that $f$ can be extended to a $Y$-isomorphism $\tilde{f} :
\PP(E') \to \PP(E)$ satisfying $\tilde{f} \circ j' = j \circ f$.  To
this end, considering the long exact sequence associated to applying
$p\pushforward$ to the short exact sequence
\begin{equation}
\label{eq:sequnce-quadric-fibration}
0 \to \OO_{\PP(E)/Y}(-1) \tensor p\pullback L\dual \mapto{s_q}
\OO_{\PP(E)/Y}(1) \to j\pushforward \OO_{Q/Y}(1) \to 0. 
\end{equation}
and keeping in mind that $R^ip\pushforward\OO_{\PP(E)/Y}(-1)=0$ for
$i=0,1$ (as $p : \PP(E) \to Y$ is a projective bundle of positive
relative dimension by assumption), we arrive at a canonical
isomorphism $p\pushforward \OO_{\PP(E)/Y}(1) \isom
\pi\pushforward\OO_{Q/Y}(1)$.  In particular, we have a canonical
identification $E\dual = \pi\pushforward\OO_{Q/Y}(1)$.  We have a
similar identification $E'{}\dual = \pi\pushforward'\OO_{Q'/Y}(1)$.

We claim that $f\pullback \OO_{Q/Y}(1) \isom \OO_{Q'/Y}(1)
\tensor \pi'{}\pullback N$ for some line bundle $N$ on $Y$.  Indeed,
over the generic fiber, we have $f\pullback \OO_{Q/Y}(1)_{\eta} =
f\pullback_\eta \OO_{Q_\eta}(1) \isom \OO_{Q'_\eta}(1)$ by the case of
smooth quadrics (as $q$ is generically (semi)regular) over a field,
cf.\ \cite[Lemma~69.2]{elman_karpenko_merkurjev}.  Then the exactness
of \eqref{eq:minimal} in the middle finishes the proof of the claim.

Finally, by the projection formula and our assumption that $\pi' : Q'
\to Y$ is of positive relative dimension, we have that $f$ induces an
$\OO_Y$-module isomorphism
$$
E{}\dual\tensor N\dual 
\isom
\pi\pushforward \OO_{Q/Y}(1) \tensor \pi\pushforward' \pi'{}\pullback
N\dual 
\isom 
\pi\pushforward f\pushforward (f\pullback\OO_{Q/Y}(1)\tensor\pi'{}\pullback N\dual) 
\isom
\pi\pushforward' \OO_{Q'/Y}(1) 
=
E'{}\dual
$$
with induced dual isomorphism $\vp : E' \to N \tensor E$.  Now define
$\tilde{f} : \PP(E') \to \PP(E)$ to be the composition of
$\PProj(\vp\dual) : \PP(E') \to \PP(N \tensor E)$ with the natural
$Y$-isomorphism $\PP(N \tensor E) \to \PP(E)$, as earlier in this
proof.  Then by the construction of $\tilde{f}$, we have that
$\tilde{f}\pullback\OO_{\PP(E)/Y}(1) \isom \OO_{\PP(E')/Y}(1)\tensor
p'{}\pullback N$ and that $j \circ f = \tilde{f} \circ j'$ (an
equality that can be checked on fibers using
\cite[Thm.~69.3]{elman_karpenko_merkurjev}).
Equivalently, there exists an isomorphism $\tilde{f}\pullback(\OO_{\PP(E)(2)/Y} \tensor p\pullback L) \isom
\OO_{\PP(E')/Y}(2) \tensor p'{}\pullback L'$ taking $f\pullback s_q$
to $s_{q'}$.  However, as $\tilde{f}\pullback (\OO_{\PP(E)(2)/Y}
\tensor p\pullback L) \isom \OO_{\PP(E')/Y}(2)\tensor
p'{}\pullback(N^{\tensor 2} \tensor L')$, we have an isomorphism
$p'{}\pullback L' \isom p'{}\pullback(N^{\tensor 2}\tensor L)$.  Upon
taking pushforward, we arrive at an isomorphism $\lambda : L' \to
N^{\tensor 2}\tensor L$.  By the construction of $\vp$ and $\lambda$,
it follows that $(\vp,\lambda)$ is a similarity $(E,q,L) \to
(E',q',L')$, proving the converse.
\end{proof}

We now recall a moduli space theoretic characterization of quadric
fibrations.  
We first recall Grothendieck's moduli characterization of projective
bundles, see \cite[II~Prop.~7.12]{hartshorne:algebraic_geometry} or
\cite[\S5.1.5(1)]{nitsure:Quot_Hilbert}.  If $E$ is a vector bundle,
the projective bundle $p: \PP(E) \to Y$ represents the moduli functor
of \linedef{line subbundles} of $E$ on the category of $Y$-schemes:
$$
u : U \to Y \quad \mapsto \quad \left\{ N \mapto{\vp} u\pullback E \right\}
$$
of equivalence classes of invertible $\OO_U$-modules $N$ and
$\OO_U$-monomorphisms $\vp : N \to u\pullback E$ with locally free
quotient, where $\vp : N \to u^* E$ is equivalent to $\vp' : N' \to
u^* E$ if and only if there exists an $\OO_U$-module isomorphism $\mu
: N' \to N$ such that $\vp' = \vp \circ \mu$.  Given a $Y$-morphism $g
: U \to \PP(E)$, pulling back the universal line subbundle
$\OO_{\PP(E)/Y}(-1) \to p\pullback E$ via $g$ gives the a line
subbundle of $u\pullback E$.

Let $(E,q,L)$ be a quadratic form and $\pi : Q \to Y$ the associated
quadric fibration.  Then $Q$ represents the moduli functor of
\linedef{isotropic} line subbundles of $E$, i.e., line subbundles on
which the quadratic form $q : E \to L$ vanishes.

\begin{theorem}
\label{thm:moduli_quadric}
The $Y$-scheme $Q$ represents the moduli functor
$$
u : U \to Y \quad \mapsto \quad \left\{ N \mapto{\vp} u^* E \; : \;
u^*q|_N = 0 \right\} 
$$
of line subbundles $\vp : N \to u\pullback E$ such that the restriction
of the quadratic form $u\pullback q : u\pullback E \to u\pullback L$
to $N$ is identically zero.
\end{theorem}
\begin{proof}
Composing a $Y$-morphism $g : U \to Q$ with the closed embedding $j :
Q \to \PP(E)$, we obtain a line subbundle $\vp : N \to \pi\pullback
E$, where $N = (j \circ g)\pullback \OO_{\PP(E)/Y}(-1) = g\pullback
\OO_{Q/Y}(-1)$.  Thus is suffices to observe that the universal line
subbundle $\OO_{Q/Y}(-1) \to \pi\pullback E$ is isotropic for the
pull-back quadratic form $\pi\pullback q : \pi\pullback E \to
\pi\pullback L$ on $Q$.
\end{proof}

Finally, we need the following generalization of the notions of
pencils, nets, and webs of quadrics.

\begin{definition}
\label{linearspan}
Let $E$ and $L$ be vector bundles on $S$ and $r : S = \PP(L\dual) \to
Y$.  To any vector bundle-valued quadratic form $(E, \bm{q}, L)$,
define the \linedef{linear span quadratic form} $(r\pullback E, q,
\OO_{S/Y}(1))$ on $S$ by the sheaf morphism $q : r\pullback E \to
\OO_{S/Y}(1)$ over $S$ associated, via adjunction, to the sheaf
morphism $\bm{q} : E \to L = r\pushforward \OO_{S/Y}(1)$ over $Y$.
One immediately checks that the linear span is a line bundle-valued
quadratic form on $S$.  The associated quadric fibration $Q \to S$ is
called the \linedef{linear span quadric fibration}.

We apply this construction in the following situation.  Let
$(E,q_i,L_i)$ for $1 \leq i \leq m$ be a finite set of line
bundle-valued quadratic forms on the same vector bundle $E$.  Letting
$L = L_1 \oplus \dotsm \oplus L_m$, then 
$$
\bm{q} = q_1 \oplus \dotsm \oplus q_m : E \to L_1 \oplus \dotsm \oplus
L_m = L
$$ 
is a vector bundle-valued quadratic form, to which we can form the
linear span quadratic form.
\end{definition}

As we are mostly interested in linear span quadratic forms in the
sequel, we will now denote by $S$ our base space.  (We keep in mind
the case where $S \to Y$ is a projective bundle.)

In most geometric applications, we will consider quadric fibrations
with ``good'' properties, which we summarize in the following definition.

\begin{definition}
\label{def:generic-fibration}
We say that a finite set of generically (semi)regular primitive
quadratic forms $(E,q_i,L_i)$ (or quadric fibrations $Q_i \to Y$) for
$1 \leq i \leq m$ is \linedef{generic} if the following properties
hold:
\begin{enumerate}
\item the images of $L\dual_i \to S^2(E\dual)$ (associated to $q_i$
via Lemma~\ref{lem:quad}\eqref{lem:quad.3}) span an $\OO_Y$-submodule
$L\dual \subset S^2(E\dual)$ of rank $m$,

\item the associated linear span quadric fibration $Q \to S$ has
simple degeneration with regular discriminant divisor,

\item the associated relative intersection $X \to Y$ of the quadric
fibrations $Q_i \subset \PP(E)$ is a relative complete intersection.
\end{enumerate}
By a \linedef{generic relative intersection of quadrics} we mean any
relative intersection $X \to Y$ of a generic set of quadric fibrations.
\end{definition}

We now mention how the regularity of the discriminant divisor is
related to the simple degeneration hypothesis.

\begin{proposition}
\label{prop:generic_remarks}
Let $S$ be a scheme smooth over an arbitrary field $k$ of
scharacteristic $\neq 2$.  Let $\pi:Q \to S$ be a flat generically regular
quadric fibration with discriminant divisor $D$.  Then $D$ is smooth
over $k$ if and only if $Q$ is smooth over $k$ and $Q\to S$ has simple
degeneration.
\end{proposition}
\begin{proof}
As $Q$ is locally of finite presentation (by construction) and flat
(by hypothesis) over $S$, and $S$ is smooth (hence locally of finite
presentation and flat, see \cite[Tome~4~Thm.~17.5.1]{EGA4}) over $k$,
we have that $Q$ is locally of finite presentation and flat over $k$,
being a composition of such morphisms.  Thus $Q$ is smooth over $k$ if
and only if $Q$ is geometrically regular, see
\cite[Tome~4~Cor.~17.5.2]{EGA4}.  To prove the smoothness of $Q$ over
$k$, it would be enough to verify that for each point $x \in S$, the
local quadric $Q \times_S \OO_{S,x}$ is a smooth scheme over $k$,
hence geometrically regular. Indeed, for any point $y$ of $Q$, the
local ring $\OO_{Q,y}$ would then be geometrically regular, as it
resides on the local quadric $Q \times_S \OO_{S,\pi(y)}$.  If $x \in S
\bslash D$, 
then $Q \times_S \OO_{S,x}$ is actually smooth over $\OO_{S,x}$.  So
we can assume that $x \in D$.

Let $Q \to S$ be defined by a quadratic form $(E,q,L)$ of rank $n$ on
$S$.  Over the residue field of the local ring $\OO_{S,x}$ at $x$, the
regular subform $\ol{q}_0$ of $q_x$ can be diagonalized (being a
quadratic form over a field of characteristic $\neq 2$) as
$\sum_{i=1}^{r} \ol{a}_i x_i^2$ for some $0 < r < n$ (since $q_x$ is
nonzero by the flatness assumption) and $\ol{a}_i \neq 0$.  By
\cite[Cor.~3.4]{baeza:semilocal_rings}, this lifts to an orthogonal
decomposition $q_{\OO_{S,x}} = q_0 + q_1$, where $q_0 = \sum_{i=1}^r
a_i x_i^2$ is a regular quadratic form of rank $r < n$ over
$\OO_{S,x}$ lifting $\ol{q}_0$, and $q_1$ is a quadratic form of rank
$n-r$ over $\OO_{S,x}$ that vanishes modulo the maximal ideal
$\maxideal_{S,x}$. As a consequence, we find that the discriminant of
$q_{\OO_{S,x}}$, which defines $D$ at $x$, is in
$\maxideal_{S,x}^{n-r}$.
Hence, if $D$ is smooth at $x$ then the discriminant of
$q_{\OO_{S,x}}$ must be contained in $\maxideal_{S,x} \bslash
\maxideal_{S,x}^2$, and hence the regular subform $q_0$ of
$q_{\OO_{S,x}}$ has rank $n-1$, i.e., $(E,q,L)$ and $Q \to S$ have
simple degeneration at $x$.

Thus we are left to prove that if $Q \to S$ has simple degeneration,
then the smoothness of $Q$ over $k$ is equivalent to the smoothness of
$D$ over $k$.  To this end, we can assume that $k$ is algebraically
closed and pass the complete local ring $\widehat{\OO}_{S,x} \isom
k[[u_1,\dotsc,u_d]]$, where $d$ is the dimension of $S$ at $x$, see
\cite[Tome~4~Prop.~17.5.3]{EGA4}.  If $\delta \in \maxideal_{S,x}$ is
a local equation of $D$ at $x$, then up to a unit in $\widehat{\OO}_{S,x}$, the discriminant of $q_1$ (which has rank 1
by hypothesis) is $\delta$.  Then we can write (noting that over
$\widehat{\OO}_{S,x}$ all units are squares)
$$
Q \times_S \Spec \widehat{\OO}_{S,x} \isom \Spec
k[[u_1,\dotsc,u_d]][x_1,\dotsc,x_n]/(x_1^2 + \dotsm + x_{n-1}^2 +
\delta x_n^2).
$$
An application of the jacobian criterion shows that $Q \times_S \Spec
\widehat{\OO}_{S,x}$ is smooth over $k$ if and only if $\delta \in
k[[u_1,\dotsc,u_n]]$ contains a nonzero linear term in the $u_i$
(i.e., $D$ is smooth at $x$) and $c$ is a unit.
\end{proof}

Note that special cases of Proposition~\ref{prop:generic_remarks} are
proved in \cite[I~Prop.~1.2(iii)]{beauvillejaco} and
\cite[Lemma~5.2]{hassett_varilly:K3}.

We remark that $D$ can have mild singularities and yet the total space
of the quadric bundle $Q$ can still be smooth, but then the quadric
bundle must have worse degeneration over the singular points of $D$,
see \cite[I~Prop.~1.2]{beauvillejaco} for example.

\subsection{Smooth sections and hyperbolic splitting}
\label{subsec:hyperbolic_algebraic}

In this section, on the way to developing a splitting principle for
isotropic subbundles, we first establish some local algebraic results
for isotropic subbundles of possibly degenerate quadratic forms.

\begin{definition}\label{regsect}
A section $s : S \to Q$ of a quadric fibration $Q \to S$ is called
\linedef{smooth} if the image of $s$ only consists of smooth points of
the fibers of $Q\to S$.  An isotropic line subbundle $N \subset E$ of
a quadratic form $(E,q,L)$ is called \linedef{smooth} if the
associated section of its quadric fibration is smooth.
\end{definition}

We relate smoothness of the base and total space with smoothness of a section.

\begin{lemma}
\label{lem:regsect}
Let $\pi : Q \to S$ be a flat morphism of schemes with a section $s$.
If $y \in S$ is a regular point such that $x=s(y) \in Q$ is regular,
then $x \in Q_y$ is a regular point of the fiber.  In particular, if
$Q$ and $S$ are smooth schemes over a field $k$, then any section of
$\pi$ is smooth.
\end{lemma}
\begin{proof}
Denoting by $j : Q_y \to Q$ the closed embedding of a fiber, we have
local homomorphisms $\pi\sheafsharp : \OO_{S,y} \to \OO_{Q,x}$ (which
is a splitting of $s\sheafsharp : \OO_{Q,x} \to \OO_{S,y}$) and
$j\sheafsharp : \OO_{Q,x} \to \OO_{Q_y,x}$ (which is surjective),
whose composition factors through the residue field
$\kappa:=\kappa(y)=\kappa(x)$.  Thus the induced sequence of
$\kappa$-vector spaces
$$
0 \to \maxideal_{S,y}/\maxideal_{S,y}^2 \mapto{\pi\sheafsharp} 
\maxideal_{Q,x}/\maxideal_{Q,x}^2 \mapto{j\sheafsharp} 
\maxideal_{Q_y,x}/\maxideal_{Q_y,x}^2 \to 0
$$
is split exact.  If $\dim_\kappa \maxideal_{S,y}/\maxideal_{S,y}^2 =
\dim\, \OO_{S,y}$ and $\dim_\kappa \maxideal_{Q,x}/\maxideal_{Q,x}^2 =
\dim\, \OO_{Q,x}$, then $\dim_\kappa
\maxideal_{Q_y,x}/\maxideal_{Q_y,x}^2 = \dim\, \OO_{Q,x} - \dim\,
\OO_{S,y} = \dim\, \OO_{Q_y,x}$ by the local fiber dimension theorem
using the flatness of $\pi$, see
\cite[Thm.~15.1]{matsumura:commutative_ring_theory}.  Hence $x$ is
regular in the fiber $Q_y$.  For the final claim, the geometric points
$\ol{y} \in S$ and $\ol{x}=s(\ol{y}) \in Q$ are regular by hypothesis,
hence $\ol{x} \in Q_{\ol{y}}$ is regular. Thus $x \in Q_y$ is a smooth point.
\end{proof}

\begin{remark}
\label{rem:smooth_section}
Assume that $S$ is a smooth scheme over a field $k$ of characteristic
$\neq 2$ and that $\pi : Q \to S$ is a flat quadric fibration with
smooth discriminant divisor $D \subset S$.  Then by
Proposition~\ref{prop:generic_remarks}, $Q$ is smooth over $k$, and
hence by Lemma~\ref{lem:regsect}, any section of $\pi$ is smooth.
\end{remark}

We have the following algebraic
reinterpretation of the smoothness of an isotropic line subbundle in terms of the bilinear radical defined in Section \ref{subsec:Quadratic_forms_and_quadric_fibrations}.

\begin{lemma}
\label{lem:smooth_section_means}
An isotropic line subbundle $N$ of $(E,q,L)$ is smooth if and only if
$R_q \cap j\pullback N = 0$.
\end{lemma}

This motivates the following definition.  An isotropic subbundle $N
\subset E$ (of any rank) of a generically (semi)regular quadratic form
$(E,q,L)$ is called \linedef{smooth} if $R_q \cap j\pullback N = 0$.

Let $g : N \to E$ be the inclusion of an isotropic subbundle.  The
\linedef{orthogonal complement} $g^{\perp} : N^{\perp} \to E$ is
defined by the exact sequence
$$
0 \to N^{\perp} \mapto{g^{\perp}} E \mapto{\psi_{q,N}} \HHom(N,L)
$$
where $\psi_{q,N}$ is the composition of $\psi_q: E \to \HHom(E,L)$
with the projection $\HHom(E,L) \to \HHom(N,L)$.  Note that if $N$ is
isotropic then $N \subset N^{\perp}$.

Given a vector bundle $N$ and a line bundle $L$ on $S$, the
\linedef{hyperbolic quadratic form} $H_L(N)$ on $S$ is defined by the
evaluation pairing $N \oplus \HHom(N,L) \to L$.  The isotropic
subbundle $N \subset H_L(N)$ is smooth and is maximal amonst isotropic
subbundles.  Such an isotropic subbundle is called a
\linedef{lagrangian}.

\begin{lemma}[Local hyperbolic splitting]
\label{lem:local_hyperbolic_splitting}
Let $(\OO,\maxideal)$ be a local ring.  Let $N$ be an isotropic direct
summand of a generically (semi)regular quadratic form $(E,q,L)$ over
$\OO$.
If $N$ is smooth then we have an orthogonal decomposition $(E,q,L)
\isom H_L(N) \perp (N^{\perp}/N,q|_{N^{\perp}/N},L)$ such that the
inclusion of the first factor fits into the following commutative
diagram
\begin{equation}
\label{eq:local_splitting_diagram}
\begin{split}
\xymatrix@R=16pt@C=30pt{
0 \ar[r]& N \ar[r]\ar[d]& H_L(N) \ar[r]\ar[d]& \HHom(N,L) \ar[r]\ar@{=}[d]& 0\\
0 \ar[r]& N^{\perp} \ar[r]^{g^{\perp}}& E \ar[r]^(.35){\psi_{q,N}}& \HHom(N,L) & \\
}
\end{split}
\end{equation}
of free $\OO$-modules, where the top sequence arises from the
inclusion of the lagrangian $N \subset H_L(N)$.
\end{lemma}
\begin{proof}
First note that since $N$ is smooth, by
Lemma~\ref{lem:smooth_section_means} we have a decomposition $\ol{E}
\isom \ol{N} \oplus R_q \oplus \ol{M}$ over the residue field
$k(\maxideal)$ of $\OO$, where $\ol{M} = \ol{E}/(\ol{N}+R_q)$.  Fixing
a generator $L=l\OO$, we identify $L=\OO$ throughout the proof.

First, suppose $N=v\OO$ has rank 1.  Since $\ol{v} \notin R_q$ there
exists $\ol{w} \in \ol{M}$ such that $b_q(\ol{v},\ol{w}) = a\in
k(\maxideal)\mult$.  Then $\ol{v}k(\maxideal) \oplus
\ol{w}k(\maxideal)$ is a regular orthogonal summand of
$(\ol{E},\ol{q})$, which can be lifted (by
\cite[I~Cor.~3.4,~Prop.~3.2\textit{ii}]{baeza:semilocal_rings}) to a
regular orthogonal summand $H=v\OO \oplus w\OO$ of $(E,q)$, for some
lift $w \in E$ of $\ol{w}$ over $\OO$.  Hence there is an orthogonal
decomposition $(E,q) \isom (H,q|_H) \perp (H^{\perp},q|_{H^{\perp}})$.
We use the following classical trick to modify this orthogonal
decomposition.  Choosing a lift $c \in \OO\mult$ of $a\inv$, then
$\{cv,w-cq(w)v\}$ forms a hyperbolic basis of $H$ with lagrangian $N$,
hence the isometric inclusion $H_L(N) \isom (H,q|_H) \subset (E,q)$
defining the middle vertical arrow in the diagram.  The commutativity
of the diagram is then immediate.  Finally, we see that $E \isom w\OO
\oplus N^{\perp}$, hence that $N^{\perp} \isom N \oplus H^\perp$ and
thus $H^\perp \isom N^{\perp}/N$.  We thus arrive at the desired
decomposition $(E,q,L) \isom H_L(N) \perp
(N^{\perp}/N,q|_{N^{\perp}/N},L)$.

Now suppose that $N$ has rank $n$ and write $N \isom N_1 \oplus N_2$
with $\rk N_1 = 1$. Then $N \cap N_1^{\perp}/N_1$ has rank $n-1$ and
its reduction to $k(\maxideal)$ does not intersect $R_q$.  Hence we
can apply induction on $\rk N$.
\end{proof}

\begin{proposition}
\label{prop:regular_isotropic}
Let $(E,q,L)$ be a generically (semi)regular quadratic form over a
scheme $S$ and $g : N \to E$ an isotropic subbundle.  If $N$ is smooth
then the sequence
$$
0 \to N^{\perp} \mapto{g^{\perp}} E \mapto{\psi_{q,N}} \HHom(N,L) \to 0
$$
is exact.  In particular, $N^{\perp}$ is locally free.
\end{proposition}
\begin{proof}
For an isotropic $\OO_S$-submodule $N \subset E$, consider the
commutative diagram of $\OO_S$-modules
$$
\xymatrix@R=16pt@C=30pt{
0\ar[r]&N^{\perp}\ar[d]\ar[r]^{g^{\perp}}&E\ar[d]\ar[r]^(.35){\psi_{q,N}}&\HHom(N,L)\ar@{=}[d]\\ 
0\ar[r]&N^{\perp}/N\ar[r]&E/N\ar[r]^(.35){\psi_q/N}&\HHom(N,L) 
}
$$
with exact rows.  Here $\psi_q/N$ is defined by $v + N \mapsto (n
\mapsto b_q(v,n))$.  Hence the surjectivity of $\psi_{q,N}$ is
equivalent to the surjectivity of $\psi_q/N$.  Both of these are local
questions.

Let $\OO$ be the local ring of a point $x$ of $S$.  Since $N$ is
smooth, we can apply Lemma \ref{lem:local_hyperbolic_splitting} to $N$
over $\OO$ and we have an orthogonal decomposition $(E,q,L)_\OO \isom
H_L(N)_\OO \perp (N^{\perp}/N,q|_{N^{\perp}/N},L)_\OO$ of quadratic
forms over $\OO$.  But then we have the following commutative diagram
of $\OO$-modules
$$
\xymatrix@R=16pt@C=30pt{
0\ar[r]&(N^{\perp}/N)_\OO\ar@{=}[d]\ar[r]&E_\OO\ar[d]\ar[r]&H_L(N)_\OO\ar[d]\ar[r]&0\\
0\ar[r]&(N^{\perp}/N)_\OO\ar[r]&(E/N)_\OO\ar[r]^(.4){\psi_{q}/N}&\HHom(N,L)_\OO&
}
$$
whose top row defines a splitting of the inclusion $H_L(N)_\OO \to
E_\OO$ from the middle column of \eqref{eq:local_splitting_diagram}.
In particular, $\psi_q/N$ is surjective over each $\OO$, hence is
surjective over $S$.
\end{proof}

\begin{definition}
Let $(E,q,L)$ be a quadratic form over a scheme and $N \to E$ be a
smooth isotropic subbundle.  Then $q|_{N^{\perp}} : N^{\perp} \to L$
vanishes on $N$, hence defines a quadratic form $q' : N^{\perp}/N \to
L$ on $E'=N^{\perp}/N$.  We call $(E',q',L)$ the \linedef{reduced
quadratic form} associated to $N$.
\end{definition}

The following is a generalization of the hyperbolic splitting
principle for regular quadratic forms (see
\cite[I~Thm.~3.6]{baeza:semilocal_rings}).

\begin{theorem}[Isotropic splitting principle] 
\label{thm:isotropic_splitting_principle}
Let $S$ be a scheme and $(E,q,L)$ be a generically (semi)regular
quadratic form over $S$.  Let $(E',q',L)$ be the reduced quadratic
form associated to a smooth isotropic subbundle $g : N \to E$.  Then
there exists a (Zariski) locally trivial affine bundle $p : V \to S$
such that $p\pullback(E,q,L) = H_{p\pullback L}(p\pullback N) \perp
p\pullback(E',q',L)$.  In particular, then there exists a Zariski open
covering of $S$ such that $(E,q,L)|_U = H_{L|_U}(N|_U) \perp
(E',q',L)|_U$ for each element $U$ of the covering.
\end{theorem}
\begin{proof}
Following Fulton~\cite[\S2]{fulton:flag_bundles}, let $V' \subset
\HHom(E,N)$ be the sheaf of local retractions of $g : N \to E$, i.e.,
whose sections over $U \to S$ are $\OO_U$-module morphisms $\varphi :
E|_{U} \to N|_{U}$ such that $\varphi \circ g|_{U} = \id_{N}|_U$.
Then $V'$ has a (left) translation action by the $\OO_S$-module
$\HHom(\coker(g),N) \subset \HHom(E,N)$, making $V'$ into a (left)
torsor.  By abuse of notation, denote by $p':V' \to S$ the
corresponding Zariski locally trivial affine bundle, which has a
section if and only if $g$ is split over $S$.  The pullback
$p'\pullback E$ has a tautological map to $p'\pullback N$, which is a
retraction of $p'\pullback g$, hence $p'\pullback N$ is a direct
summand of $p'\pullback E$.

By Proposition~\ref{prop:regular_isotropic}, $g^{\perp} : N^{\perp}
\to E$ is locally free with locally free quotient.  Let $V \subset
\HHom(p'\pullback E,p'\pullback N^{\perp})$ be the sheaf of local
retractions of $g^{\perp}$ over $V'$.  Denote by $p: V \to S$ the
corresponding composition of Zariski locally trivial affine bundles.
Then $p\pullback N^{\perp}$ is a direct summand of $p\pullback E$.
Considering the Nine Lemma applied to the commutative diagram
$$
\xymatrix@R=16pt{
0 \ar[r] & N \ar@{=}[d] \ar[r] & N^{\perp} \ar[d] \ar[r] &
N^{\perp}/N \ar[d] \ar[r] & 0 \\
0 \ar[r] & N \ar[r]^g & E \ar[r] & E/N \ar[r] & 0 \\
}
$$
of $\OO_S$-modules, we have (using Proposition~\ref{prop:regular_isotropic})
an exact sequence
$$
0 \to N^{\perp}/N \to E/N \to \HHom(N,L) \to 0,
$$
which becomes split upon application of $p\pullback$.  Note that
$p\pullback N^{\perp}=(p\pullback N)^{\perp}$ (which can be checked
locally).  Finally, the direct sum decomposition $p\pullback E \isom
p\pullback \bigl( N \oplus \HHom(N,L)\bigr) \oplus p\pullback
(N^{\perp}/N)$, induces the required isometry of quadratic forms, as
in the affine case \cite[I~Thm.~3.6]{baeza:semilocal_rings} or
\cite[I~Prop.~3.7.1]{knus:quadratic_hermitian_forms}.  Indeed, the
associated bilinear form induces a pairing $b_q : N \times E/N^{\perp}
\to L$, which is locally checked to be a perfect pairing.  By
Proposition~\ref{prop:regular_isotropic}, this is a perfect pairing $N
\times \HHom(N,L) \to L$, which remains perfect after pulling back to
$V$.  If $p\pullback\! \HHom(N,L)$ is isotropic in $p\pullback E$,
then the pairing pulls back to the associated bilinear form of the
hyperbolic quadratic form $p\pullback H_L(N)$ and we are done.
Otherwise, we employ the same trick appearing in the proof of
Lemma~\ref{lem:local_hyperbolic_splitting} to modify the orthogonal
decomposition of $p\pullback E$, in order to make the summand
$p\pullback\bigl( N \oplus \HHom(N,L) \bigr)$ hyperbolic.
The final claim follows since $p$ has a section Zariski locally.
\end{proof}

\begin{corollary}
\label{cor:degeneration_reduction}
The quadratic forms $(E,q,L)$ and $(E',q',L)$ have the same
degeneration divisor. In particular, $(E,q,L)$ has simple degeneration
if and only if $(E',q',L)$ does.
\end{corollary}
\begin{proof}
The degeneration divisors $D$ and $D'$ of $(E,q,L)$ and $(E',q',L)$
agree over a Zariski open covering of $S$ by
Theorem~\ref{thm:isotropic_splitting_principle} since hyperbolic
spaces have trivial discriminant.  Hence $D=D'$.  The final statement
is local and follows from Lemma~\ref{lem:local_hyperbolic_splitting}.
\end{proof}

\subsection{Quadric reduction by hyperbolic splitting}
\label{subsec:hyperbolic_geometric}

We now describe a geometric interpretation, in terms of quadratic
fibrations, of the reduced quadratic form associated to a smooth
isotropic line subbundle.  Given a quadric hypersurface $Q \subset
\PP^{n-1}$ over an arbitrary field $k$ and a smooth $k$-point $x$,
consider the embedded tangent space $T_x Q$.  The intersection $T_x Q
\cap Q$ is a quadric cone with vertex $x$, over a quadric $Q' \subset
\PP^{n-2}$.  In particular, projecting $T_x Q \cap Q$ away from $x$, we
find $Q'$ as a quadric hypersurface in the target $\PP^{n-3}$, which
we call the \linedef{quadric reduction by hyperbolic splitting}
associated to $Q$ and $x$.  Recently, this construction was considered
\cite{ingalls-khalid} in the study of nets of quadrics over $\PP^2$
and associated K3 surfaces.

Conversely, consider a quadric $Q'\subset \PP^{n-3}$ and an embedding
of $\PP^{n-3}$ as a hyperplane of $\PP^{n-2}$.  Let $\varepsilon:
\widetilde{\PP}^{n-2} \to \PP^{n-2}$ be the blow-up of $\PP^{n-2}$
along $Q'$. Then the quadric $Q$ is obtained as the image of
$\widetilde{\PP}^{n-2}$ when contracting the proper transform of
$\PP^{n-3}$.

\begin{proposition}
Let $N$ be a smooth isotropic line subbundle of a primitive quadratic
form $(E,q,L)$ over a scheme $S$.  Let $Q$ and $Q'$ be the quadric
fibrations associated to $(E,q,L)$ and the reduced quadratic form
$(E',q',L)$ associated to $N$. Let $s : S \to Q$ be the smooth section
associated to $N$.  Then the fiber $Q'_x$ is the quadric reduction of
$Q_x$ associated to the smooth point $s(x)$.
\end{proposition}
\begin{proof}
Since $s(x)$ is a smooth point of $Q_x$, the projective space
$\PP(N_x^{\perp})$ is canonically identified with the embedded tangent
space of $Q_x$ at $s(x)$, see
\cite[Prop.~22.1]{elman_karpenko_merkurjev}.  Moreover, the quotient
map $N_x^{\perp} \to N_x^{\perp}/ N_x$ corresponds to the projection
away from $s(x)$.  Hence it is immediate that the reduced quadratic
form $q'_x$ on $N_x^{\perp}/N_x$, induced by $q_x$ and $N_x$, defines
the quadric $Q'_x$ on $\PP(N_x^{\perp}/N_x)$.
\end{proof}

\subsection{Even Clifford algebra and Clifford bimodule}
\label{subsec:Even_Clifford_algebra}

In this section, we give a new tensorial construction of the even
Clifford algebra of a line bundle-valued quadratic form.  In
Appendix~\ref{appendix_clifford}, we show that it coincides with the
one in \cite[\S3.3]{kuznetsov:quadrics}.  The usefulness of this
construction lies in its universal property, cf.\ \cite[\S1]{auel:Brdim}.

Let $(E,q,L)$ be a (line bundle-valued) quadratic form of rank $n$ on
a scheme $S$.  Write $n=2m$ or $n=2m+1$. Inspired
by~\cite[II~Lemma~8.1]{book_of_involutions}, we define ideals $J_1$
and $J_2$ of the tensor algebra $T(E\tensor E\tensor L\dual)$,
generated by
$$
v\tensor v \tensor f - f(q(v)), \quad \mbox{and} \quad 
u\tensor v \tensor f \tensor v \tensor w \tensor g - 
f(q(v)) \, u \tensor w \tensor g,
$$
respectively, for sections $u,v,w$ of $E$ and $f,g$ of $L\dual$.  We
define the \linedef{even Clifford algebra} of $(E,q,L)$ as the
quotient algebra
$$
\CliffAlg_0(E,q,L) = T(E\tensor E\tensor L\dual)/(J_1 + J_2).
$$
Then $\CliffAlg_0(E,q,L)$ has an $\OO_S$-module filtration
$$
\OO_S = F_0 \subset F_2 \subset \dotsm \subset F_{2m} =\CliffAlg_0(E,q,L),
$$
where $F_{2i}$ is the image of the truncation of the tensor algebra
$T^{\leq i}(E \tensor E \tensor L\dual)$ in $\CliffAlg_0(E,q,L)$, for
each $0 \leq i \leq m$.  The fact that $F_{2m}=\CliffAlg_0(E,q,L)$ is
a consequence of the Poincar\'e--Birkhoff--Witt theorem, see
\cite[Prop.~3.5]{bichsel_knus:values_line_bundles}.  This filtration
has associated graded pieces $F_{2i}/F_{2(i-1)} \isom \exterior^{2i} E
\tensor (L\dual)^{\tensor i}$.
In particular, $\CliffAlg_0(E,q,L)$ is a locally free $\OO_S$-algebra
of rank $2^{n-1}$.  By its tensorial construction, the even Clifford
algebra has the following universal property:\ given an
$\OO_S$-algebra $\ka$ and an $\OO_S$-module morphism $\psi : E \tensor
E \tensor L\dual \to \ka$ satisfying
\begin{equation}
\label{eq:univ}
\psi(v\tensor v \tensor f) = f(q(v))\cdot 1_{\ka} \quad \mbox{and} \quad 
\psi(u\tensor v \tensor f) \cdot \psi(v \tensor w \tensor g) = 
f(q(v)) \, \psi(u \tensor w \tensor g),
\end{equation}
there exists a unique $\OO_S$-algebra homomorphism $\Psi :
\CliffAlg_0(E,q,L) \to \ka$ satisfying $\psi = \Psi \circ i$, where $i
: E \tensor E\tensor L\dual \to \CliffAlg_0(E,q,L)$ is the canonical
morphism.

When $L=\OO_S$, and $(E,q)$ is a classical ($\OO_S$-valued) quadratic
form, then we can define the ``full'' \linedef{Clifford algebra} as
\begin{equation}
\label{eq:full_clifford}
\CliffAlg(E,q) = T(E)/J
\end{equation}
where $J$ is the ideal generated by $v\tensor v - q(v)$ for sections
$v$ of $E$.  As $J$ is generated by relations in even degree, the full
Clifford algebra inherits a $\Z/2\Z$-graded structure $\CliffAlg(E,q)
= \CliffAlg_0(E,q) \oplus \CliffAlg_1(E,q)$ into components of even
and odd degree.  The canonical morphism $E \tensor E \tensor \OO_S = E
\tensor E \to \CliffAlg_0(E,q)$ satisfies the universal property
\eqref{eq:univ}.  The induced $\OO_S$-algebra morphism
$\CliffAlg_0(E,q,\OO_S) \to \CliffAlg_0(E,q)$ is an isomorphism.  If
$L \neq \OO_S$, then quadratic forms with values in $L$ do not
generally enjoy a ``full'' Clifford algebra, of which the even
Clifford algebra is the even degree part.  In this case, the Clifford
bimodule is a replacement for the odd degree part.

Now we define the Clifford bimodule.  Inspired by
\cite[II~\S9]{book_of_involutions}, we proceed as follows.  The
$\OO_S$-module $E \tensor T(E \tensor E \tensor L\dual)$ has a natural
right $T(E\tensor E \tensor L\dual)$-module structure denoted by
$\tensor$; we now define a commuting left $T(E \tensor E \tensor
L\dual)$-module structure denoted by $*$ and defined by
$$
(v_1\tensor u_1\tensor f_1 \tensor \dotsm \tensor v_r \tensor u_r
\tensor f_r) * w = 
v_1\tensor (u_1\tensor v_2 \tensor f_1 \tensor u_2 \tensor v_3 \tensor f_2
\tensor \dotsm \tensor u_r \tensor w \tensor f_r),
$$
for sections $w, u_i, v_i$ of $E$ and $f_i$ of $L\dual$.  We define
the \linedef{Clifford bimodule} of $(E,q,L)$ as the quotient module
$$
\CliffAlg_1(E,q,L) = E \tensor T(E\tensor E \tensor L\dual)/(E\tensor
J_1 + J_1 * E).
$$
One immediately checks that $E \tensor J_2 \subset J_1 * E$ and $J_2 *
E \subset E \tensor J_1$, hence $\CliffAlg_1(E,q,L)$ inherits a
$\CliffAlg_0(E,q,L)$-bimodule structure.  Also, $\CliffAlg_1(E,q,L)$
has an $\OO_S$-module filtration
$$
E = F_1 \subset F_3 \subset \dotsm \subset F_{2m+1} =\CliffAlg_1(E,q,L),
$$
where $F_{2i+1}$ is the image of the truncation $E \tensor T^{\leq
i}(E \tensor E \tensor L\dual)$ in $\CliffAlg_1(E,q,L)$, for each $0
\leq i \leq m$.  This filtration has associated graded pieces
$F_{2i+1}/F_{2i-1} \isom \exterior^{2i+1} E \tensor (L\dual)^{\tensor
i}$.  In particular, $\CliffAlg_1(E,q,L)$ is a locally free
$\OO_S$-module of rank $2^{n-1}$.

For a list of basic functorial properties of the even Clifford algebra
and Clifford bimodule, see Appendix~\ref{appendix_clifford}.  In
particular,
Proposition~\ref{prop:properties}~\eqref{C_0-proj_similarity} shows
that the $\OO_S$-algebra isomorphism class $\kc_0(E,q,L)$ is an
invariant of the projective similarity class of $(E,q,L)$, hence by
Proposition \ref{prop:proj_sim_quadric}, of the associated quadric
fibration $\pi : Q \to S$.  We can thus speak of \emph{the} even
Clifford algebra (defined up to isomorphism) of a quadric fibration.

\subsection{The discriminant cover:\ even rank case}
\label{subsec:clifford_discriminant}

Let $(E,q,L)$ be a generically regular quadratic form of even rank on
a scheme $S$ and $\CliffZ = \CliffZ(E,q,L)$ the center of
$\CliffAlg_0=\CliffAlg_0(E,q,L)$. We call the associated affine
morphism $f : T = \SSpec \CliffZ \to S$ the \linedef{discriminant
cover}.

\begin{lemma}
\label{lem:disc_flat}
Let $(E,q,L)$ be a generically regular quadratic form of even rank
over a locally factorial integral scheme. Then $\CliffZ$ is a locally
free $\OO_S$-algebra of rank 2, hence $f : T \to S$ is a finite flat
morphism of degree 2.
\end{lemma}
\begin{proof}
As pointed out before
\cite[IV~Prop.~4.8.3]{knus:quadratic_hermitian_forms}, the proof of
the statement in the affine case
\cite[IV~Prop.~4.8.1]{knus:quadratic_hermitian_forms} goes through
over any factorial domain.
\end{proof}

We denote by $\CliffB_0 = \CliffB_0(E,q,L)$ the $\OO_T$-algebra
associated to the $\CliffZ$-algebra $\CliffAlg_0$.  By
Proposition~\ref{prop:properties}~\eqref{C_0-center}, $\CliffZ$ is 
\'etale over every point of $S$ where $(E,q,L)$ is regular.
By Proposition~\ref{prop:properties}~\eqref{C_0-Azumaya}, $\CliffB_0$
is Azumaya over every point of $T$ laying over a point of
$S$ where $(E,q,L)$ is regular.  The following generalization will be
extremely important in the sequel (cf.\ \cite[Sect. 3.5]{kuznetsov:quadrics}).

\begin{proposition}
\label{prop:simple_azumaya}
Let $(E,q,L)$ be a generically regular quadratic form of even rank on
a locally factorial integral scheme $S$. Let $f : T \to S$ be the
discriminant cover.  Then $\CliffB_0$ is an Azumaya algebra over every
point of $T$ lying over a point of $S$ where $(E,q,L)$ has simple
degeneration.  In particular, if $(E,q,L)$ has simple degeneration on
$S$, then $\CliffB_0$ is an Azumaya algebra on $T$.
\end{proposition}
\begin{proof}
Let $(\OO,\maxideal)$ be the local ring of a point of $S$ where
$(E,q,L)$ has simple degeneration.  The center $\CliffZ_\OO$ of
$\CliffAlg_0$ restricted to $\OO$, is a free $\OO$-algebra of rank 2
by Lemma~\ref{lem:disc_flat}.  Suppressing the choice of a
trivialization of $L$ over $\OO$, we write $(E,q)_\OO = (E_1,q_1)
\perp (\OO,\quadform{\pi})$ for $q_1$ a semiregular quadratic form of
odd rank over $\OO$ and $\pi \in \OO$.  We claim that the map $\psi :
\CliffAlg_0(E_1,q_1) \tensor_{\OO} \CliffZ_\OO \to \CliffAlg_0$,
defined by the canonical monomorphism $\CliffAlg_0(E_1,q_1) \to
\CliffAlg_0$ (and multiplication in $\CliffAlg_0$), is a
$\CliffZ_\OO$-algebra isomorphism.  As $(E_1,q_1)$ is semiregular of
odd rank, $\CliffAlg_0(E_1,q_1)$ is an Azumaya algebra on $S$ (cf.\
Proposition~\ref{prop:properties}(\ref{C_0-center})), hence remains
Azumaya upon tensoring by $\CliffZ_\OO$.  In particular, $\psi$ is a
monomorphism.  By
\cite[IV~Prop.~7.3.1]{knus:quadratic_hermitian_forms}, $\psi$ is
surjective upon identifying $\CliffZ \isom
\CliffAlg(\CliffZ_1(E_1,q_1) \tensor \quadform{-\pi})$, where
$\CliffZ_1(E_1,q_1) = \CliffZ(E_1,q_1) \cap \CliffAlg_1(E_1,q_1)$ is
an invertible $\OO$-module.  This shows that $\CliffB$ is a locally
free $\OO_T$-algebra of constant rank that is Zariski locally Azumaya,
hence is Azumaya.  A different proof can be found in
\cite[Prop.~3.13]{kuznetsov:quadrics}, cf.\
\cite[Prop.~1.13]{auel_parimala_suresh}.
\end{proof}

\subsection{The discriminant stack:\ odd rank case}
\label{subsec:clifford_discriminant_odd}

Let $(E,q,L)$ be a generically regular quadratic form of odd rank
$n=2m+1$ on $S$.  We further assume that 2 is invertible on $S$.
Recall that the discriminant divisor $D \subset S$ is the zero locus
of the global section $\disc(q) \in \Gamma(Y,(\det E\dual)^{\tensor 2}
\tensor L^{\tensor n})$.  We define the \linedef{discriminant stack}
to be the square root stack $\stack{T}$ associated to the pair $((\det
E\dual)^{\tensor 2} \tensor L^{\tensor n}, \disc(q))$ over $S$, cf.\
\cite[\S2.2]{cadman}.  As 2 is invertible, $\stack{T}$ is a
Deligne--Mumford stack (see \cite[Thm.~2.3.3]{cadman}) and has coarse
moduli space $f : \stack{T} \to S$ (see \cite[Cor.~2.3.7]{cadman}).
Furthermore, if $S$ and $D$ are smooth over a field $k$, then
$\stack{T}$ is smooth (see \cite[Ex.~2.4.5]{cadman}).

We have the following results analogous to Lemma~\ref{lem:disc_flat}
and Proposition~\ref{prop:simple_azumaya} (cf.\
\cite[Sect. 3.6]{kuznetsov:quadrics}).  Recall the construction of the
``full'' Clifford algebra \eqref{eq:full_clifford} in the case where
$L=\OO_S$.

\begin{lemma}
\label{lem:center_full_odd}
Let $(E,q)$ be a generically regular quadratic form (with values in
$\OO_S$) of odd rank over a locally factorial integral scheme. Then
the center $\CliffZ$ of the ``full'' Clifford algebra is a locally
free $\OO_S$-algebra of rank 2.
\end{lemma}
\begin{proof}
See the proof of Lemma~\ref{lem:disc_flat}.
\end{proof}

\begin{proposition}
\label{prop:simple_azumaya_odd}
Let $(E,q,L)$ be a generically regular quadratic form of odd rank
$n=2m+1$ on a locally factorial integral scheme $S$ with 2 invertible.
Let $f : \stack{T} \to S$ be the discriminant stack.  Then there
exists a locally free algebra $\CliffB_0$ on $\stack{T}$ such that
$f\pushforward \CliffB_0 = \CliffAlg_0$, and which is Azumaya over
every point of $\stack{T}$ lying over a point of $S$ where $(E,q,L)$
has simple degeneration.  In particular, if $(E,q,L)$ has simple
degeneration on $S$, then $\CliffB_0$ is an Azumaya algebra on
$\stack{T}$.
\end{proposition}
\begin{proof}
We first give a local description of the discriminant stack (cf.\
\cite[\S3.6]{kuznetsov:quadrics}).  Let $U \subset S$ be a Zariski
affine open trivializing $E$ and also $L$ via $l : \OO_{U} \to
L|_{U}$.  Then there is an $\OO_{U}$-valued quadratic form
$(E|_U,l\inv q|_U)$ on $U$.  Hence we can consider the classical
``full'' Clifford algebra $\CliffAlg_U=\CliffAlg(E|_U,l\inv q|_U)$,
and $\CliffZ_U$ its center.  By Lemma~\ref{lem:center_full_odd},
$\CliffZ_U$ is a locally free $\OO_{U}$-algebra of rank 2.  As $S$ is
integral, we have $\CliffZ_U \isom \OO_U[z]/(z^2-d)$, where $d =
l^{-n}\disc(q)|_U$ by
\cite[IV~Props.~4.8.5(1),~4.8.9(1)]{knus:quadratic_hermitian_forms}.
Then $T_U = \SSpec \CliffZ_U \to U$ is a finite flat double cover
branched along $D \cap U$.  By the local description of the root stack
(cf.\ \cite[Ex.~2.4.1]{cadman}), we have that $\stack{T}|_U \isom
[T_U/\mu_2]$, where $\mu_2$ acts $\OO_U$-linearly on $T_U$ via the
trivial character on $1_{\CliffZ_U}$ and the standard character on
$z$.  As 2 is invertible, this coincides with the action of the Galois
group $\Z/2\Z \isom \mu_2$ of $T_U$ over $U$.

We now construct the $\OO_{\stack{T}}$-algebra $\CliffB_0$.  For a
Zariski affine open $U \subset S$, let $T_U \to U$ and $\CliffAlg_U$
as above, and denote by $\CliffB_U$ the $\OO_{T_U}$-algebra associated
to $\CliffAlg$.  Negation on $E$, when restricted to $U$, induces an
$\OO_U$-automorphism of $\CliffAlg_U$ of order 2, which restricts to
the Galois automorphism of $\CliffZ_U$ over $\OO_U$ (since $\CliffZ_U$
is generated by an element $z$ in odd degree).  Hence $\CliffB_U$ is
naturally a $\Z/2\Z\isom \mu_2$-equivariant $\OO_{T_U}$-algebra for
the Galois action on $T_U$ over $U$, hence defines an
$\OO_{[T_U/\mu_2]}$-algebra.  Since the equivariant structure over $U$
is induced by a global automorphism of $E$ over $S$, this construction
glues over a Zariski affine open cover to yield an
$\OO_{\stack{T}}$-algebra $\CliffB_0$.

To check where $\CliffB_0$ is Azumaya, we work locally.  Let
$(\OO,\maxideal)$ be the local ring of a point of $S$ where $(E,q,L)$
has simple degeneration and let $l : \OO \to L_\OO$ be a
trivialization (if $(E,q,L)$ is regular at $\OO$, then the argument
simplifies).  We can write $(E_{\OO},l\inv q_{\OO}) = (E_1,q_1) \perp
(\OO,\quadform{\pi})$ for $q_1$ a regular quadratic form of even rank
over $\OO$ and $\pi \in \OO$.  The center $\CliffZ_\OO$ of the full
Clifford algebra $\CliffAlg_\OO = \CliffAlg(E_\OO,l\inv q_\OO)$ is a
free $\OO$-algebra of rank 2 by Lemma~\ref{lem:center_full_odd}.  We
claim that the map $\psi : \CliffAlg(E_1,q_1) \tensor_{\OO}
\CliffZ_\OO \to \CliffAlg_\OO$, defined by the canonical monomorphism
$\CliffAlg(E_1,q_1) \to \CliffAlg_\OO$ and multiplication in
$\CliffAlg_\OO$, is a $\CliffZ_\OO$-algebra isomorphism.  As
$(E_1,q_1)$ is regular of even rank, $\CliffAlg(E_1,q_1)$ is an
Azumaya $\OO$-algebra (cf.\
Proposition~\ref{prop:properties}(\ref{C_0-center})), hence remains
Azumaya upon tensoring with $\CliffZ_\OO$.  In particular, $\psi$ is a
monomorphism.  By
\cite[IV~Prop.~7.2.1]{knus:quadratic_hermitian_forms}, $\psi$ is
surjective upon identifying $\CliffZ_\OO \isom \CliffAlg(\CliffZ^0_\OO
\tensor \quadform{\pi})$, where $\CliffZ_\OO^0 = \ker(\tr :
\CliffZ(E_1,q_1) \to \OO)$ coincides with the submodule generated by
$z$ as above.  This shows that $\CliffB_0$ is glued from
$\mu_2$-equivariant Azumaya algebras of constant rank, hence defines
an Azumaya $\OO_{\stack{T}}$-algebra.  A different proof can be found
in \cite[Prop.~1.15]{kuznetsov:quadrics}.
\end{proof}

\subsection{A Morita equivalence}
\label{subsec:Morita_equivalence}

The main result of this section is Theorem~\ref{thm:3}.  Leading up to
its proof, we provide some useful formulas for the even Clifford
algebra of a (line bundle-valued) quadratic form.  See
\cite[IV~Thm.~1.3.1]{knus:quadratic_hermitian_forms} for analogous
formulas when the value line bundle is trivial.

\begin{lemma}
\label{lem:even_clifford_perp}
Let $(E_1,q_1,L)$ and $(E_2,q_2,L)$ be quadratic forms over a scheme
$S$.  Then there is an isomorphism of $\OO_S$-algebras
$$
\CliffAlg_0(q_1 \perp q_2) \isom
\CliffAlg_0(q_1)\tensor\CliffAlg_0(q_2) \oplus
\CliffAlg_1(q_1)\tensor\CliffAlg_1(q_2)\tensor L\dual
$$
where the algebra structure on the right is given by multiplication in
$\kc_0$, by the module action of $\kc_0$ on $\kc_1$, and by the
multiplication map in Proposition
\ref{prop:properties_C_1}~\eqref{C_1-C_0-module}.
\end{lemma}
\begin{proof}
See \cite[Eq.~9]{auel:Brdim} for a proof making use of the universal
property \eqref{eq:univ}.
\end{proof}

\begin{proposition}
\label{prop:split_hyperbolic_case}
Let $(E,q,L) \isom H_L(N) \perp (E',q',L)$ be an orthogonal
decomposition of quadratic forms on a scheme $S$, with $N$ a line
bundle and $(E',q',L')$ primitive.  Then we have an $\OO_S$-algebra
isomorphism
$$
\CliffAlg_0(q) \isom
\EEnd_{\CliffAlg_0(q')}(P),
$$
where $P = \CliffAlg_0(q')\oplus N\dual \tensor \CliffAlg_1(q')$ is a
locally free right $\CliffAlg_0(q')$-module of rank 2.  In particular,
$\CliffAlg_0(q)$ is Morita $S$-equivalent to
$\CliffAlg_0(q')$. Furthermore, assuming that $S$ is integral and
locally factorial and that $E$ has even rank, then there is an
$\OO_S$-algebra isomorphism $\CliffZ(q)\isom \CliffZ(q')$ of centers,
with respect to which the induced left and right actions of the
centers on $P$ commute.
\end{proposition}
\begin{proof}
By Lemma~\ref{lem:even_clifford_perp}, we have 
$$
\CliffAlg_0(q) \isom
(\OO_S \oplus \OO_S)\tensor\CliffAlg_0(q') \oplus (N\oplus
\HHom(N,L))\tensor \CliffAlg_1(q')\tensor L\dual.
$$
Then we have to trace through the multiplication to show that we can
write this algebra in a block matrix decomposition
$$
\left(
\begin{array}{cc}
\CliffAlg_0(q') & N \tensor \CliffAlg_1(q') \tensor L\dual \\
N\dual \tensor \CliffAlg_1(q') & \CliffAlg_0(q')
\end{array}
\right)
$$
using the fact that
$$
\HHom_{\CliffAlg_0(q')}\bigl(N\dual \tensor
\CliffAlg_1(q'),\CliffAlg_0(q')\bigr) \isom 
N \tensor\HHom_{\CliffAlg_0(q')}\bigl( \CliffAlg_1(q'),\CliffAlg_0(q')\bigr)\isom 
N \tensor \CliffAlg_1(q') \tensor L\dual.
$$
In this last step, we needed the fact that the multiplication map from
Proposition~\ref{prop:properties_C_1}\eqref{C_1-C_0-module} yields an
isomorphism
$$
\CliffAlg_1(q') \tensor_{\CliffAlg_0(q')} \CliffAlg_1(q') \to
\CliffAlg_0(q') \tensor L
$$
of $\CliffAlg_0(q')$-bimodules (since $q'$ is assumed primitive).
For the final statement, restricting the isomorphism of
Lemma~\ref{lem:even_clifford_perp} to the center yields an
$\OO_S$-algebra monomorphism $\psi: \CliffZ(q) \to \CliffZ(H_L(N))
\tensor \CliffZ(q')$.  Note that $\CliffZ(H_L(N)) \isom \OO_S \times
\OO_S$.  Considering the Galois automorphisms $\iota$ of
$\CliffZ(H_L(N))$ and $\iota'$ of $\CliffZ(q')$ over $\OO_S$, then as
in the proof of \cite[Prop.~2.3]{auel:Brdim}, we compute that
$(1-\iota\tensor \iota')$ annihilates the image of $\psi$.  Hence we
arrive at an $\OO_S$-monomorphism $\CliffZ(q) \to \CliffZ(q')$, which
is an isomorphism over $S\bslash D$ by
\cite[IV~\S4.4]{knus:quadratic_hermitian_forms} and over $D$ by
Corollary~\ref{cor:degeneration_reduction}, where $D$ is the
discriminant divisor of $q$.
\end{proof}

Let $S$ be a locally ringed topos.  We say that $\OO_S$-algebras $\ka$
and $\ka'$ are \linedef{Morita $S$-equivalent} if the fibered
categories of finitely presented (right) modules $\COH(S,\ka)$ and
$\COH(S,\ka')$ are $S$-equivalent, i.e., inducing $\OO_S$-module
isomorphisms on $\HHom$ sheaves, see
Lieblich~\cite[\S2.1.4]{lieblich:thesis} or Kashiwara--Schapira
\cite[\S19.5]{kashiwara_schapira:categories_sheaves}.  First recall
``Morita theory'' in the context of locally ringed topoi, which is
proved in \cite[Prop.~2.1.4.4]{lieblich:thesis} and
\cite[Thm.~19.5.4]{kashiwara_schapira:categories_sheaves}.

\begin{theorem}[Morita theory]
\label{thm:morita}
Let $S$ be a locally ringed topos and $\ka$ and $\ka'$ be
$\OO_S$-algebras.  Then:
\begin{enumerate}
\item[I.] Any invertible $\ka'\tensor_{\OO_S}\ka\op$-module $\kp$ gives
rise to Morita $S$-equivalences
$$
\kp \tensor_{\ka}- : \COH(S,\ka) \to \COH(S,\ka'), \qquad 
\HHom_{\ka'}(\kp,-) : \COH(S,\ka') \to \COH(S,\ka).
$$
\item[II.] Any Morita $S$-equivalence $F : \COH(S,\ka) \to \COH(S,\ka')$
is isomorphic to $\kp \tensor_{\ka}-$ for some invertible
$\ka'\tensor_{\OO_S}\ka\op$-module $\kp$.
\end{enumerate}
\end{theorem}

We also recall that if $\ka$ and $\ka'$ are Azumaya algebras on $S$
then Morita $S$-equivalence is equivalent to Brauer equivalence, cf.\ \cite[Prop.~2.1.5.6]{lieblich:thesis}.

Let $T \to S$ be a morphism of locally ringed topoi and fix a
$\Gm$-gerbe $\gerbe{Y}\to T$ (e.g., the gerbe of splittings of an
fixed Azumaya algebra).  Given a sheaf $\kv$ on $\gerbe{Y}$, denote by
$\kv \times \Gm \to \kv$ the natural inertial action.  We recall that
$\gerbe{Y}$ defines a class in $H^2(T,\Gm)$, see
\cite[IV~\S3]{giraud}.  The Brauer group of Azumaya algebras $\Br(T)$
naturally injects into the torsion subgroup of $H^2(T,\Gm)$, see
\cite[I~\S2]{grothendieck:Brauer}.

\begin{definition}
A \linedef{$\gerbe{Y}$-twisted sheaf} on $T$ is an
$\ko_{\gerbe{Y}}$-module $\kv$ such that the inertial action $\kv
\times \Gm \to \kv$ is equal to the right action associated to the
left module action $\Gm \times \kv \to \kv$.
\end{definition}

\begin{proposition} 
\label{approximation} 
Let $T$ be an integral noetherian Artin stack with generic point
$\eta$.  Let $\gerbe{Y} \to T$ be a $\Gm$-gerbe.  Then:
\begin{enumerate}
\item Any quasi-coherent $\ko_{\gerbe{Y}}$-module is the colimit of its
coherent $\ko_{\gerbe{Y}}$-submodules.
\item Any coherent $\gerbe{Y}$-twisted sheaf on $\eta$ extends to a coherent
$\gerbe{Y}$-twisted sheaf on $T$.
\end{enumerate}
\end{proposition}
\begin{proof}
The first statement follows from
\cite[Prop.~15.4]{champs-algebriques}, as any $\Gm$-gerbe over an
Artin stack is itself an Artin stack, cf.\
\cite[2.2.1.5]{lieblich:moduli_twisted_sheaves}.  The second statement
immediately follows from the first.
\end{proof}

We need the following generalization of
\cite{auslander_goldman:brauer_group_commutative_ring} or
\cite[II~Cor.~1.10]{grothendieck:Brauer} to the setting of stacks.

\begin{proposition}
\label{prop:generic_injectivity}
Let $T$ be an regular integral noetherian Artin stack with generic
point $\eta$.  Then the restriction map $\Br(T) \to \Br(\eta)$ is
injective.
\end{proposition}
\begin{proof}
The proof in \cite[Prop.~3.1.3.3]{lieblich:period-index} (with the
hypothesis that $T$ is an algebraic space) immediately generalizes to
any noetherian Artin stack by appealing to
Proposition~\ref{approximation}.
\end{proof}

This can be interpreted as saying that two Azumaya algebras on a
noetherian Artin stack that are Zariski locally Brauer equivalent are
actually Brauer equivalent. Finally, we can state the main result of
this section, of which Theorem~\ref{thm:3} is a special case.

\begin{theorem}
\label{thm:hyperbolic_splitting_morita}
Let $(E,q,L)$ be a quadratic form over a regular integral scheme $S$,
with simple degeneration along a regular divisor $D$.  In the case of
odd rank, assume that 2 is invertible on $S$.  Let $(E',q',L')$ be the
reduced quadratic form associated to a smooth isotropic subbundle $N
\to E$.
Then the even Clifford algebras $\CliffAlg_0(E,q,L)$ and
$\CliffAlg_0(E',q',L)$ are Morita $S$-equivalent.  Furthermore, the
associated Azumaya algebras $\CliffB_0$ and $\CliffB_0'$ on the
discriminant cover or stack are Brauer equivalent.
\end{theorem}

As an immediate consequence, we get the following in terms of derived
categories.  Denote by $\beta$ the Brauer class of the Azumaya algebra
$\CliffB_0$ on the discriminant cover $T \to S$ or stack $\stack{T}
\to S$, see
\S\ref{subsec:clifford_discriminant}--\ref{subsec:clifford_discriminant_odd}.

\begin{corollary}
\label{morita}
Let $Q \to S$ be a quadric fibration over a regular integral scheme
with simple degeneration along a regular divisor and $Q' \to S$
obtained by quadric reduction along a smooth section.  In odd relative
dimension, assume that 2 is invertible on $S$.  Let $\kc_0$ and
$\kc_0'$ be the respective even Clifford algebras.  Then there are
equivalences $\Db(S,\kc_0) \simeq \Db(S,\kc_0')$ and $\Db(T,\beta)
\simeq \Db(T,\beta')$ or $\Db(\stack{T},\beta) \simeq
\Db(\stack{T},\beta')$ depending on the parity of the relative
dimension.  
\end{corollary}

\begin{proof}[Proof of Theorem~\ref{thm:hyperbolic_splitting_morita}]
By induction on the rank of $N$ using
Theorem~\ref{thm:isotropic_splitting_principle}, we can assume that $N
\subset E$ is an isotropic line subbundle.  By
Corollary~\ref{cor:degeneration_reduction}, $(E',q',L)$ also has
simple degeneration along $D$.  We can assume that $D$ is nonempty,
otherwise the statement of the theorem is classical, see
\cite[IV~Prop.~8.1.1]{knus:quadratic_hermitian_forms}.  Write
$\CliffAlg_0 = \CliffAlg_0(E,q,L)$ and
$\CliffAlg_0'=\CliffAlg_0(E',q',L)$.

First we handle the case of even rank.  Let $f : T \to S$ be the
discriminant cover as in \S\ref{subsec:clifford_discriminant} and
$\CliffB_0$ and $\CliffB_0'$ the Azumaya algebras on $T$ associated to
$\CliffAlg_0$ and $\CliffAlg_0'$ by
Proposition~\ref{prop:simple_azumaya}.  Since $S$ and $D$ are regular
and $f:T \to S$ is a finite flat morphism branched along $D$, we have
that $T$ is regular.  By
Theorem~\ref{thm:isotropic_splitting_principle} and
Propositions~\ref{prop:split_hyperbolic_case}, we have that
$\CliffB_0$ and $\CliffB_0'$ are Zariski locally Brauer equivalent,
hence they are Brauer equivalent
Proposition~\ref{prop:generic_injectivity}.  In particular,
$\CliffAlg_0 = f\pushforward \CliffB_0$ and $\CliffAlg_0'=
f\pushforward \CliffB_0'$ are Morita $S$-equivalent.

Now we deal with the case of odd rank.  
Let $f : \stack{T} \to S$ be the discriminant stack as in
\S\ref{subsec:clifford_discriminant_odd} and $\CliffB_0$ and
$\CliffB_0'$ the Azumaya algebras on $\stack{T}$ associated to
$\CliffAlg_0$ and $\CliffAlg_0'$ by
Proposition~\ref{prop:simple_azumaya_odd}.  Since $S$ and $D$ are
regular, the local discriminant covers are regular, hence $\stack{T}$
is regular, being covered by quotient stacks of regular schemes by
finite \'etale group schemes (we use again our hypothesis that 2 is
invertible).  Since $S$ is integral and $D$ is nonempty, $\stack{T}$
is integral.  By Theorem~\ref{thm:isotropic_splitting_principle} and
Propositions~\ref{prop:split_hyperbolic_case}, we have that
$\CliffAlg_0$ and $\CliffAlg_0'$ are Zariski locally Morita equivalent
on $S$, hence $\CliffB_0$ and $\CliffB_0'$ are Zariski locally Brauer
equivalent on $\stack{T}$ (by the uniqueness of the coarse moduli
space), hence they are Brauer equivalent by
Proposition~\ref{prop:generic_injectivity} (hence Morita
$\stack{T}$-equivalent). In turn, this implies that $\CliffAlg_0 =
f\pushforward \CliffB_0$ and $\CliffAlg_0'= f\pushforward \CliffB_0'$
are Morita $S$-equivalent.
\end{proof}

\begin{remark}
\label{rem:FM_kernel}
The equivalence in the proof of Corollary~\ref{morita} is actually a
Fourier--Mukai functor whose kernel is an object $P \in \Db(S\times
S)$ whose local structure over $S$ (see
Theorem~\ref{thm:isotropic_splitting_principle}) is described by
Proposition~\ref{prop:split_hyperbolic_case}.
\end{remark}

\subsection{Two useful results from the algebraic theory of quadratic forms}

We use two results from the classical algebraic theory of quadratic
forms:\ the Amer--Brumer theorem and a Clifford algebra condition for
isotropy of quadratic forms of rank 4.  To fit the context, we will
state these results in the geometric language of quadric fibrations.

\begin{theorem}[Amer--Brumer Theorem]
\label{amer-brumer}
Let $X \to Y$ be the relative complete intersection of two quadric
fibrations over an integral scheme $Y$ and let $Q \to S$ be the
associated linear span quadric fibration.  Then $X \to Y$ has a
rational section if and only if $Q \to S$ has a rational section.
Furthermore, if $X \to Y$ has a section (resp.\ smooth section), then
so does $Q \to S$.
\end{theorem}
\begin{proof}
The first assertion immediately reduces to the classical Amer--Brumer
theorem for fields (see \cite[Thm.~17.14]{elman_karpenko_merkurjev})
by going to the generic point.  Any section of $X \to Y$ immediately
yields a section of $Q \to S$ and the final assertion follows since
the regular locus of a fiber of $X \to Y$ is contained in the
intersection of the regular loci of the corresponding fibers of the
two quadric fibrations.
\end{proof}

\begin{remark}
\label{rem:Amer}
There is an amplification of the Amer--Brumer Theorem
(cf. Leep~\cite[Thm.~2.2]{leep:amer_brumer_arbitrary}, see also
\cite{amer:theorem} or \cite{pfister:amer}), stating that $X \to Y$
rationally contains a linear subspace of a given dimension if and only
if the linear span $Q \to S$ does.
\end{remark}

We deduce one corollary which will be useful in the sequel.  While (at
least in characteristic zero) we could appeal to
\cite{campana_peternell_pukhlikov} or \cite{graber_harris_starr},
here we give a direct argument.

\begin{lemma}
\label{lem:C_2}
Let $X \to Y$ be a relative complete intersection of two quadric
fibrations over a smooth complete curve $Y$ over an algebraically
closed field.  If $X \to Y$ has relative dimension $> 2$ then it has a
section.
\end{lemma}
\begin{proof}
Let $Q \to S$ be the associated linear span quadric fibration
represented by a quadratic form $q$ over $S$ of rank $\geq 5$.  Since
$Y$ is an integral curve, $S$ is an integral surface, whose function
field $K$ (over an algebraically closed field) is thus a $C_2$-field
(cf.\ \cite[Thm.~97.7]{elman_karpenko_merkurjev}).  Hence $q$ has a
nontrivial zero over $K$, i.e., $Q \to S$ has a rational section.
Then by the Amer--Brumer Theorem~\ref{amer-brumer}, $X\to Y$ has a
rational section, hence a section (since $Y$ is a smooth curve and $X$
is proper).
\end{proof}

Now we come to a geometric rephrasing of a classical fact about the
isotropy of quadratic forms of rank 4.

\begin{theorem}
\label{thm:well-known}
Let $\pi : Q \to S$ be a quadric surface fibration with simple
degeneration along a smooth divisor over a regular integral scheme $S$
and let $T \to S$ be its discriminant cover.  Then $\CliffB_0 \in \Br(T)$
is trivial if and only if $\pi$ has a rational section.
\end{theorem}
\begin{proof}
Let $(E,q,L)$ be a quadratic form of rank 4 over $S$ whose associated
quadric surface fibration is $Q \to S$.  Let $K$ be the function field
of $S$ and $L$ the function field of $T$.  Then on the generic fiber,
$\CliffAlg_{0,K}$ is the even Clifford algebra of the quadratic form
$(E,q,L)_K$ of rank 4 over $K$.  By
\cite[Thm.~6.3]{knus_parimala_sridharan:rank_4} (also see
\cite[2~Thm.~14.1,~Lemma~14.2]{scharlau} in characteristic $\neq 2$
and \cite[II~Prop.~5.3]{baeza:semilocal_rings} in characteristic $2$),
the regular quadratic form $q_K$ is isotropic (i.e., $\pi : Q \to S$
has a rational section) if and only if $q_L$ is isotropic if and only
if $\CliffB_{0,L} \in \Br(L)$ is trivial.  Since $T$ is regular (since
$S$ and the discriminant divisor are), $\Br(T) \to \Br(L)$ is injective
(see Proposition~\ref{prop:generic_injectivity}).  In particular,
$\CliffB_0 \in \Br(T)$ is trivial if and only if $\CliffB_{0,L} \in
\Br(L)$ is trivial.
\end{proof}

We shall state an often (albeit implicitly) used corollary of this
after recalling some obvious facts about rationality of relative
schemes.  Let $k$ be an arbitrary field and $S$ an integral $k$-scheme
with function field $k(S)$.  To any morphism $\pi : Z \to S$ denote by
$Z_{k(S)} = Z \times_{S} \Spec k(S) \to \Spec k(S)$ its generic fiber.

\begin{lemma}
\label{lem:rationality}
Let $Z$ and $Z'$ be schemes over an integral $k$-scheme $S$.
\begin{enumerate}
\item If the $k(S)$-schemes $Z_{k(S)}$ and $Z'_{k(S)}$ are
$k(S)$-birational then the $k$-schemes $Z$ and $Z'$ are $k$-birational.

\item Assume that $S$ is $k$-rational.  If the $k(S)$-scheme
$Z_{k(S)}$ is $k(S)$-rational then the $k$-scheme $Z$ is $k$-rational.
\end{enumerate}
\end{lemma}
\begin{proof}
For (1), note that we have a canonical isomorphism of function fields
$k(Z') \isom k(S)(Z'_{k(S)})$.  Hence if $Z_{k(S)}$ and $Z'_{k(S)}$
are $k(S)$-birational, then the total ring of fractions $k(Z) \isom
k(S)(Z_{k(S)})$ and $k(Z') \isom k(S)(Z'_{k(S)})$ are
$k(S)$-isomorphic.  In particular, they are $k$-isomorphic, hence $Z$
and $Z'$ are $k$-birational.

For (2), we apply (1) with $Z'=\PP^n_S$.  First note that since $S$ is
$k$-rational, we have that $\PP^n_S$ is $k$-rational.  Now by (1),
$Z_{k(S)}$ is $k(S)$-rational implies $Z_{k(S)}$ is
$k(S)$-birational to $\PP^n_{k(S)}$ implies $Z$ is $k$-birational to
$\PP^n_S$ implies that $Z$ is $k$-rational.
\end{proof}

\begin{corollary}
\label{cor:well-known_rationality}
With the hypotheses of Theorem~\ref{thm:well-known}, assume that $S$
is a rational $k$-scheme.  If the class of $\CliffB_0$ in $\Br(T)$ is
trivial then $Q$ is $k$-rational.
\end{corollary}
\begin{proof}
Let $Q_{k(S)}$ be the generic fiber of $Q \to S$, which is a smooth
quadric over the function field $k(S)$.  Note that $Q_{k(S)}$ is
$k(S)$-rational if and only if (see
\cite[Prop.~22.9]{elman_karpenko_merkurjev}) $Q_{k(S)}$ has a
$k(S)$-rational point (i.e., $\pi : Q \to S$ has a rational section)
if and only if (by Theorem~\ref{thm:well-known}) $\CliffB_0 \in
\Br(T)$ is trivial.  Now since $S$ is rational, if $Q_{k(S)}$ is
$k(S)$-rational then $Q$ is $k$-rational, by
Lemma~\ref{lem:rationality}.
\end{proof}

As a simple application, Corollary~\ref{cor:well-known_rationality}
gives a proof, in the spirit of quadratic forms, of one of the main
results of \cite{hassett:rational_cubic}:\ if $Z$ is a smooth cubic
fourfold containing a $k$-plane such that the associated quadric
surface bundle $Z' \to \PP^2$ has simple degeneration, then $Z$ is
$k$-rational if the associated Brauer class of $\CliffB_0$ on the
discriminant cover (which is a K3 surface of degree 2) is trivial.

\section{Categorical tools}\label{section:tools}

In this section, we introduce the main categorical tools that we will
apply to geometric examples.  The main new technical tool that we
introduce (in \S\ref{subsec:relative-hpd}) is relative version of
homological projective duality, which we apply in the case of quadric
fibrations.  This is a generalization of Kuznetsov's original theory
\cite{kuznetsov:hpd}.  We note that most definitions in this section
were originally stated over an algebraically closed field of
characteristic zero.  Where necessary, we will show that this
hypothesis is often not needed to prove the basic results we'll
utilize.

\subsection{Semiorthogonal decompositions over a base}
\label{basechange}
Let $k$ be a field and $\cat{T}$ a $k$-linear triangulated category.
Any unadorned product of $k$-schemes will denote a fiber product over
$k$.  In their seminal work \cite{bondal_kapranov:reconstructions},
Bondal and Kapranov define semiorthogonal decompositions for
$k$-linear triangulated categories, in the case where $k$ is
algebraically closed of characteristic zero.  We will briefly recall
the definitions and main results from
\cite[\S2--3]{bondal_kapranov:reconstructions}, arguing that they are
still valid over any field.

Given a full triangulated subcategory $\cat{A}$ of $\cat{T}$, we
denote by $\cat{A}^{\perp}$ its \linedef{right orthogonal}, that is
the full subcategory of $\cat{T}$ whose objects are all the $B$
satisfying $\Hom_{\cat{T}}(A,B)=0$. Similarly we define the
\linedef{left orthogonal} $^\perp\cat{A}$.  A full triangulated
subcategory $\cat{A}$ of $\cat{T}$ is called \linedef{right} (resp.\
\linedef{left}) \linedef{admissible} if the embedding functor admits a
right (resp.\ left) adjoint. The subcategory $\cat{A}$ is
\linedef{admissible}, if it is both right and left admissible.  Notice
that the original definition of admissibility is different, but
equivalent to this one
\cite[Def.~1.2,~Prop.~1.5]{bondal_kapranov:reconstructions}

\begin{proposition}[{\cite[Prop.~1.5]{bondal_kapranov:reconstructions}}]
\label{prop-admissible-and-decompositions}
Let $\cat{A}$ be a right (resp.\ left) admissible subcategory of
$\cat{T}$. Then $\cat{T}$ is generated by $\cat{A}$ and
$\cat{A}^{\perp}$ (resp.\ $\cat{A}$ and $^\perp\cat{A}$) as a
triangulated category.  In particular, if $\cat{A}$ is right
admissible, $\cat{A}^{\perp}$ is left admissible, and if $\cat{A}$ is
left admissible, $^\perp\cat{A}$ is right admissible.
\end{proposition}
\begin{proof}
The proof of \cite[Prop.~1.5]{bondal_kapranov:reconstructions} works
over any field:\ the first statement is the implication $b)
\Rightarrow c)$ in the notation of
\cite{bondal_kapranov:reconstructions}; the second and third
statements follow from the implication $c) \Rightarrow b)$.
\end{proof}

\begin{definition}[{\cite[Def.~2.4]{bondal_orlov:semiorthogonal}}]
\label{def-semiortho}
A \linedef{semiorthogonal decomposition} of $\cat{T}$ is an ordered
sequence of admissible subcategories $\cat{A}_1,
\dotsc, \cat{A}_n$ of $\cat{T}$ such that:
\begin{itemize}
\item for all objects $A_i$ of $\cat{A}_i$ and $A_j$ of $\cat{A}_j$,
$\Hom_{\cat{T}}(A_i,A_j) = 0$ for all $i>j$, and

\item for every object $T$ of $\cat{T}$, there is a chain of morphisms
$0=T_n \to T_{n-1} \to \dotsm \to T_1 \to T_0 = T$ such that the cone
of $T_k \to T_{k-1}$ is an object of $\cat{A}_k$ for all $1 \leq k
\leq n$.
\end{itemize}
Such a decomposition will be written
$$
\cat{T} = \langle \cat{A}_1, \dotsc, \cat{A}_n \rangle.
$$
\end{definition}

Now fix a scheme $Y$ smooth over $k$ and of finite Krull dimension.

\begin{definition}
Let $p : X \to Y$ be a $Y$-scheme. A strictly full subcategory
$\cat{A}$ of $\Db(X)$ is \linedef{$Y$-linear} if for every $E$ in
$\cat{A}$ and $G$ in $\Db(Y)$ we have that $p^*G \otimes E$ is in
$\cat{A}$. A semiorthogonal decomposition of $\Db(X)$ is $Y$-linear if
all the components are $Y$-linear.  If $p' : X' \to Y$ is another
$Y$-scheme, then a functor $F : \Db(X) \to \Db(X')$ is $Y$-linear if
for every $E$ in $\Db(X)$ and $G$ in $\Db(Y)$, there is a bifunctorial
isomorphism $F(p\pullback G \tensor E) \isom p'{}\pullback G \tensor
F(E)$.
\end{definition}

\begin{lemma}
\label{S-lin-deco}
Let $X$ be a $Y$-scheme. If $\cat{A}$ in $\Db(X)$ is a $Y$-linear
admissible subcategory, then both $^\perp \cat{A}$ and $\cat{A}^\perp$
are $Y$-linear.  If $X'$ is another $Y$-scheme and $F : \Db(X) \to
\Db(X')$ is a $Y$-linear functor admitting a right or left adjoint,
then this adjoint is also $Y$-linear.
\end{lemma}
\begin{proof}
The first claim is \cite[Lemma~2.36]{kuznetsov:hyp-sections}.  The
second claim is \cite[Lemma~2.33]{kuznetsov:hyp-sections} for the
right adjoint; for the left adjoint, the proof is the same, but uses
the contravariant version of Yoneda's lemma.
\end{proof}

Fourier--Mukai functors provide a geometric way of producing
admissible subcategories.  Let $X$ and $X'$ be smooth schemes over $k$
and $E$ an object of $\Db(X' \times X)$.  The \linedef{Fourier--Mukai
functor} with kernel $E$ is the functor $\Phi_E: \Db(X') \to \Db(X)$
defined by
$$
\Phi_E (-) = q_* (p^* (-) \otimes E),
$$
where $p$ and $q$ are the projections from $X' \times X$ to $X'$ and
$X$, respectively (recall the we denote derived functors as if they
were underived).  If $X$ and $X'$ are $Y$-schemes, $i : X' \times_Y X
\to X' \times X$ is the natural embedding, and $E$ is an object in
$i\pushforward \Db(X' \times_Y X) \subset \Db(X' \times X)$, then the
Fourier--Mukai functor $\Phi_E : \Db(X') \to \Db(X)$ is $Y$-linear,
see \cite[Lemma~2.35]{kuznetsov:hyp-sections}.

\begin{proposition}
\label{prop:FM-is-admiss}
Let $X$ and $X'$ be smooth schemes over $k$ and $\Phi_E : \Db(X') \to
\Db(X)$ a Fourier--Mukai functor.  If $\supp(E)$ is proper over $X'$
and $X$ then $\Phi_E$ has a right and left adjoint of Fourier--Mukai type.  If in addition, $\Phi_E$ is fully faithful,
then there is a semiorthogonal decomposition
$$
\Db(X) = \langle \cat{A}, \Phi_E(\Db(X')) \rangle,
$$
where $\cat{A}$ is the left orthogonal of $F(\Db(X'))$.
\end{proposition}
\begin{proof}
This is \cite[Lemma~2.4--2.5]{kuznetsov:hyp-sections}, noting that
since $X$ and $X'$ are smooth, the various conditions on the
finiteness of cohomological amplitude are satisfied, ensuring that the
adjoints land in the derived category of bounded complexes and are of
Fourier--Mukai type.  Also, the projectivity hypothesis can be relaxed
to properness, which is all that is required to invoke
Grothendieck--Verdier duality.
\end{proof}

We will mostly use Proposition~\ref{S-lin-deco} in the following
relative situation.  Let $X$ and $X'$ be schemes proper over $Y$ and
smooth over $k$.  Then any $Y$-linear Fourier--Mukai functor $\Phi_E :
\Db(X) \to \Db(X')$ has right and left adjoints.
In fact, there is a stronger result about the existence of adjoints.

\begin{proposition}
\label{prop:with_saturation}
Let $X$ and $X'$ be smooth schemes over $k$.  If $X'$ is proper, then
any exact functor $F : \Db(X') \to \Db(X)$ has a left and right
adjoint.  In particular, if $F$ is fully faithful, then there is a
semiorthogonal decomposition
$$
\Db(X) = \langle \cat{A}, F(\Db(X')) \rangle,
$$
where $\cat{A}$ is the left orthogonal of $F(\Db(X'))$.
\end{proposition}
\begin{proof}
By \cite[Prop.~2.14]{bondal_kapranov:reconstructions} and
\cite[Cor.~3.1.5]{bondal_vdB}, $\Db(Y)$ is saturated, i.e., every
functor (covariant or contravariant) of finite type, from $\Db(Y)$ to
the category of $k$-vector spaces, is representable.  By
\cite[Prop.~2.6,~2.14]{bondal_kapranov:reconstructions}, any full
triangulated subcategory is admissible if it is saturated.
\end{proof}

In particular, this shows that any fully faithful functor $\Db(\Spec
k) \to \Db(X)$ has admissible image.

\begin{corollary}
\label{cor:from-ff-emb-to-semiorth}
Let $X$ be smooth and $X'_1, \ldots, X'_n$ be smooth proper schemes over
$k$, and $F_i: \Db(X'_i) \to \Db(X)$ fully faithful functors, such that
$F_i(\Db(X'_i)) \subset F_j(\Db(X'_j))^{\perp}$ whenever $i > j$.  Then
there is a semiorthogonal decomposition:
$$\Db(X) = \langle \cat{A}, F_1(\Db(X'_1)), \ldots, F_n(\Db(X'_n))\rangle,$$
where $\cat{A}$ is the left orthogonal of the category generated by the $F_i(\Db(X'_i))$. 
\end{corollary}

So far, we have seen that many basic definitions and formal results
concerning semiorthogonal decompositions can be given and hold over
any field $k$.  The main question is, given a smooth
projective variety $X$ over $k$, how to produce a semiorthogonal
decomposition of $\Db(X)$. In the literature, almost all constructions
are described over an algebraically closed field of characteristic
zero, even when this restrictive assumption is not necessary.

We will now present some simple descent results for semiorthogonal
decompositions.
Given a smooth scheme $X$ over $k$, consider its scalar extension
$\ol{X}:= X \times_{\Spec(k)} \Spec(\ol{k})$. We will also denote
$\overline{E}$ the pullback of a coherent sheaf $E$ on $X$, under the
projection morphism $\overline{X} \to X$.  Denote by
$\overline{\cat{A}}\subset \Db(\overline{X})$ the base change of a
triangulated subcategory $\cat{A}\subset \Db(X)$ to $\ol{k}$.  A
theory of base change for triangulated category has been developed in
\cite{kuznetbasechange}.  In order to study how semiorthogonal
decompositions behave under base change, the following is very useful.

\begin{lemma}[{\cite[Lemma~2.12]{orlovequivabel}}]
\label{orlovslemma}
Let $X$ and $X'$ be smooth schemes over $k$.  The Fourier--Mukai
functor $\Phi : \Db(X') \to \Db(X)$ is an equivalence (resp.\ fully
faithful) if and only if the Fourier--Mukai functor $\ol{\Phi}:
\Db(\ol{X}') \to \Db(\ol{X})$ is an equivalence (resp.\ fully
faithful).
\end{lemma}

\begin{lemma}
\label{lem:base-change-semiorth}
Let $X$ be a smooth scheme over $k$.  Suppose that there exist
admissible triangulated categories $\cat{A}_i$ in $\Db(X)$ such that
$\Db(\ol{X}) = \langle \ol{\cat{A}}_1, \ldots, \ol{\cat{A}}_n
\rangle$.  Then $\Db(X) = \langle \cat{A}_1, \ldots, \cat{A}_n
\rangle$.
\end{lemma}
\begin{proof}
The ordered sequence $\cat{A}_1, \ldots, \cat{A}_n$ of subcategories
is semiorthogonal by flat base change. Consider $\langle \cat{A}_1,
\ldots, \cat{A}_n \rangle$ and its orthogonal complement $\cat{A}$,
which are both admissible in $\Db(X)$ by
Proposition~\ref{prop-admissible-and-decompositions}.  By hypothesis,
we have $\ol{\cat{A}}=0$.  This means that for each object $A$ of
$\cat{A}$, we have that $\overline{A}=0$, and hence $A=0$, since
otherwise, $A$ would have a nonzero homology group, which would remain
nonzero over $\overline{k}$ by flat base change. Hence $\cat{A}=0$ and
the ordered sequence of subcategories generate $\Db(X)$.
A similar argument appears in the proof of
\cite[Prop.~2.1]{an-au-ga-za}.
\end{proof}

We observe that Lemma~\ref{lem:base-change-semiorth} holds if we
replace $\ol{k}$ with any field extension of $k$.
A special case of semiorthogonal decompositions is provided by exceptional collections.

\begin{definition}[{\cite[Sect. 2]{bondal:representations}}]
\label{def-except}
An object $E$ of $\cat{T}$ is \linedef{exceptional} if $\Hom_{\cat{T}}
(E,E) = k$ and $\Hom_{\cat{T}}(E,E[i])=0$ for all $i \neq 0$.  An
ordered sequence $(E_1,\dotsc ,E_l)$ of exceptional objects is an
\linedef{exceptional collection} if $\Hom_{\cat{T}}(E_j,E_k[i])=0$ for
all $j>k$ and for all $i \in \ZZ$.
\end{definition}

If $E$ in $\cat{T}$ is an exceptional object, the triangulated
subcategory generated by $E$ (that is, the smallest full triangulated
subcategory of $\cat{T}$ containing $E$) is equivalent to the derived
category of $\Spec k$, see \cite[\S6]{bondal:representations}.  By
Proposition~\ref{prop:with_saturation}, this subcategory is
admissible.  Hence, given an exceptional collection $(E_1,\dotsc
,E_l)$ in the derived category $\Db(X)$ of a smooth scheme,
Corollary~\ref{cor:from-ff-emb-to-semiorth} provides a semiorthogonal
decomposition (see also \cite[\S2]{bondal_orlov:semiorthogonal})
$$
\Db(X) = \langle \cat{A}, E_1, \dotsc, E_l\rangle,
$$
where $E_i$ denotes, by abuse of notation, the category generated by
$E_i$ and $\cat{A}$ is the full triangulated subcategory consisting of
objects $A$ satisfying $\Hom_{\cat{T}}(E_i,A)=0$ for all $1 \leq i
\leq l$.

\subsection{Semiorthogonal decomposition for quadric fibrations}

We now describe the semiorthogonal decomposition of the derived
category of a quadric fibration given by
Kuznetsov~\cite{kuznetsov:quadrics}, generalizing the work of
Kapranov~\cite{kapranov:quadric}, \cite{kapranov:derived}. See
\S\ref{sec:Quadratic_forms} for precise definitions of the notions of
(line bundle-valued) quadratic forms, quadric fibration, and even
Clifford algebras.  Let $Y$ be a smooth scheme and $\pi: Q \to Y$ a
flat quadric fibration of relative dimension $n-2$ associated to a
quadratic form $q: E \to L$ on a locally free $\OO_Y$-module $E$ of
rank $n \geq 2$ (see \S\ref{subsec:quadric_fibrations}).
Denote by $\ko_{Q/Y}(1)$ the restriction to $Q \to Y$ of the relative
ample line bundle $\ko_{\PP(E)/Y}(1)$ and by $\kc_0 = \kc_0(E,q,L)$
the even Clifford algebra, which is a locally free $\OO_Y$-algebra
whose isomorphism class is an invariant of $\pi : Q \to Y$ (see
\S\ref{subsec:Even_Clifford_algebra}).  

\begin{theorem}[{\cite[Thm. 4.2]{kuznetsov:quadrics}}] 
\label{semiortquad}
Let $\pi: Q \to Y$ be a quadric fibration of relative dimension $n-2$
over a scheme $Y$ smooth over a field.  There is a fully faithful
Fourier--Mukai functor $\Phi: \Db(Y,\kc_0) \to \Db(Q)$ and a
semiorthogonal decomposition
$$
\Db(Q) = \langle \Phi \Db(Y,\kc_0), \pi^* \Db(Y)(1), \dotsc, \pi^*
\Db(Y)(n-2) \rangle.
$$
\end{theorem}
\begin{proof}
We observe that this result has been stated in
\cite[Thm.~4.2]{kuznetsov:quadrics} over an algebraically closed field
of characteristic 0.  We will briefly explain why the main technical
results needed in \cite{kuznetsov:quadrics} hold over a general field,
the rest of the proof being very formal.  

First, in Proposition~\ref{prop:compare}, we prove that the even
Clifford algebra $\CliffAlg_0$ and Clifford bimodule $\CliffAlg_1$
defined in \S\ref{subsec:Even_Clifford_algebra} are isomorphic to the
ones defined in \cite[\S3]{kuznetsov:quadrics} (which are only correct
in characteristic $\neq 2$).  Second, the fact that the coordinate
algebra and homogeneous Clifford algebra (see Appendix
\ref{appendix_clifford}) of a flat quadric fibration are Koszul dual
algebras (as stated in \cite[Lemma~3.1]{kuznetsov:quadrics}) holds
over a general base scheme by sheafifying the construction in
\cite[Ch.~2,\S6,~Ex.~4]{poli-posi}.  Third, the construction in
\cite[Lemma~4.5]{kuznetsov:quadrics} of a Fourier--Mukai kernel
inducing a fully faithful functor $\Db(Y,\kc_0) \to \Db(Q)$ is
explicit and carries over to our notion of Clifford algebra and
bimodule.  The proof, in \cite[Lemma~4.4]{kuznetsov:quadrics}, of the
semiorthogonality of the other components is a formal consequence of
Grothendieck--Verdier duality for $\pi$ and the cohomology of the
exact sequence \eqref{eq:sequnce-quadric-fibration}.  Finally, and
most importantly, the resolution of the diagonal of a quadric
fibration stated in \cite[Thm.~2.4]{kuznetsov:quadrics} holds over a
general base scheme by sheafifying the construction in
\cite[\S4.4]{kko}.  This allows the proof, in
\cite[p.~1365]{kuznetsov:quadrics}, that the subcategories generate
$\Db(Q)$, to carry over.
\end{proof}

More generally, if $\pi : Z \to Y$ is flat morphism of smooth schemes,
which is a relative hypersurface in a projective bundle $\PP(E) \to Y$
in such a way that $\omega_{Z/Y} = \ko_{\PP(E)/Y}(-l)|_Z$, then the
following well-known result carries over to the relative setting.

\begin{proposition}\label{decompo-of-relative}
For all $i \in \ZZ$, the functors $\pi^*(-) \tensor \ko_{Z/Y}(i)$ are
fully faithful. For all $j \in \ZZ$, there is a semiorthogonal
decomposition:
$$
\Db(Z) = \langle \cat{A}^j, \pi^* \Db(Y) (j), \dotsc, \pi^* \Db(Y)(j+l-1)
\rangle,
$$
where $\cat{A}^j$ is the orthogonal complement.
\end{proposition}

\begin{proof}
The proof of fully faithfulness and orthogonality goes exactly as the
one for projective bundles in \cite[\S2]{orlov:blowup}.  In fact, one
constructs a functor from $\Db(Y)$ to $\Db(Z)$ as the derived functor
associated to the inverse image functor for coherent sheaves. If $\rk
E =2$, then $\pi$ is finite flat, hence affine and locally free.  If
$\rk E \geq 3$, then $\pi_* \ko_Z = \ko_Y$, and then we have that
$\pi$ is flat hence $\pi^*$ is fully faithful (see Lemma 2.1 of
\cite{orlov:blowup}). Tensoring with a line bundle is an equivalence
of $\Db(Z)$, so composing $\pi^*$ with $- \otimes \ko_{Z/Y (i)}$ is a
fully faithful functor for all $i$ integer.  Finally, the
subcategories $\pi^* \Db(Y) (j), \dotsc, \pi^* \Db(Y)(j+l-1)$ are
admissible and the projection formula (Lemma 2.5, \cite{orlov:blowup})
shows that they are semi-orthogonal.  Then use Corollary
\ref{cor:from-ff-emb-to-semiorth} to conclude.
\end{proof}

\subsection{Relative homological projective duality for quadric fibrations}
\label{subsec:relative-hpd}

Homological projective duality was introduced by Kuznetsov
\cite{kuznetsov:hpd} in order to study derived categories of
hyperplane sections (see also \cite{kuznetsov:hyp-sections}).  In
particular, it can be applied to relate the derived category of a
complete intersection of quadrics inside the projective space to the
derived category of the associated linear system of quadrics
\cite[\S5]{kuznetsov:quadrics}.  In what follows, we spell out this
latter construction over a smooth base scheme $Y$, providing a
straightforward generalization of Kuznetsov's construction to quadric
fibrations and their intersections.

\begin{remark}
Homological projective duality involves \linedef{noncommutative
schemes}, by which we mean (following Kuznetsov
\cite[\S2.1]{kuznetsov:quadrics}) a scheme $N$ together with an
$\ko_{N}$-algebra $\ka$, coherent as an $\ko_N$-module, and .  Morphisms are
defined accordingly.  By definition, a noncommutative scheme $(N,\ka)$
has $\Coh(N, \ka)$ as category of coherent sheaves and $\Db(N,\ka)$ as
bounded derived category. Following Bondal--Orlov
\cite[\S5]{bondal_orlov:ICM2002}, a noncommutative resolution of
singularities $(N,\ka)$ of a possibly singular scheme $N$ is a
torsion-free $\ko_N$-algebra $\ka$ of finite rank such that
$\Coh(N,\ka)$ has finite homological dimension (i.e., is smooth in the
noncommutative sense).  For $Y$ any scheme, a flat $Y$-noncommutative
scheme (or a noncommutative scheme over a base $Y$) is a pair
$(N,\ka)$ with $\pi : N \to Y$ a $Y$-scheme and $\pi\pushforward\ka$ a
flat $\ko_Y$-module.  Also, $(N,\ka)$ is projective if $N$ is.  A
Fourier--Mukai functor $\Phi : \Db(N,\ka) \to \Db(N',\ka')$ is an
integral transform whose kernel is an object $E$ in $\Db(N \times N',
\ka\op \boxtimes \ka')$.
\end{remark}
\begin{remark*}
In the case when $\ka$ and $\ka'$ are Azumaya algebras, this coincides
with the notion of ``twisted'' Fourier--Mukai functors developed in
\cite{canonaco-stellari}.
\end{remark*}

We now recast the basic notions of homological projective duality from
\cite[Def.~4.1]{kuznetsov:hpd} in a relative setting.  Let $M \to Y$
be a flat projective morphism of schemes smooth over a field $k$ and
$\ko_{M/Y}(1)$ a relatively very ample line bundle on $M$.

\begin{definition}
A \linedef{Lefschetz decomposition} of $\Db(M)$ with respect to
$\ko_{M/Y}(1)$ is a semiorthogonal decomposition

\begin{equation}\label{eq:def-of-lefschetz}
\Db(M) = \langle \cat{A}_0, \cat{A}_1(1), \dotsc, \cat{A}_{i-1}(i-1) \rangle,
\end{equation}
with
$$
0 \subset \cat{A}_{i-1} \subset \dotsc \subset \cat{A}_0.
$$
\end{definition}

Let $\PP(V) \to Y$ be a projective bundle and $f:M \ra \PP(V)$ be a
$Y$-morphism such that $f^*\ko_{\PP(V)/Y}(1)\cong \ko_{M/Y}(1)$.  Let
$\cali M \subset M \times_Y \PP(V^{\vee})$ be the universal hyperplane
section
$$
\cali M :=\{ (p,H)\in M \times_Y \PP(V^{\vee}) \; : \; p\
\mathrm{belongs\ to\ } H\}.
$$
We denote by ${\mathrm{pr}_1}: \cali M \to M$ the restriction of the
projection onto the first factor.
\begin{definition}[{\cite[Def 6.1]{kuznetsov:hpd}}]
\label{def:hpdual}
Let $V$ be a vector bundle over a smooth scheme $Y$ over a field $k$.
By a \linedef{homological projective duality pair} (or \linedef{HP
dual pair}) over $Y$ we mean the data of scheme $M$ smooth over $k$
and flat over $Y$, a noncommutative scheme $(N,\ka)$ flat over $Y$, a
morphism $f : M \to \PP(V)$, a morphism $g : N \to \PP(V\dual)$, a
Lefschetz decomposition of $\Db(M)$ with respect to $\OO_{M/Y} =
f\pullback\OO_{\PP(V)/Y}(1)$, and
an object $E$ in $\Db({\cali M} \times_{\PP(V\dual)} N,
\OO_{\cali M} \boxtimes \ka^{\mathrm{op}})$, such that the
Fourier--Mukai functor $\Phi_{E} : \Db(N,\ka) \to \Db({\cali M})$ is
fully faithful and gives the semiorthogonal decomposition
\begin{equation}\label{eq-deco-of-univ-hyp}
\Db({\cali M}) = \langle \Phi_{E} (\Db(N,\ka)), \cat{A}_1(1) \boxtimes \Db(\PP(V\dual)), \dotsc,
\cat{A}_{i-1}(i-1) \boxtimes \Db(\PP(V\dual)) \rangle.
\end{equation}
\end{definition}

Let the rank of $V$ be $n$.  For a vector subbundle $L \subset V\dual$
with locally free quotient, denote its orthogonal by $L^{\perp}
\subset V$, and consider the following relative linear sections
$$
M_L = M \times_{\PP(V)} \PP(L^{\perp}), \quad N_L = N \times_{\PP(V\dual)} \PP(L).
$$
of $M$ and $N$.  If $\ka$ is an $\ko_N$-algebra, then denote by $\ka_L
= \ka \boxtimes_{\ko_{\PP(V\dual)/Y}} \ko_{\PP(L)/Y}$. 

We now prove that the main result of the theory of homological
projective duality holds in the relative setting.

\begin{theorem}
\label{thm:HPD}
Let $V$ be a vector bundle on a smooth scheme $Y$ over $k$.  Let $M$
and $(N,\ka)$ be an HP dual pair over $Y$. If the Lefschetz
decomposition is $Y$-linear then:

\begin{itemize}
\item[(i)] $N$ is smooth over $k$
and admits a \linedef{dual} Lefschetz decomposition
$$
\Db(N,\ka) = \langle \cat{B}_{j-1}(1-j), \dotsc, \cat{B}_1 (-1),
\cat{B}_0 \rangle, \quad\quad 0 \subset \cat{B}_{j-1} \subset \dotsm
\subset \cat{B}_1 \subset \cat{B}_0
$$
\item[(ii)] for any vector subbundle $L \subset V\dual$ of rank $r$
with locally free quotient such that
$$
\dim\, M_L = \dim\, M - r, \quad \mbox{and}\quad \dim\, N_L = \dim\, N + r - n,
$$ 
there exists a triangulated category $\cat{C}_L$ and
semiorthogonal decompositions:
\begin{align*}
\Db(M_L) & = \langle \cat{C}_L, \cat{A}_r(1), \dotsc, \cat{A}_{i-1}(i-r) \rangle,\\
\Db(N_L,\ka_L) & = \langle \cat{B}_{j-1} (N-r-j), \dotsc, \cat{B}_{N-r}(-1), \cat{C}_L \rangle.
\end{align*}
\end{itemize}
\end{theorem}

\begin{proof}
The proof goes along the lines of \cite[\S5--6]{kuznetsov:hpd}.
Notably, \cite[Rem.~6.4]{kuznetsov:hpd}, gives three conditions (which
we will denote by (1), (2) and (3) in this proof) that are sufficient
for homological projective duality to hold for a given HP dual
pair. The description in \cite{kuznetsov:hpd} is very precise, so for
the sake of readability, we refrain from giving too many details here
and we address the interested reader to the original paper.  Roughly,
the first of these conditions requires the existence of a fully
faithful Fourier--Mukai functor $\Phi_{E}:\Db(N,\ka) \to \Db(\cali M)$
satisfying the orthogonality condition as in
\eqref{eq-deco-of-univ-hyp}. Hence, Condition (1) is part of the
definition of HP dual pair.

Condition (2) and (3) require roughly that the functor $\Phi_{E}^*
\circ {\mathrm{pr}_1}^* : \Db(M) \to \Db(N,\ka)$ is well-behaved on
all the \it primitive \rm components of the Lefschetz
decomposition. Here $\Phi_{E}^*$ denotes the left adjoint to
$\Phi_{E}$, while the $i$-th primitive component of a Lefschetz
decomposition is the orthogonal complement of $\cat{A}_i$ in
$\cat{A}_{i+1}$. A full and accurate treatment is given in \cite[\S
4]{kuznetsov:hpd}, which requires an immense amount of notation that
we refrain from introducing here.

Conditions (2) and (3) are proved in the case where $Y$ is a point in
\cite[Prop.~5.7,~Cor.~5.8]{kuznetsov:hpd} and
\cite[Lemma~5.9,~Prop.~5.10]{kuznetsov:hpd} respectively.  As remarked
upon in \cite[Rem.~5.13]{kuznetsov:hpd}, these statements from
\cite[\S5]{kuznetsov:hpd} hold true also under the wider hypothesis
assumed in our statement (even more generally, without projectivity
assumptions on $Y$).  Indeed, $Y$-linearity allows us to work
relatively over the scheme $Y$, replacing $\Hom_X$ by $p_*
\underline{\Hom}_X$ in all the homological arguments presented in
\cite[\S 5]{kuznetsov:hpd}.  As the three conditions of
\cite[Rem.~6.4]{kuznetsov:hpd} are satisfied, the Theorem is proved.
\end{proof}

\medskip 

Next, we will apply relative homological projective duality to the
case of flat quadric fibrations.  Let $E$ be a vector bundle of rank
$n$ on a smooth
scheme $Y$. Consider the projective bundle $p : M = \PP(E) \to Y$, the
relative ample line bundle $\ko_{M/Y}(1)$, and the semiorthogonal
decomposition (see \cite[Thm.~2.6]{orlov:blowup})

\begin{equation}
\label{eq:lefsch-of-Proj}
\Db(M) = \langle p^* \Db(Y) (-1), p^*\Db(Y), \dotsc, p^* \Db(Y)(n-2)\rangle.
\end{equation}
Let us denote by $m = \lfloor n/2 \rfloor$ and put
$$
\cat{A}_0= \cat{A}_1 = \dotsc = \cat{A}_{m-1} = \langle p^*\Db(Y)(-1),
p^*\Db(Y) \rangle,
$$
$$
\cat{A}_{m} = 
\begin{cases} 
\langle p^*\Db(Y)(-1), p^*\Db(Y) \rangle & \text{if $n$ is even}\\
\langle p^*\Db(Y)(-1) \rangle & \text{if $n$ is odd.}
\end{cases}
$$
Then the decomposition \eqref{eq:lefsch-of-Proj} is a Lefschetz
decomposition
\begin{equation}\label{eq:lefsch-explicit}
\Db(M) = \langle \cat{A}_0, \cat{A}_1(2), \dotsc, \cat{A}_{m}(2m) \rangle,
\end{equation}
with respect to the relative double Veronese embedding $f: M = \PP(E)
\to \PP(S_2 E) = : \PP(V)$, as $f^* \ko_{\PP(S^2 E)/Y} (1) \isom
\ko_{\PP(E)/Y}(2)$.  Here, we use the submodule $S_2E$ of symmetric tensors for defining the relative Veronese embedding
(working even in characteristic 2), given our convention for defining
projective bundles, see \S\ref{sec:Quadratic_forms} for more details.
Recall the canonical isomorphism $S^2E\dual \isom (S_2E)\dual$.

\begin{definition}
Let ${\cali Q} \subset \PP(E) \times_Y \PP((S_2E)\dual) = \PP(E)
\times_Y \PP(S^2E\dual)$ be the universal hyperplane section with
respect to $f$, then we will refer to the projection
$$
\pi : {\cali Q} \to \PP(S^2E\dual),
$$
as the \linedef{universal relative quadric fibration} in $\PP(E)$.
\end{definition}

Roughly speaking, this is equivalent to the fact that the double
Veronese embedding $f$ is defined by the full linear systems of
quadrics on $\PP(E)$, hence the universal hyperplane section carries
the universal family of relative quadrics in $\PP(E)$. Indeed, if a
section $s : Y \to \PP(S^2E\dual)$ of $\pi$ corresponds to a line
subbundle $L\dual \subset S^2E\dual \isom (S_2E)\dual$, then the
pullback $s^*{\cali Q} \to Y$ is the quadric fibration associated to
the quadratic form defined by Lemma~\ref{lem:quad}.  Thus $\pi: {\cali
Q} \to \PP(S^2 E\dual)$ is a flat quadric fibration.  Let $\kc_0$ be
its even Clifford algebra.
\begin{theorem}\label{thm:relative-hpd-fot-quadrics}
Let $Y$ be a smooth projective scheme and $E$ a vector bundle.  Then
$\PP(E)$ and the noncommutative variety $(\PP(S^2 E\dual),\kc_0)$ form
an HP dual pair over $Y$ with respect to the Lefschetz decomposition
\eqref{eq:lefsch-explicit}.
\end{theorem}
\begin{proof}
The pair $\PP(E)$, $(\PP(S^2 E\dual), \kc_0)$ forms a HP dual pair
thanks to the semiorthogonal decomposition of the universal quadric
fibration.  
\end{proof}

Now we describe some consequences of Theorem \ref{thm:HPD} in the case
of flat quadric fibrations.  Let $(E_i,q_i,L)$ be a finite set of
primitive generically (semi)regular quadratic forms.  Denote by
$L\dual \to S^2 E\dual$ the $\OO_Y$-submodule generated by the line
subbundles $L_i\dual$.  Then the linear section $M_{L\dual} = \PP(E)
\times_{\PP(S_2E)} \PP(L\dual{}^{\perp})$ (which we denote by $X$) is
a relative intersection of the quadric fibrations $Q_i \to Y$ in
$\PP(E)$. Indeed, the projection map $\pi : X \to Y$ has fibers the
intersection of the fibers of $Q_i \to Y$ in the projective space
given by the fibers of $\PP(E)$.  On the other hand, the linear
section $N_{L\dual} = \PP(S^2E\dual) \times_{\PP(S^2E\dual)}
\PP(L\dual)$ (which we denote by $S$) is precisely $\PP(L\dual)
\subset \PP(S^2 E\dual)$.  Then the restriction $\kc_0|_{\PP(L\dual)}
= \kc_0 \boxtimes_{\OO_{\PP(S^2E\dual)}}\ko_{\PP(L\dual)}$ (which we
shamelessly denote by $\kc_0$) of the even Clifford algebra of ${\cali
Q} \to \PP(S^2E\dual)$ to $\PP(L\dual)$ is then isomorphic to the even
Clifford algebra of the corresponding linear span (see
Definition~\ref{linearspan}) quadric fibration $Q \to \PP(L\dual)$
associated to the $Q_i \to Y$.
We assume that $L\dual \isom \oplus_i L_i\dual$ and that this relative
intersection is complete. Notice that we have $\omega_{X/Y} =
\ko_{X/Y}(2m-n)$.  We will record the following application of
Theorems~\ref{thm:HPD} and \ref{thm:relative-hpd-fot-quadrics} for
future use.

\begin{theorem}[HP duality for quadric fibration intersections]
\label{hpd}
Let $Y$ be a smooth scheme, $Q \to S$ a linear span of $m$ quadric
fibrations of relative dimension $n-2$ over $Y$, and $X \to Y$ their
relative complete intersection.  Let $\kc_0$ be the even Clifford
algebra of $Q \to S$.
\begin{enumerate}
\item \label{hpd_Fano} 
If $2m < n$, then the fibers of $X \to Y$ are Fano
and relative homological projective duality yields
\begin{equation*}
\Db(X) = \langle \Db(S,\kc_0), \pi^* \Db(Y)(1) \dotsc \pi^* \Db(Y)(n-2m)\rangle.
\end{equation*}

\item \label{hpd_CY.2}
If $2m=n$ then the fibers of $X \to Y$ are generically
Calabi--Yau and relative homological projective duality yields
\begin{equation*}
\Db(X) \simeq \Db(S,\kc_0).
\end{equation*}

\item \label{hpd_CY.3}
 If $2m > n$, then the fibers of $X \to Y$ are generically of
general type and there exists a fully faithful functor $\Db(X) \to
\Db(S,\kc_0)$ with explicitly describable orthogonal complement.
\end{enumerate}

\end{theorem}

\begin{remark*}
We note that the claims on the Kodaira dimension of the fibers
contained in \eqref{hpd_CY.2} and \eqref{hpd_CY.3} hold only for the
generic fibers since it can change over a closed subscheme of the
base. Just think about a family of plane cubics with a singular
central fiber.
\end{remark*}

\subsection{Semiorthogonal decompositions, representability, and rationality}
\label{subsec:categor-repre}

Since the pioneering work of Bondal--Orlov
\cite{bondal_orlov:semiorthogonal}, semiorthogonal decompositions have
proved themselves to be a very useful tool in studying Fano varieties
and their birational properties, see \cite{kuznetsov:v14} and
\cite{kuznetsov:v12} for example.  In particular, semiorthogonal
decompositions should encode, in a categorical way, the geometric
information contained in the intermediate jacobian of a Fano
threefold, see \cite{bolognesi_bernardara:representability}.  Unlike
the theory of intermediate jacobians, this categorical approach can be
extended to higher dimensions.  We now introduce the notion of
categorical representability, which places our results in a general
conjectural framework for the study of rationality via semiorthogonal
decompositions.

\begin{definition}[{\cite[Def.~2.3]{bolognesi_bernardara:representability}}]
\label{def-rep-for-cat}
A triangulated category $\cat{T}$ is \linedef{representable} in
dimension $m$ if it admits a semiorthogonal decomposition
$$
\cat{T} = \langle \cat{A}_1, \dotsc, \cat{A}_l \rangle,
$$
such that for each $1 \leq i \leq l$, there exists a smooth projective
variety $M_i$ with $\dim\, M_i \leq m$, such that $\cat{A}_i$ is
equivalent to an admissible subcategory of $\Db(M_i)$.
\end{definition}

\begin{remark}
\label{connectd}
Notice that in the definition, we can assume the categories
$\cat{A}_i$ to be indecomposable and also the varieties $M_i$ to be
connected. Indeed, the derived category $\Db(M)$ of a scheme $M$ is
indecomposable if and only if $M$ is connected (see
\cite[Example 3.2]{bridgeland:triangulated}).
\end{remark}

\begin{definition}[{\cite[Def. 2.4]{bolognesi_bernardara:representability}}]
\label{repre}
Let $X$ be a smooth projective variety of dimension $n$. We say that
$X$ is \linedef{categorically representable} in dimension $m$ (or
equivalently in codimension $n-m$) if $\Db(X)$ is representable in
dimension $m$.
\end{definition}

In the sequel, we develop general classes of threefolds and fourfolds
with a fibration in intersection of quadrics over $\PP^1$, where
categorical representability is intimately related to rationality.

\section{Genus 1 fibrations}\label{elliptic}

In this section, we introduce a warm-up example:\ a genus 1 fibration
$X \to Y$ (by \linedef{genus 1 fibration} we mean a proper flat
surjective map whose generic fibers are smooth genus 1 curves) over a
smooth projective variety $Y$, obtained as the generic (see
Definition~\ref{def:generic-fibration}) relative complete intersection
of two quadric surface fibrations.

\subsection{Clifford algebras of genus 1 fibrations}

Let $Q_1 \to Y$ and $Q_2 \to Y$ be a generic pair of quadric surface
fibrations defined by quadratic forms $(E,q_1,L_1)$ and $(E,q_2,L_2)$
over a scheme $Y$.  Consider the $\PP^1$-bundle $S=\PP(L_1\dual \oplus
L_2\dual) \to Y$, the linear span quadric fibration $Q \to S$ (see
Definition \ref{linearspan}), and its associated even Clifford algebra
$\kc_0$.  Then the generic relative complete intersection $X \to Y$,
of $Q_1\to Y$ and $Q_2\to Y$ in $\PP(E)$, is a genus 1 fibration and
$\CliffAlg_0$ gives rise to an Azumaya algebra $\CliffB_0$ on the
discriminant cover $T \to S$ (see Proposition
\ref{prop:simple_azumaya}), whose Brauer class we denote by $\beta \in
\Br(T)$.

By the genericity hypothesis, the discriminant divisor $D$ of $Q \to
S$ is smooth. It intersects each fiber of the ruled surface $S \to Y$
in a closed point of degree 4.  We note that the composite map $T \to
S \to Y$ is thus also a genus 1 fibration, since every generic fiber
is a double cover of the projective line branched at 4 points.  We
point out that for a geometric point $y$ of $Y$ of characteristic
zero, the fiber $X_y$ is a principal homogeneous space under the
jacobian $J(T_y)$, see \cite[Thm.~4.8]{reid:thesis}.

\begin{theorem}\label{thm:elliptic-fibr}
Let $X \to Y$ be a genus 1 fibration arising as the generic relative
complete intersection of two quadric surface fibrations.  Then there
is an equivalence $\Db(X) \simeq \Db(T, \beta)$.  Moreover, $\beta \in
\Br(T)$ is trivial if and only if the genus 1 fibration $X \to Y$
admits a section.
\end{theorem}
\begin{proof}
Relative homological projective duality (Theorem
\ref{hpd}\eqref{hpd_CY.2}) yields an equivalence $\Db(X) \simeq
\Db(S,\kc_0)$.  Since $T \to S$ is affine, we have an equivalence
$\Db(S,\kc_0) \simeq \Db(T,\beta)$, since $\beta$ is the class in the
Brauer group corresponding to the Azumaya algebra $\CliffB_0$.

As for the second assertion, the genus 1 fibration $X \to Y$ admits a
section if and only if it admits a rational section if and only if (by
the Amer--Brumer Theorem~\ref{amer-brumer}) $Q \to S$ has a rational
section if and only if (by Lemma~\ref{thm:well-known}) $\beta \in
\Br(T)$ is trivial.
\end{proof}

On the other hand---as we shall also detail for a higher dimensional
case in \S\ref{section:4fold}---the genus 1 fibration $T \to Y$ is the
relative moduli space of spinor bundles on $X \to Y$ (actually,
restrictions to $X$ of spinor bundles on the quadrics of the pencil,
see \cite[Ch.~2]{bondal_orlov:semiorthogonal} for details). The
obstruction to the existence of a universal class on $X \times_Y T$ is
given by a Brauer class $\omega$.  Combining results from
\cite{bondal_orlov:semiorthogonal} and \cite{caldararu:thesis}, one
gets an equivalence $\Db(T, \omega) \simeq \Db(X)$ by the
Fourier--Mukai functor whose kernel is the $\omega$-twisted universal
sheaf on $X \times_Y T$.

Moreover, we note that a similar result was already proved by
C\u{a}ld\u{a}raru \cite[\S4]{caldararu:thesis}.  Indeed, suppose that
$X \to Y$ is a smooth elliptic fibration and let $J \to Y$ be the
relative jacobian. There is an element $\alpha$ in $\Br(J)$ giving the
obstruction to $J \to Y$ to be a fine moduli space for degree zero
relative line bundles on $X \to Y$.  By considering the
$\alpha$-twisted Poincar\'e line bundle on $X \times_Y J$, one gets an
equivalence $\Db(X) \simeq \Db(J,\alpha)$.  That the triviality of
$\alpha$ is equivalent to the existence of a section of $X \to Y$ is
detailed in \cite[\S4]{caldararu:thesis} and arises from the fact that
the image of $\alpha$ under the canonical map $\Br(J) \to \Sha(J/Y)$
coincides with the class of $X\to Y$.

\begin{question}[See Question~\ref{question4}]
\label{question_elliptic}
Let $X \to Y$ be a smooth genus 1 fibration arising as the generic
relative complete intersection of two quadric surface fibrations.  Let
$J \to Y$ be the jacobian fibration and $T \to S \to Y$ the
discriminant cover of the linear span.
Are $J \to Y$ and $T \to Y$ related (or at least their Brauer groups)?
If so, are the classes $\alpha$, $\beta$, and $\omega$ related (or
possibly coincide)?
\end{question}

\section{Del Pezzo fibrations of degree four}\label{section:delpezzo}

Del Pezzo fibrations over $\PP^1$ form a very important class of
varieties in the classification of smooth projective threefolds with
negative Kodaira dimension.  Notably, they form one of the three
classes to which any such threefold can be geometrically birationally
equivalent. 
In this section, whenever we use the results from
\cite{alekseev:dP4}, we need to work over $\CC$. We will state when a broader generality for the field of definition is possible.

\subsection{Minimal Del Pezzo fibrations}

\begin{definition}
A 3-fold $X$ is a minimal Del Pezzo fibration if there exists a proper flat morphism $\delta: X
\to C$ to a smooth projective curve $C$ whose generic fiber is a Del
Pezzo surface and $\Pic(X) =\delta^* \Pic (C) \oplus \ZZ$.
\end{definition}

In this section, we employ our techniques to study quartic Del Pezzo
fibrations $X \to C$ over a curve.  Over an algebraically closed
field, such fibrations can always be realized as relative complete
intersections of two three-dimensional quadric fibrations over $C$.
Our study will be mainly devoted to the problem of rationality in the
case where $C = \PP^1$.  In fact, the rationality of such a Del Pezzo
fibration only depends on the Euler characteristic, as shown by
Alexeev \cite{alekseev:dP4} and Shramov \cite{shramov:dp4}.  Starting
from the classical point of view of \cite{alekseev:dP4} and the
categorical tools developed in
\cite{bolognesi_bernardara:conic_bundles} and
\cite{bolognesi_bernardara:representability}, we lead up to a proof of
Theorem~\ref{thm:1}, establishing a purely categorical criterion of
rationality of quartic Del Pezzo fibrations over $\PP^1$.

Let $Y$ be a scheme over a field $k$ and $\pi: X \to Y$ be a quartic
Del Pezzo fibration with $X$ be smooth over $k$. Then there exists a
vector bundle $E$ of rank 5 such that $X \subset \PP(E)$ is the
complete intersection of two quadric fibrations $Q_1 \to Y$ and $Q_2
\to Y$ of relative dimension 3 (see \cite{shramov:dp4}).  We consider
here the generic case, according to
Definition~\ref{def:generic-fibration}. We get a $\PP^1$-bundle over
$S \to Y$, and a fibration $Q \to Y$ of relative dimension 3.
Moreover, $\omega_{X/Y} = \ko_{X/Y}(-1)$ and thus, by
Proposition~\ref{decompo-of-relative}, for all integers $j$, we have a
semiorthogonal decomposition
\begin{equation}\label{deco-delpe}
\Db(X) = \langle \cat{A}^j_X, \Db(Y)(j) \rangle.
\end{equation}
We observe moreover that $\cat{A}^j_X$ does not contain $\Db(Y)(l)$ for any $l\neq j$. To simplify notation let us denote $\cat{A}_X = \cat{A}_X^1$. 
Letting $\kc_0$ be the even Clifford algebra of $Q \to S$, then
relative homological projective duality (Theorem
\ref{hpd}\eqref{hpd_Fano}) provides an equivalence $\cat{A}_X \simeq
\Db(S,\kc_0)$.

\subsection{Reduction by hyperbolic splitting vs.\ Alexeev's
construction}
\label{subsec:Alexeev}

In the case where $Y=C$ is a smooth projective complex curve, Alexeev
\cite{alekseev:dP4} has shown that $X$ is birational to a conic bundle
$Q''$ over a ruled surface over $S' \to C$. Part of this birationality
result is based on the construction of a particular smooth section of
$X \to C$. Hence, by the (easy direction of the) Amer--Brumer theorem
(Theorem~\ref{amer-brumer}), there exists also a smooth section of $Q
\to S$, along which we can perform quadric reduction to obtain a conic
bundle $Q' \to S$.  The mere existence of a smooth section is
guaranteed by Lemma~\ref{lem:C_2}.  Even though Alexeev's construction
is not given in terms of quadric reduction, but by explicit birational
transformations, we will show that $S' = S$, that the conic bundle
$Q'' \to S$ is birational to the $Q' \to S$, and that their associated
Clifford algebras are Morita equivalent.

\begin{definition}
A {\em standard conic bundle} $\pi: Z \to T$ over a surface $T$ is a
proper flat surjective morphism whose geometric fibers are isomorphic
to plane conics, such that for any irreducible curve $B \subset T$ the
surface $\pi^{-1}(B)$ is irreducible (this second condition is also
called relative minimality).
\end{definition}
 
The discriminant locus of a standard conic bundle is a curve $D
\subset T$, which can be possibly empty, with at most double points
\cite[Prop.~1.2]{beauvillejaco}.  The fiber of $\pi$ over a smooth
point of $D$ is the union of two lines intersecting in a single point,
while the fiber over a node is a double line.  Recall that any conic
bundle is birationally equivalent to a standard one via elementary
transformations \cite{sarkisov-bira}, and that, if $T$ is a rational
surface, then the discriminant double cover $\tilde{D} \to D$
identifies the isomorphism class of the associated Clifford algebra
$\kc_0$ \cite[Lemma~3.2]{bolognesi_bernardara:conic_bundles}.

Recall that $\pi:X \to C$ is the complete intersection of two three
dimensional quadric fibrations $Q_1 \to C$ and $Q_2 \to C$ given by
line bundle valued quadratic forms on a rank 5 vector bundle $E$ over
$C$. Moreover, $X$ is embedded in $\PP(E)$ via its relative
anticanonical system $-\omega_{X/C}$, that is $\omega_{X/C} =
\ko_{X/C}(-1) = \ko_{\PP(E)/C}(-1)|_{X}$. Hence, the anticanonical
bundle corresponds to the relative hyperplane section of $X$ over $C$.

One of the main ideas of \cite{alekseev:dP4} is summarized in the
following.

\begin{proposition}\label{alexbundle}
There exists a standard conic bundle $Q''$, over a ruled surface $S'\to C$, which is birational to $X$.
\end{proposition}
\begin{proof}
We will only give a sketch of the construction; for complete proofs
see \cite{alekseev:dP4}. Let $s$ be a smooth section of $\pi$, and
consider the blow-up $\eta:\widetilde{X} \to X$ along $s$, and the
composition $\tilde{\pi}:= \pi \circ \eta:\widetilde{X} \to C$, which
is a flat two-dimensional fibration.  The exceptional divisor of the
blow-up is a ruled surface $\widetilde{F} \to C$.  The linear system
$|-\omega_{X/C}-s|$ gives the relative projection of $X$ off $s$ over
$C$.  This gives a rational map $X \dashrightarrow \PP(\kf)$, where
$\PP(\kf) \to C$ is a projective bundle of relative dimension
3. Resolving this rational map by blowing up $s$ gives a closed
embedding $\widetilde{X} \hookrightarrow \PP(\kf)$, corresponding to
the linear system $| -\omega_{X/C}-\widetilde{F}| =
|-\omega_{\tilde{X}/C}|$, so that $\tilde{X}$ is relatively
anticanonically embedded in $\PP(\kf)$.

For any point $c$ of $C$, considering the projective anticanonical
embedding $X_c \subset \PP(E_c) \simeq \PP^4$, we can think of $X_c$ as the blow up
of a projective plane along 5 points in general position.
Alexeev shows that the smooth section $s$ can be chosen such that, over
all but a finite number of points of $C$, the point cut out by $S$ on the fiber $X_c$
is in general position with respect to those 5 points \cite{alekseev:dP4}. This means that
the blow-up of $X_c$ along that point is a degree 3 del Pezzo surface. It follows that all but a finite number of 
the fibers of $\widetilde{X} \to C$ are cubic surfaces anticanonically embedded in the corresponding fiber of $\PP(\kf) \to C$. 

Consider the linear system $|L_2|=|-\omega_{\widetilde{X}/C} - \wt{F}| = |\eta^* (-\omega_{X/C}) - 2\wt{F}|$ on $\widetilde{X}$.
On each fiber, $\widetilde{F}$ is a line, while
for all but a finite number of points of $C$, the anticanonical class corresponds to the hyperplane section.
So the linear system $|L_2|$ forms, fiber-wise in all but a finite number of fibers, a pencil of residual conics.
Working by birational transformations along the fibers
which are not cubic surfaces, Alexeev finally describes a standard conic bundle $Q''$, over a ruled surface
$S' \to C$. Denote by $|L_1|:=|-\omega_{X/C}-2s|$, and notice that
$\eta^\ast L_1 = L_2$, so that $S'=\PP(\pi_*(-\omega_{X/C}-2s))$. The conic
bundle $Q'' \to S'$ is standard and birational to $X$ \cite{alekseev:dP4}.

On the other hand, consider the quadric fibration $Q \to S$, the
linear span of the quadric pencils $Q_1 \to C$ and $Q_2 \to C$.  In
particular, $S \to C$ is a $\PP^1$-bundle.  Since $s$ is a smooth
section of $X$, it provides a smooth section of $Q$ and hence we get,
by hyperbolic splitting, a conic bundle $Q' \to S$. By
\cite[Thm.~3.2]{colliot-thelene_sansuc_swinnerton-dyer:quadrics_I},
$X$ is birational to $Q'$, so $Q'$ is birational to $Q''$. 
\end{proof}

\begin{theorem}\label{conicbundles}
The two ruled surfaces $S$ and $S'$ are isomorphic, and the two conic bundles $Q'$ and $Q''$ are isomorphic over the compement of a finite number of fibers.
\end{theorem}

\begin{proof}
Consider once again the linear system $|L_2|
= |-\omega_{X/C} - 2s|$ on $X$. Since $\omega_{X/C}=\ko_{X/C}(-1)$,
this linear system is the pencil of vertical hyperplanes in $\PP(E)$
which are tangent to $X$ at $s$.  Let us denote by $Q_{\lambda} \to C$
a quadric fibration in the quadric pencil $Q \to S$. That is, we are
fixing a section $\lambda$ of the $\PP^1$-bundle $S \to C$ and we are
considering the corresponding 3-dimensional quadric fibration.  In
particular, for each $\lambda$, $s$ is contained in $Q_{\lambda}$.  If
we let $\lambda$ vary in the pencil and take the relative tangent
hyperplanes to $s$ in each $Q_{\lambda}$ in $Q$, these give all the
relative tangent hyperplanes to $X$ at $s$. Hence they cut out on $X$
the full linear system $|L_2|$.

In order to prove our claim, let us first explain the relation between the relative pencil of quadrics $Q \to S$ in $\PP(E) \times \PP^1$
and relative conics, residual to $\wt{F}$, that compose the linear system $|L_1|$ on $\widetilde{X}$. Take again a fixed 3-dimensional quadric fibration $Q_{\lambda} \to C$ and the tangent hyperplane $T_{\lambda}$ to it at $s$. The
surface $T_{\lambda}\cap X$ (which is an element of $|L_2|$) has a fibration in quartic curves $G_{\lambda} \to C$, induced by restriction under the inclusion $G_\lambda\subset Q_{\lambda}$. Now, $s$ is contained in $G_\lambda$ and all the quartic curves are singular at the section $s$. On the other hand, $T_{\lambda}$
cuts out a family of quadric cone surfaces $\wt{Q}_{\lambda} \to C$ from
$Q_{\lambda} \to C$ that contains $G_{\lambda}$. We denote $\wt{G}_{\lambda} \to C$ the family of bases of the cone.

Consider now the rational projection $\PP(E) \dashrightarrow \PP(\kf)$
from the center $s$.  We recognize the rational map that sends $X$
onto $\widetilde{X}$, the rational inverse of $\eta$. It sends
$T_{\lambda}$ to a family of 2-planes $\PP^2_{\lambda} \to C$
containing $\wt{F}$. It sends $Q_{\lambda}$ to its base conic
$\wt{G}_{\lambda}\subset \PP^2_{\lambda}$. It sends $G_{\lambda}$
birationally onto $\wt{G}_{\lambda}$ (note in fact that we are
projecting fiber-wise from the double point of quartic curve, hence
the image has fiber-wise degree two). We note that $\wt{G}_{\lambda}$
is contained in $\widetilde{X}$ and that it is a divisor of the linear
system of conics $|L_1|$. If $\widetilde{X}_c$ is a cubic surface in
$\PP(\kf_c) = \PP^3$, we get the residual conics with respect to the
exceptional line.  Furthermore, by construction, $\wt{G}_{\lambda}$ is
the quadric reduction of $Q_{\lambda}$ with respect to the section
$s$.

Finally, we get that $S= \PP(\pi_*(L_1))=\PP(\pi_*(-\omega_{X/C}-2s)) =S'$. On the fibers where $s$ cuts out a general point, the conic pencils $Q' \to S$ and $Q'' \to S$ coincide.
\end{proof}

\subsection{Clifford algebras associated to quartic Del Pezzo
fibrations and rationality}

The previous constructions give three quadric fibrations $Q \to S$,
$Q' \to S$, and $Q'' \to S$.  Now we will show that their even
Clifford algebras are Morita equivalent.

Let $\kc_0$ (resp.\ $\kc_0'$) be the even Clifford algebra associated
of $Q \to S$ (resp.\ $Q' \to S$).  The following is then an immediate
consequence of the Morita invariant of the Clifford algebra under
hyperbolic splitting (Corollary~\ref{morita}) and of relative
homological projective duality (Theorem \ref{hpd}\eqref{hpd_Fano}).

\begin{corollary}
\label{equicat}
Let $\pi:X\to C$ be a generic quartic Del Pezzo fibration over a
smooth proper curve over a field of characteristic $\neq 2$.  Then the
categories $\cat{A}_X$, $\Db(S,\kc_0)$, and $\Db(S,\kc_0')$ are
equivalent.
\end{corollary}

Now recall that the even Clifford algebra of a conic bundle gives rise
to a (generally ramified) 2-torsion element of $\Br(k(S))$.  Notably,
if two conic bundles coincide on an open subset of $S$ then their even
Clifford algebra define the same Brauer class.  Let $\kc_0''$ be the
even Clifford algebra of $Q'' \to S$.  Considering Theorem \ref{conicbundles}, the following is now immediate.

\begin{corollary}
\label{equicat2}
The category $\cat{A}_X$ is equivalent to $\Db(S,\kc_0'')$.
\end{corollary}

An interesting case of quartic Del Pezzo fibrations arises when
$C=\PP^1$ and it makes sense to wonder about the rationality of $X$.
By \cite[Thm.~3.2]{colliot-thelene_sansuc_swinnerton-dyer:quadrics_I},
$X$ is $k(t)$-birational to the conic bundle $Q'\to S$ and Lemma
\ref{lem:rationality} implies that the rationality of $X$ is
equivalent to that of $Q'$.  Alexeev \cite{alekseev:dP4} and Shramov
\cite{shramov:dp4}, by reducing to the study of $Q''$, proved that the
rationality of $X$ depends only on the topological Euler
characteristic of $X$.

On the other hand, if $C=\PP^1$, then $S$ is a Hirzebruch surface. For
conic bundles on Hirzebruch surfaces there exists a criterion of
rationality in terms of derived categories.

\begin{theorem}\label{marmic}{\rm (\cite[Thm. 1.2]{bolognesi_bernardara:conic_bundles})}
If $S$ is a Hirzebruch surface or $\PP^2$ over $\CC$, then $Q''$ is
rational if and only if it representable in codimension 2.  In
particular, there is a semiorthogonal decomposition
$$
\Db(S,\kc_0'') = \langle \Db(\Gamma_1), \dotsc, \Db(\Gamma_k), E_1,
\dotsc, E_l \rangle,
$$
where $\Gamma_i$ are smooth projective curves and $E_i$ exceptional objects if and only if we have $J(Y) = \oplus J(\Gamma_i)$.
\end{theorem}
\begin{proof}
While \cite[Thm.~1.2]{bolognesi_bernardara:conic_bundles} is stated
only for \it minimal \rm rational surfaces, the same arguments can be
used for the non-minimal Hirzebruch surface $S:=\FF_1$ as well. One
implication is a corollary of
\cite[Thm.~1.1]{bolognesi_bernardara:conic_bundles}: from the
semiorthogonal decomposition it is possible to split the intermediate
Jacobian of $X$ into Jacobian of curves as principally polarized
abelian varieties. This is sufficient to get rationality, thanks to
Shokurov \cite[Thm.~10.1]{shokuprym}. Actually, Shokurov states the
Theorem for minimal surfaces, but he actually proves it for $\FF_1$ as
well \cite[\S10]{shokuprym}, \cite[Rem.~1]{iskovconicduke}.  The
other implication follows by the case-by-case analysis from
\cite[\S6]{bolognesi_bernardara:conic_bundles}.  Indeed, (see
\cite[Conj.~I,~Rem.~1]{iskovconicduke}) the discriminant divisor of a
rational conic bundle over $S$ has either a trigonal or a
hyperelliptic structure induced by the natural fibration $S \to
\PP^1$.
\end{proof}

Finally, we can give a proof of Theorem~\ref{thm:1}, obtaining a
categorical criterion for rationality of the threefold $X$.

\begin{proof}[Proof of Theorem~\ref{thm:1}]
For a quartic Del Pezzo fibration $X \to \PP^1$, the semiorthogonal
decomposition takes the form
$$
\Db(X) = \langle \cat{A}_X, \ko_{X/\PP^1}(1), \pi^* \ko_{\PP^1}(1) \otimes
\ko_{X/\PP^1}(1) \rangle,
$$
where $\ko_{X/\PP^1}(1)$ and $\pi^* \ko_{\PP^1}(1) \otimes
\ko_{X/\PP^1}(1) $ are exceptional objects. The result then follows as
a corollary of Corollary~\ref{equicat2} and Theorem~\ref{marmic}.
\end{proof}

\section{Relative intersections of two 4-dimensional quadrics}\label{section:4fold}

In this section, we use our techniques to study the derived category
of a fourfold $X$ admitting a fibration $X \to \PP^1$ in intersections
of two four-dimensional quadrics.  Moreover we propose a
conjecture---in the spirit of Kuznetsov's on cubic fourfolds
\cite{kuznetsov:cubic_fourfolds}---relating the rationality of such
varieties to their categorical representability. Our analysis is based
on the study of the semiorthogonal decompositions of the quadric
fibrations involved and the even Clifford algebras associated to the
relative surface fibration defined by homological projective duality.
We will also present a moduli space-theoretic way to attack this
problem, which will lead us to consider twisted universal sheaves.
In this section we will eventually need to work over an algebraically
closed field of characteristic zero.

\subsection{Intersections of 4-dimensional quadrics and Clifford algebras}

Let $Y$ be a scheme and $\pi: X \to Y$ a fibration in complete
intersections of two four-dimensional quadrics.  More precisely, we
suppose that there exists a vector bundle $E$ of rank 6
such that $X \subset \PP(E)$ is the
relative complete intersection of two quadric fibrations $Q_1 \to Y$
and $Q_2 \to Y$ of relative dimension 4. In this case, we have
$\omega_{X/Y} = \ko_{X/Y}(-2)$ and thus, by
Proposition~\ref{decompo-of-relative}, for every integer $j$, we have
a semiorthogonal decomposition
\begin{equation}
\label{deco-fourfold}
\Db(X) = \langle \cat{A}^j_X, \Db(Y)(j), \Db(Y)(j+1) \rangle.
\end{equation}
To simplify the notation, let us denote $\cat{A}_X= \cat{A}_X^1$. 

The two quadric bundles $Q_1 \to Y$ and $Q_2 \to Y$ span a quadric
bundle $Q \to S$ over a $\PP^1$-bundle $S \to Y$.  Let $\kc_0$ be the
even Clifford algebra of $Q\to S$.  Recall that by relative
homological projective duality (Theorem~\ref{hpd}\eqref{hpd_Fano}) we
have $\cat{A}_X \simeq \Db(S,\kc_0)$.  If $Q \to S$ has a regular
section, then by quadric reduction we obtain a quadric surface
fibration $Q' \to S$, with associated even Clifford algebra $\kc_0'$.
Assuming we are in the generic situation (see
Definition~\ref{def:generic-fibration}), then $Q \to S$, and therefore
$Q' \to S$ (by Corollary~\ref{cor:degeneration_reduction}), has simple
degeneration with smooth discriminant divisor.  By
Theorem~\ref{morita}, we have $\Db(S,\kc_0) \simeq \Db(S,\kc_0')$ and
$\kc_0$ and $\kc_0'$ define Azumaya algebras $\CliffB_0$ and
$\CliffB_0'$ on the discriminant double cover $T \to S$, with the same
Brauer class $\beta \in \Br(T)$.

Note that if $Y$ is a curve over an algebraically closed field $k$,
then $X\rightarrow Y$ always has a generically smooth section by
Lemma~\ref{lem:C_2}.  Then, by the Amer--Brumer
Theorem~\ref{amer-brumer}, $Q \to S$ always has a generically smooth
section. In order to perform quadric reduction we need a smooth
section:\ we will explicitly underline when it is needed.  When $X$ is
smooth over $k$ (e.g., if $Y$ is smooth over $k$ of characteristic
$\neq 2$ and $X \to Y$ is generic, by
Proposition~\ref{prop:generic_remarks}), then any section of $X\to Y$
is smooth by Lemma~\ref{lem:regsect}.

As previously recalled, relative homological projective duality
(Theorem~\ref{hpd}\eqref{hpd_Fano}) gives the first description of the
category $\cat{A}_X$.  If we suppose that $Q \to S$ has a globally
smooth section (which is the case, by Theorem~\ref{amer-brumer}, if
$X \to Y$ has a smooth section), then this can be extended further
thanks to Theorem~\ref{morita}, and the fact that $Q \to S$ has even
relative dimension.  

\begin{corollary}\label{coro:equiva-in-4fold-case}
Let $\pi: X \to Y$ be a generic relative complete intersection of two
four-dimensional quadric fibrations over an integral scheme $Y$ smooth
over $k$.  If the linear span quadric fibration $Q \to S$ has a smooth
section (e.g., if $\pi$ has a smooth section) then the categories
$\cat{A}_X$, $\Db(S,\kc_0)$, $\Db(S,\kc_0')$, and $\Db(T,\beta)$ are
equivalent.
\end{corollary}
\begin{proof}
This is a corollary of relative homological projective duality
(Theorem~\ref{hpd}\eqref{hpd_Fano}) and the Morita invariance of the
even Clifford algebra under quadric reduction
(Corollary~\ref{morita}). In fact, the equivalence
$\Db(S,\kc_0)\simeq \Db(T,\beta)$ does not require the existence of a
smooth section.
\end{proof}

Consider $Y = \PP^1$. The rationality of $X$ is in general unknown. The techniques developed
in this paper and the language of categorical representability (see
Definition~\ref{repre}) allow us to state a conjecture relating the
rationality and the derived category of $X$, in the same spirit as
Kuznetsov's conjecture on cubic fourfolds
\cite{kuznetsov:cubic_fourfolds}.

\begin{conjecture}\label{kuzne-type-conj}
Let $X \to \PP^1$ be a fibration in complete intersections of two
four-dimensional quadrics over an algebraically closed field of
characteristic zero.
\begin{itemize}
\item {\bf Weak version.}  
The fourfold $X$ is rational if and only if it is categorically
representable in dimension at most two.

\item {\bf Strong version.}  
The fourfold $X$ is rational if and only
if $\cat{A}_X$ is representable in dimension at most two.
\end{itemize}
\end{conjecture}

We say that a morphism $X \to Y$ \linedef{contains a line} over $Y$ if
$X$ contains a surface generically ruled over $Y$ via the restriction
of the morphism. We will prove the strong version of Conjecture
\ref{kuzne-type-conj} in two cases:\ if the Brauer class $\beta$ on
$T$ is trivial then we prove that $\cat{A}_X \simeq \Db(T)$ and $X$ is
rational; if $X$ contains a line over $\PP^1$ then $X$ is rational and
we establish an equivalence $\cat{A}_X \simeq \Db(T)$.

Even though, as Lemma~\ref{lemma:lineistrivial} shows, this second case
is contained in the first one, it still deserves its own treatment.
Indeed, relative quadric intersections containing lines were already
considered in the literature (see
\cite{colliot-thelene_sansuc_swinnerton-dyer:quadrics_I}), where the
associated rational parameterizations can be explicitly described.
Also, in this case, we provide an independent moduli space-theoretic
interpretation of the equivalence $\cat{A}_X \simeq \Db(T)$ that
generalizes to other contexts.

\begin{proposition}
\label{prop:gamma-0-is-rational}
Let $\pi: X \to Y$ be a generic relative complete intersection of two
four-dimensional quadric fibrations over an integral scheme $Y$ smooth
over $k$. Then $X$ and $Q'$ are $k$-birationally equivalent.  If
furthermore, $\beta \in \Br(T)$ is trivial, and $Y$ is $k$-rational,
then $X$ is $k$-rational and there is an equivalence $\cat{A}_X \simeq
\Db(T)$.
\end{proposition}
\begin{proof}
For the first claim, by
\cite[Thm.~3.2]{colliot-thelene_sansuc_swinnerton-dyer:quadrics_I},
$X$ and $Q'$ are $k(Y)$-birational, hence $k$-birational by
Lemma~\ref{lem:rationality}.  For the second claim, we have that
$\beta$ is trivial in $\Br(T)$ if and only if $Q'$ has a rational
section over $S$, by Theorem~\ref{thm:well-known}.  In this case,
since $Y$ is $k$-rational (hence $S$ is also $k$-rational), $Q'$ is
thus $k(S)$-rational, hence $k$-rational.  But by the first claim $Q'$
is $k$-birational to $X$.  The last assertion is a straightforward
consequence of Corollary \ref{coro:equiva-in-4fold-case}.
\end{proof}

\subsection{Quadric intersection fibrations containing a line}
The case of complete intersections of two quadrics of dimension $n
\geq 4$ over a field that contain a line was studied in
\cite{colliot-thelene_sansuc_swinnerton-dyer:quadrics_I}.  In
particular, an explicit geometric proof of the rationality of $X \to
\PP^1$ can be given as a special case (see also
\cite[Prop.~2.2]{colliot-thelene_sansuc_swinnerton-dyer:quadrics_I}).

\begin{proposition}
\label{linerational}
Let $k$ be a field, $Y$ be an integral $k$-rational variety, and $W
\to Y$ be a generic relative complete intersection of two quadric
fibrations of dimension $n \geq 4$.  If $W$ contains a line over $Y$,
then $W$ is $k$-rational.
\end{proposition}
\begin{proof}
Consider the generic fiber $W_\eta\subset \PP_{k(Y)}^{n+1}$, and
$l$ the line contained in $W_\eta$.  The fibers of the projection
$p_l:\PP_{k(Y)}^{n+1} \dashrightarrow \PP_{k(Y)}^{n-1}$ off the line $l$ are the
2-dimensional planes containing the line $l$. Consider now the two
quadrics $Q_{1,\eta}$ and $Q_{2,\eta}$, cutting out $W_\eta$. A generic fiber
of the projection $p_l$ intersects $Q_{i,\eta}$ along a pair of lines:\
$l$ and a second line $l_i$, for $i=1,2$.  The two lines $l_1$ and
$l_2$ meet in just one point in each fiber of $p_l$. Then by
restricting $p_l$ to $W_\eta$ we get the desired birational map
$
W_\eta \dashrightarrow \PP_{k(Y)}^{n-1}.
$
Then $W$ is birational to a projective bundle over the $k$-rational scheme $Y$, and hence
$W$ is $k$-rational.
\end{proof}

In our case of a generic relative complete intersection $X \to Y$ 
of two quadrics of dimension $4$, we can say more.  We prove that
containing a line implies the vanishing of the associated Brauer
class.  Then rationality, and indeed Conjecture~\ref{kuzne-type-conj}
for this case, is then a consequence of Proposition
\ref{prop:gamma-0-is-rational}.

\begin{lemma}
\label{lemma:lineistrivial}
Let $\pi: X \to Y$ be a generic relative complete intersection of two
four-dimensional quadric fibrations over a regular integral scheme $Y$.
If $X$ contains a line over $Y$ then $\beta$ is trivial.
\end{lemma}
\begin{proof}
If $X \to Y$ contains a line then $Q \to S$ contains a line $l$.  A
local computation proves that, if $Q\to S$ has simple degeneration,
then there exists a smooth section $s$ whose image is contained in the
line $l$.  Let $Q'$ be the quadric reduction along $s$. In the
process, $l$ is contracted to a section of $Q'$ over $S$. 
By Theorem~\ref{thm:well-known}, the Brauer class $\beta'$ associated
to $Q' \to S$ is trivial. Finaly, by Corollary~\ref{morita}, we have
$\beta=\beta'$ in $\Br(T)$, hence $\beta$ is trivial as well.
\end{proof}

If the base scheme $Y=\PP^1$ over an algebraically closed field
then we can even say more.

\begin{lemma}
\label{lemma:lineifftrivial}
Let $k$ be algebraically closed and $\pi: X \to \PP^1$ be a generic
relative complete intersection of two four-dimensional quadric
fibrations. Then $X$ contains a line over $\PP^1$ if and only if
$\beta$ is trivial.
\end{lemma}
\begin{proof}
One direction is Lemma \ref{lemma:lineistrivial}.  For the other, by
Lemma \ref{lem:C_2}, $\pi$ has a rational section (which is smooth by
Remark~\ref{rem:smooth_section}), hence by the Amer--Brumer
Theorem~\ref{amer-brumer}, so does the associated quadric fibration $Q
\to S$.  Then we can perform quadric reduction and obtain a quadric
surface fibration $Q'\to S$. The Brauer class associated to $Q'$ is
again $\beta$, hence by Theorem~\ref{thm:well-known}, $Q'$ has a
rational section. The existence of a rational section of $Q'$ implies
the existence of a rational line over $\PP^1$ in $Q$.  Finally, we use
the Amer Theorem (cf.\ Remark~\ref{rem:Amer}) to obtain a rational
line over $\PP^1$ inside $X$.
\end{proof}

\begin{corollary}
\label{firstway}
Let $k$ be algebrically closed and $X \to \PP^1$ be a generic relative
complete intersection of two four-dimensional quadrics.  If $X$
contains a line over $\PP^1$ then $X$ is rational and $\cat{A}_X$ is
equivalent to $\Db(T)$, hence is representable in dimension 2. 
\end{corollary}

We recall that if $M$ is a K3 surface and $\beta \in \Br(M)$ is
nontrivial, then $\Db(M,\beta)$ is never equivalent to $\Db(M)$
\cite[Rem.~7.10]{huybrechts-stellari}.  We do not know if the same
result holds for the class of surfaces $T$ that carry a fibration $T
\to \PP^1$ in hyperelliptic curves, and the Brauer classes $\beta \in
\Br(T)$, considered above.  Such a result would strengthen Corollary~\ref{firstway}.

\subsection{A moduli space-theoretic interpretation}

On the other hand, the rich geometry of the fibration $X \to \PP^1$
allows one to give another description of $\cat{A}_X$ as a derived
category of twisted sheaves on $T$ over the complex numbers.

More generally, let $W \to Y$ be the intersection of two
$2m$-dimensional quadric fibrations over a smooth projective variety,
and $Q \to S$ the associated quadric span.  For any point $y$ of $Y$,
there is a hyperelliptic curve $T_y$ parameterizing maximal isotropic
subspaces of the fiber $Q_y$.  In fact, $T_y$ is also the fine moduli
space for spinor bundles over $W_y$ (see
\cite{bondal_orlov:semiorthogonal} and \cite{ottaviani:spinor_bundles}
for definitions and constructions).  One can in general consider the
same moduli problem in the relative case, see
\cite[Thm.~4.3.7]{huybrechts-lehn}.  The most substantial difference
concerns the existence of a universal family on the product $W \times
T$.  While this can be constructed explicitly on $W_y \times T_y$
\cite[\S 2]{bondal_orlov:semiorthogonal}, the obstruction to the
existence of a universal family over $T$ is given by an element
$\omega$ of the Brauer group $\Br(T)$. That is, we have an
$\omega$-twisted universal sheaf $E$ on $W \times T$
\cite{caldararu:thesis}. The existence of a linear space of dimension
$(m-1)$ in $W$ over $Y$ implies the existence of a universal object
and then the vanishing of $\omega$.

Combining this description of $T$ as a non-fine moduli space and the
semiorthogonal decomposition in the absolute case given in
\cite{bondal_orlov:semiorthogonal}, we get an equivalence between
$\cat{A}_W$ and a category of twisted sheaves on $T$.

\begin{proposition}
\label{prop:mod-theoretical}
Let $Y$ be a smooth projective variety over $\CC$ and $W \to Y$ be a
generic relative complete intersection of two $2m$-dimensional quadric
fibrations.
Then there is an equivalence $\cat{A}_W \simeq \Db(T,\omega)$.  In
particular, if $W$ contains an $(m-1)$-plane over $Y$, then $\omega =
0$ and $W$ is categorically representable in dimension 2.  Moreover,
in this case, if $Y$ is rational then so is $W$.
\end{proposition}
\begin{proof}
By an $\omega$-twisted skyscraper sheaf $\kf_t$ of $T$ we mean a
simple $\omega$-twisted sheaves supported on the closed point $t$ of
$T$ (such a sheaf exists for every closed point, e.g., by
\cite[Cor.~1.2.6]{caldararu:thesis}).  Then one can check, as in
the untwisted case (e.g., \cite[Prop.~3.17]{huybrechts-book}), that the
collection of all $\omega$-twisted skyscraper sheaves forms a spanning
class (as defined in \cite[Def.~1.47]{huybrechts-book}) for
$\Db(T,\omega)$.

The $\omega$-twisted universal family $E$ over $W \times T$ is
a $pr_2^* \omega$-twisted sheaf such that for each closed point $y$ of
$Y$, the sheaf $E_y$ is the universal spinor bundle on the product $W \times_{k(y)} T = W_y \times T_y$.  Hence,
the restricted Fourier--Mukai functor $\Phi_{E_y} : \Db(T_y) \to
\Db(W_y)$ is fully faithful by
\cite[Thm.~2.7]{bondal_orlov:semiorthogonal} (note that the Brauer
class $\omega_y$ is trivial on $T_y$ by Tsen's theorem).
We observe that, for each structure sheaf $\kf_t$ of a
closed point $t$ of $T$ on the fiber $T_y$, we have that $\Phi_{E}(\kf_t)
= j_* \Phi_{E_y}(\kf_t)$, where $j : W_y \to W$ is the closed
embedding of the fiber.  As $j$ is an affine morphism, hence
cohomologically acyclic, we have that
$$\mathrm{Hom}_{\Db(T,\omega)}(\kf_{t_1},\kf_{t_2}) = \mathrm{Hom}_{\Db(W)}(\Phi_E(\kf_{t_1}), \Phi_E(\kf_{t_2})),$$
for any closed point $t_1$ and $t_2$ of $T$.
Hence we can apply \cite[Prop.~1.49]{huybrechts-book} to prove that $\Phi_E$ is fully faithful.

The existence of an $(m-1)$-plane over $Y$ allows one to construct the
universal family, so that $\omega =0$, and the categorical
representability in dimension 2 follows. Recalling Theorem
\ref{linerational}, we get that if $W$ contains a linear space of
dimension $(m-1)$ over $Y$ (and hence a line), it is rational.
\end{proof}

Proposition \ref{prop:mod-theoretical} generalizes (and gives a
different proof of) Lemma \ref{lemma:lineistrivial}.

Recall that in \cite{bondal_orlov:semiorthogonal}, it is shown that if $Z$ is a
smooth intersection of quadric fourfolds and contains a line (which
is always the case over an algebraically closed field of
characteristic zero), then there is a direct way to construct a
universal family of spinor bundles on $Z$. In the relative case this
translates into another proof of Proposition
\ref{prop:mod-theoretical}.

Consider now $X \to \PP^1$ the intersection of two generic quadric
fibrations of relative dimension 4 (that is, $m=2$ in the above
discussion).  The family $T \to \PP^1$ factors through $T \to S \to
\PP^1$, where $T \to S$ is the discriminant double cover of the
associated linear span quadric fibration.  We obtained an equivalence
$\Db(T,\omega) \simeq \cat{A}_X$ providing another description of
$\cat{A}_X$.  Some natural questions remain.

\begin{question}
\label{question4} 
Let $\pi: X \to \PP^1$ be a generic relative complete intersection of two
four-dimensional quadric fibrations over an algebraically closed field and $T \to S$ the associated
discriminant cover.
\begin{enumerate}
\item Are $\beta, \omega \in \Br(T)$ the same Brauer class? 
\item Does $X$ contain a line over $\PP^1$ if and only if $\omega$ is trivial?
\item Is $X$ rational if and only if it contains a line over $\PP^1$?
\end{enumerate}
\end{question}

We conjecture that the answer to the first (hence also the second)
question is positive, whereas it is expected that the third question
has a negative answer.

\subsection{An explicit fourfold example}

The aim of this section is to give an explicit example of a rational
fourfold fibered over $\PP^1$ in intersections of two
four-dimensional quadrics.  As before it will be embedded in a $\PP^5$
bundle $\PP(E)$ over $\PP^1$.

\smallskip

Let us first develop the geometric construction over a point $p$. Let
$X_p$ be the intersection of two quadrics in $\PP(E)_p \simeq \PP^5$.
First recall that, if $X_p$ contains a line $L$, then the projection
off $L$ is a rational morphism from $X_p$ into $\PP^5$, inducing a birational
map between $X_p$ and a $\PP^3$ (see
\ref{linerational}).  The corresponding birational map decomposes as
the blow-up of $L$ followed by the contraction of a surface swept by
lines intersecting $L$ onto a curve of genus 2 and degree 5 in $\PP^3$
(see \cite[Prop.~3.4.1(ii)]{isko_fano}). Notably the exceptional
divisor is sent onto a quadric surface in $\PP^3$ containing the genus
2 curve.

\smallskip

On the other hand, suppose that we have a curve $B$ of genus 2 and
degree 5 in $\PP^3$, for example embedded by the linear system $\vert
2K_B+ D \vert$, where $K_B$ is the canonical divisor and $D$ a degree
one divisor. Then standard computations, that we omit, show that the
birational inverse map from $\PP^3$ onto $X_x$ is given by the ideal
of cubics in $\PP^3$ containing the curve $B$. By easy Riemann--Roch
computations (see for instance \cite[\S3]{bertin:trisecants}) $B$ is
contained in a quadric surface and it comes:

\begin{enumerate}
\item[\textit{a)}] either as a divisor of type (2,3) on a smooth quadric surface (the quadric surface is the
trisecant scroll of $B$)
\label{case1}

\item[\textit{b)}] or as the following construction: the intersection of a quadric cone of corank 1 with a
cubic surface passing through
the singularity is a sextic curve that decomposes in two components. The first is a line, a ruling
of the cone. The second,
residual to the line, is our curve $B$. Also in this case, the quadric cone is the trisecant
scroll of $B$.
\label{case2}
\end{enumerate}

The two different kinds of genus 2 degree 5 curves in $\PP^3$
correspond to the different values of the normal sheaf
$N_{L/X_p}$. That is: in the generic case, the normal sheaf is $\ko_L
\oplus \ko_L$. In this case the exceptional divisor of $L$ is $\PP^1
\times \PP^1$ and it is isomorphically sent onto the smooth quadric
surface of case \textit{a)} by the projection off $L$. On the other
hand, we can also have a normal sheaf equal to $\ko_L(1) \oplus
\ko_L(-1)$.  In this case the exceptional divisor is isomorphic to
$\mathbb{F}_2$, that is the desingularization of the quadric cone that
contains $B$ in case \textit{b}).  There are no other possible normal
bundles (see for example \cite[Lemma~3.3.4]{isko_fano}). In both cases
the map given by cubics contracts the trisecant scroll (or its
desingularization) onto the line $L\subset X_x$. Note that if we work
over complex numbers, $X$ always contains a line.

A semiorthogonal decomposition 
$$
\Db(X) = \langle \Db(B), \ko_X, \ko_X (1) \rangle
$$
of the derived category of the intersection of two even dimensional
quadrics was described in \cite{bondal_orlov:semiorthogonal} over
$\CC$ by an explicit Fourier--Mukai functor. Homological projective
duality gives the decomposition over any field
\cite[Cor.~5.7]{kuznetsov:quadrics}.  In case \textit{a)}, we can
give an alternative proof based on the explicit description and
mutations.

\smallskip

Suppose we are in case \eqref{case1} and consider the diagram:
$$
\xymatrix{
& E \ar[dl] \ar@{^{(}->}[r] & W \ar[dl]_{\rho} \ar[dr]^{\varepsilon} & D \ar@{_{(}->}[l]_{i} \ar[dr]^{\pi} \\
B \ar@{^{(}->}[r] & \PP^3 & & X & L \ar@{_{(}->}[l] }
$$
where $\rho: W \to \PP^3$ is the blow-up of the hyperelliptic curve with exceptional divisor
$E$ and $\varepsilon: W \to X$ is the blow-up of the line $L$ with exceptional divisor $i: D \hookrightarrow W$. Let $\pi: D \to L$ be the $\PP^1$-bundle map.
Let us denote by $H=\rho^* \ko_{\PP^3}(1)$ and by $h= 
\varepsilon^*\ko_X(1)$ the two generators of $\Pic(W)$. Then by the previous description
we have $h = 3H - E$ and $H = h - D$, which give $E = 3H-h$, and $D=h-H$. Finally, we
have $\omega_{W} = -4H + E = -H-h$.

Now consider the semiorthogonal decomposition $\Db(\PP^3) = \langle \ko_{\PP^3}(-2),
\dotsc, \ko_{\PP^3}(1) \rangle$. By the blow-up formula we get:
$$
\Db(W) = \langle \Phi \Db(B), -2H, -H, \ko_W, H \rangle,
$$
where $\Phi$ is a fully faithful functor. Mutating $H$ to the left with respect of its
orthogonal complement, we get
$$
\Db(W) = \langle -h, \Phi \Db(B), -2H, -H, \ko_W \rangle,
$$
using \cite[Prop.~3.6]{bondal_kapranov:reconstructions} and $\omega_W = -H -h$. Now mutating $\Phi \Db(B)$ to the right
with respect to $\langle -2H, -H \rangle$, we get a fully faithful functor $\Phi' =
\Phi \circ R_{\langle -2H, -H \rangle} : \Db(B) \to \Db(W)$ and a semiorthogonal
decomposition:
\begin{equation}\label{semiorthogonal1}
\Db(W)= \langle -h, -2H, -H, \Phi' \Db(B), \ko_W \rangle.
\end{equation}
On the other hand, consider the decompositions $\Db(X) = \langle \cat{A}_X, \ko_X, \ko_X(1) \rangle$
and $\Db(L) = \langle \ko_L, \ko_L(1) \rangle$. By the blow-up formula \cite[Thm. 4.3]{orlov:blowup}, we get the following
semiorthogonal decomposition:
$$
\Db(W) = \langle i_* (\pi^* \ko_L \otimes \ko_{\pi}(-1)), i_* (\pi^* \ko_L(1)
 \otimes \ko_{\pi}(-1)), \cat{A}_X, \ko_W, h \rangle,
$$
where we identified, up to the equivalence $\varepsilon^*$ the category $\cat{A}_X$
with its pull-back to $\Db(W)$.

Now notice that $\ko_{\pi}(1) = H_D$. Then we can rewrite the semiorthogonal decomposition as:
$$
\Db(W) = \langle \ko_D(-H), \ko_D(h-H), \cat{A}_X, \ko_W, h \rangle.
$$
First, mutate $\cat{A}_X$ to the left with respect to $\ko_D(h-H)$ to get
$$
\Db(W) = \langle \ko_D(-H),\cat{A}_X, \ko_D(h-H), \ko_W, h \rangle,
$$
where we identify, up to equivalence, $\cat{A}_X$ with the image of the mutation.
Notice that $\ko_D(h-H) = \ko_D(D)$. It is standard to check, by definition
of mutation and using the evaluation sequence of the smooth divisor $D=h-H$, that the mutation of the pair
$\langle \ko_D(D), \ko_W \rangle$ gives the pair $\langle \ko_W, D \rangle=\langle \ko_W, h-H \rangle$.
Performing this mutation we get then:
$$
\Db(W) = \langle \ko_D(-H),\cat{A}_X, \ko_W, h-H,  h \rangle
$$
Now mutate $\langle h-H, h \rangle$, recalling $\omega_W = -H-h$, to the left with respect to
its orthogonal complement to get
$$
\Db(W) = \langle -2H, -H,\ko_D(-H),\cat{A}_X, \ko_W \rangle.
$$
Once again, by definition of mutation and the evaluation sequence, up to a twist with $-H$,
the mutation of the pair $\langle -H, \ko_D(-H) \rangle$ gives the pair
$\langle -H-D, -H \rangle$, and then the decomposition
\begin{equation}\label{semiorthogonal2}
\Db(W) = \langle -2H, -h, -H, \cat{A}_X, \ko_W \rangle,
\end{equation}
since $D = h-H$. Comparing the decompositions \eqref{semiorthogonal1} and \eqref{semiorthogonal2},
we get that the pair $\langle -h, -H \rangle$ is completely orthogonal and
then we get the required equivalence between $\cat{A}_X$ and $\Db(B)$.

\smallskip

Let us come to the relative setting. Suppose that we have a rational
family in the universal degree 1 Picard variety $\PIC^1$ over the
moduli space $\mathcal{M}_2$ of genus two curves. This is
just a rational family of couples $(B_t,D_t)_{t \in\PP^1}$, where
$D_t$ is a degree one divisor on the genus two curve $B_t$.
Call $\mathbb{B}$ the total space of the family $B_t$ and
$\tau: \mathbb{B} \rightarrow \PP^1$ the projection. Then $B_t$ is
embedded as a family of degree 5 curves in the rank 3 projective
bundle $\PP(U) =\PP(\tau_*(2K_{B_t}+D_t))$ over $\PP^1$. We can suppose
that we have chosen $(B_t,D_t)$ so that that each curve $B_t$
has a smooth trisecant scroll in $\PP(U_t)$ as in case \eqref{case1}.

By our preceding remarks, if we apply to $\PP(U)$ the birational
transformation given by relative cubics in $\PP(U)$ vanishing on
$\mathbb{B}$, we get a fourfold $\cX$, fibered over $\PP^1$ in
intersections of quadrics. Moreover, by construction, $X$ is clearly
rational and $\mathbb{B}$ is a fibration in hyperelliptic curves. Now
we use Orlov's formula \cite[Thm. 4.3]{orlov:blowup} for the derived
category of the blow-up of a smooth projective variety in order to get
a semiorthogonal decomposition of the derived category of $\cX$.
Moreover, by mutations, we can prove the strong version of Conjecture
\ref{kuzne-type-conj}. We note that the trisecant lines of the curves
in the family $B_t$ make up a quadric surface fibration $\cQ$
contained in $\PP(U)$. Let $\varphi$ the rational map defined on
$\PP(U)$ by the ideal of relative cubics containing $\mathbb{B}$.

\begin{theorem}\label{thm-with-mutations}
Let $\cX$ be the rational fourfold obtained as the image of $\varphi$. Then $\cX$
is categorically representable in dimension 2 as follows:
$$
\Db(\cX) = \langle \Db(\mathbb{B}), E_1, E_2, E_3 , E_4 \rangle,
$$
where $E_i$ are exceptional objects. Moreover, there is an equivalence 
$$
\cat{A}_\cX\simeq \Db(\mathbb{B}).
$$
\end{theorem}

\begin{proof}
The fourfold $\cX$ is obtained by first blowing-up $\PP(U)$ along $\mathbb{B}$ and then contracting $\cQ$ onto a ruled
surface $\mathbb{F}$.
The quadric fibration $\cQ$ is the exceptional divisor over $\mathbb{F}$ when one blows-up $\cX$ along $\mathbb{F}$.
Let us denote by $\cW$ the blow-up of $\PP(U)$ along $\mathbb{B}$, which is naturally isomorphic to the blow-up of $\cX$ along $\mathbb{F}$.

The two assertions are proved by comparing the two semiorthogonal decompositions of $\Db(\cW)$
induced by the two descriptions of $\cW$, as the blow-up on one hand of $\PP(U)$ and on the other of $\cX$. To this
end, recall that the derived category of $\PP^1$ admits a semiorthogonal decomposition by two categories generated by exceptional line bundles.
we get two semiorthogonal
decompositions of $\Db(\cW)$ given by exceptional objects with orthogonal complements respectively
$\Db(\mathbb{B})$ and $\cat{A}_\cX$. 

In order to get the required equivalence explicitly, one has to perform the same mutations as before in the relative setting.
Indeed one has a diagram
$$
\xymatrix{
& \cE \ar[dl] \ar@{^{(}->}[r] & \cW \ar[dd]^{\alpha}\ar[dl]_{\rho} \ar[dr]^{\varepsilon} & \cQ \ar@{_{(}->}[l]_{i} \ar[dr]^{\pi} \\
{\mathbb{B}} \ar@{^{(}->}[r] \ar[drr]_{\tau} & \PP(U) \ar[dr]^{p} & & \cX \ar[dl]_{q} & {\mathbb{F}} \ar[dll] \ar@{_{(}->}[l] \\
& & \PP^1\,.}
$$
We will keep the same notations for the relative line bundles in ${\mathrm{Pic}}(\cX/\PP(U))$ that
we used for ${\mathrm{Pic}}(X)$ in the previous mutations. Moreover, notice that $\rho \circ p = \alpha = \varepsilon \circ q$.
Let us moreover introduce, for $L$ in ${\mathrm{Pic}}(\cX/\PP^1)$, the notation $\alpha^* \Db(\PP^1) (L)$
for the admissible category whose objects are of the form $\alpha^*A \otimes L$, where $A$ is in $\Db(\PP^1)$.
So we end up with the following semiorthogonal decomposition (a relative version of \eqref{semiorthogonal1}):
\begin{equation}\label{eq-deco-rela-1}
\Db(\cW)= \langle \Phi \Db({\mathbb{B}}), \alpha^\ast \Db(\PP^1)(-2H), \alpha^\ast \Db(\PP^1) (-H),\alpha^\ast \Db(\PP^1),
\alpha^\ast \Db(\PP^1)(H) \rangle,
\end{equation}
by blowing up $\PP(U)$ along $\mathbb{B}$, where $\Phi$ is a fully faithful functor.
On the other hand, recall that $\Db(\mathbb{F})$ decomposes into two copies of $\Db(\PP^1)$, and consider
the blow-up formula for $\cW \to \cX$. In order to complete the proof, notice that we can
perform the same mutations as before in the relative setting in the following
sense: any occurrence of $\alpha^* \Db(\PP^1)$ can be replaced by two exceptional objects (actually, two line bundles in
$\alpha^* {\mathrm{Pic}}(\PP^1)$). The previous mutations carry on on these exceptional sets: whenever we
act by a mutation inside $\alpha^* \Db(\PP^1)$, this is just changing the choice of the pair of line bundle decomposing
it. The mutations involving relative line bundles carry on just as before. We end up hence with the following
semiorthogonal decomposition (a relative version of \eqref{semiorthogonal2}):
\begin{equation}\label{eq-deco-rela-2}
\Db(\cW)= \langle \alpha^\ast \Db(\PP^1)(-2H), \alpha^\ast \Db(\PP^1)(-h), \alpha^\ast \Db(\PP^1)(-H),
\cat{A}_\cX, \alpha^\ast \Db(\PP^1) \rangle.
\end{equation}
Comparing via mutations the decompositions \eqref{eq-deco-rela-1} and \eqref{eq-deco-rela-2} we get the proof.
\end{proof}
\begin{remark}
Notice that the proof of Theorem~\ref{thm-with-mutations} applies
whenever we consider the base to be a rational smooth projective
variety whose derived category admits a full exceptional sequence of vector bundles, as for
example any $\PP^n$ or any smooth quadric hypersurface.
\end{remark}

\begin{remark}
Theorem~\ref{thm:well-known} has already been (implicitly) used by
some authors to describe the birational geometry of high dimensional
varieties.  Namely, if a cubic fourfold contains a plane, then it is
birational to a quadric surface bundle over $\PP^2$.  The study of the
Brauer class associated to this quadric bundle has been developed in
\cite{hassett:rational_cubic} (implicitly) and
\cite{kuznetsov:cubic_fourfolds}.  In Appendix
\ref{subsec:grassmannian}, we verify that the Brauer class considered
in \cite{hassett:rational_cubic} is the same as the one considered
here and in \cite{kuznetsov:cubic_fourfolds}.
\end{remark}

\appendix

\section{A comparison of even Clifford algebras}
\label{appendix_clifford}

In his thesis, Bichsel~\cite{bichsel:thesis} constructed an even
Clifford algebra of a line bundle-valued quadratic form on an affine
scheme using faithfully flat descent.
Parimala--Sridharan~\cite[{\S4}]{parimala_sridharan:norms_and_pfaffians}
generalized this construction to any scheme.
Bichsel--Knus~\cite{bichsel_knus:values_line_bundles} and
Caenepeel--van Oystaeyen~\cite{caenepeel_van_oystaeyen} gave
constructions of a \emph{generalized} or $\ZZ$-graded Clifford
algebra, of which the even Clifford algebra is the degree zero piece.
These constructions are detailed in \cite[\S1.8]{auel:clifford}.

Independently, Kapranov~\cite[\S4.1]{kapranov:derived} introduced a
\emph{homogeneous} Clifford algebra, which was further developed by
Kuznetsov~\cite[\S3]{kuznetsov:quadrics}, to study the derived
category of projective quadrics and quadric fibrations.  In this
context, the even Clifford algebra is defined as a certain limit of
the graded pieces. In this appendix, we will show that our
construction of the even Clifford algebra coincides with that of
Kuznetsov.

First, we would like to point out some differences between our
conventions in \S\ref{sec:Quadratic_forms} and those in
\cite[\S3]{kuznetsov:quadrics}.  The value line bundle $L$ differs by
a dual and all quadratic forms are assumed to be primitive in
\cite[\S3]{kuznetsov:quadrics}.  Also, certain noncanonical
isomorphisms involving symmetric squares and dualization are used
(thereby requiring the standing hypothesis of working in
characteristic zero), and the definition of the Clifford algebra of a
quadratic form over a field given in \cite[\S2.4]{kuznetsov:quadrics}
is incorrect if 2 is not invertible.

Let $S$ be a scheme and $(E,q,L)$ be a quadratic form over $S$ (see
\S\ref{subsec:Quadratic_forms_and_quadric_fibrations}).  Kuznetsov
defines the \linedef{homogeneous Clifford algebra} as the graded
quotient
$$
\CliffKuz = T(E)/I =
\bigoplus_{n \geq 0} T^n(E)/I_n = \bigoplus_{n \geq 0} \CliffKuz_n
$$
of the tensor algebra of $E$ by the homogeneous ideal $I =
\bigoplus_{n\geq 0} I_n$ generated by 
$$
I_2 = \ker(q : S_2E \to L) \subset T^2(E).
$$  
We point out that $\CliffKuz_0=\OO_S$, $\CliffKuz_1
= E$, and that $\CliffKuz_2$ fits into a commutative diagram
\begin{equation}
\label{eq:B2}
\begin{split}
\xymatrix@R=12pt@C=30pt{
0\ar[r]&S_2 E\ar[d]^q \ar[r]&T^2(E)\ar[d]^{\vp_2}\ar[r]&\exterior^2
E\ar@{=}[d]\ar[r]&0\\ 
0\ar[r]&L\ar[r]&\CliffKuz_2\ar[r]&\exterior^2 E\ar[r]&0
}
\end{split}
\end{equation}
of locally free $\OO_S$-modules, where $\vp = \oplus_{n \geq 0} \vp_n$
is the quotient map $T(E) \to \CliffKuz$.  Via the subbundle $L \to
\CliffKuz_2$ there is a canonical morphism $\CliffKuz_{n} \tensor L
\to \CliffKuz_{n+2}$ for every $n \geq 0$.  By induction, there is a
canonical monomorphism $L^{\tensor m} \to \CliffKuz_{2m}$ for all $m
\geq 0$.  Thus there is an monomorphism $S(L) = \oplus_{n \geq 0}
L^{\tensor n} \to \CliffKuz$ of degree 2, whose image is actually
central, see \cite[Lemma~3.6]{kuznetsov:quadrics}.  Also, there is a
canonical morphism $\CliffKuz_{2m} \tensor L\dual{}^{\tensor m} \to
\CliffKuz_{2(m+1)}\tensor L\dual{}^{\tensor (m+1)}$ for every $m \geq
0$.  Kuznetsov~\cite[\S3.3]{kuznetsov:quadrics} defines the even
Clifford algebra as the colimit
$$
\mathcal{B}_0 = \varinjlim \CliffKuz_{2m}\tensor L\dual{}^{\tensor m}
$$
of this directed system of $\OO_S$-module morphisms.

\begin{proposition}
\label{prop:compare}
Let $(E,q,L)$ be a quadratic form of rank $n$ on a scheme $S$.  The
even Clifford algebras, $\CliffAlg_0$ as defined in
\S\ref{subsec:Even_Clifford_algebra} and $\mathcal{B}_0$ as defined in
\cite[\S3.3]{kuznetsov:quadrics}, are isomorphic $\OO_S$-algebras.
\end{proposition}
\begin{proof}
Tensoring the quotient morphism $\vp_2 : T^2(E) \to \CliffKuz_2$ by
$L\dual$, there is an induced morphism $\psi : T^2(E) \tensor L\dual
\to \mathcal{B}_0$.  We will verify that $\psi$ satisfies the
universal property of $\CliffAlg_0$, deducing the existence of an
$\OO_S$-algebra morphism $\Psi : \CliffAlg_0 \to \mathcal{B}_0$.

First, we observe that by diagram \eqref{eq:B2}, $\psi(v \tensor v \tensor
f) = q(v) \tensor f$ in $L \tensor L\dual \to \CliffKuz_2 \tensor
L\dual$, hence $\psi(v \tensor v \tensor f) = f(q(v))$ under the
identification $L \tensor L\dual = \OO_S \subset \mathcal{B}_0$ with
the identity in the limit.

Second, we calculate that
\begin{align*}
\psi(u \tensor v \tensor f) \, \psi(v \tensor w \tensor g) & {}= 
\vp_2(u \tensor v) \tensor f \tensor \vp_2(v \tensor w) \tensor g = 
f\tensor \vp_4(u\tensor v \tensor v \tensor w) \tensor g \\
& {} = f \tensor q(v) \tensor \vp_2(u \tensor w) \tensor g = f(q(v)) \,
\psi(u \tensor w \tensor g)
\end{align*}
using the facts that $L \subset \CliffKuz_2$ is central in
$\CliffKuz$, the quotient map $\vp = \oplus \vp_n$ is a graded
morphism, and that $\vp_4$ evaluated on the submodule $E \tensor S_2 E
\tensor E$ factors through $L \tensor \CliffKuz_2$.

Hence $\psi$ satisfies the universal property of $\CliffAlg_0$, so
there exists an $\OO_S$-algebra morphism $\Psi : \CliffAlg_0 \to
\mathcal{B}_0$.  Since these algebras are locally free of the same
rank, it suffices to check that $\Psi$ is an isomorphism on fibers.
Hence we are reduced to the case where $S$ is the spectrum of a field
$k$.  Choosing a generator $L = lk$ and a $k$-basis $e_1,\dotsc,e_n$
of $E$, then by its construction, $\CliffAlg_0$ has a $k$-basis
consisting of $e_{i_1}\dotsm e_{i_2m} l^{-m}$ where $1 \leq i_1 <
\dotsm < i_{2m} \leq n$.  For $2m \geq n$, the limit stabilizes and
$\mathcal{B}_0 = \CliffKuz_{2m}\tensor L\dual{}^{\tensor m}$, showing
that $\mathcal{B}_0$ has a $k$-basis of the same shape.  Since $\psi$
preserves basis elements of the shape $e_{i_1} e_{i_2} l\inv$ by its
very definition, $\Psi$ will preserve the shape of the entire basis.
Hence $\Psi$ is an isomorphism.
\end{proof}

Kuznetsov~\cite[\S3.3]{kuznetsov:quadrics} also defines the Clifford
bimodule (``the odd part of the Clifford algebra'') as the colimit
$$
\mathcal{B}_1 = \varinjlim \CliffKuz_{2m+1}\tensor L\dual{}^{\tensor
m}
$$
of the directed system of $\OO_S$-module morphisms $\CliffKuz_{2m+1}
\tensor L\dual{}^{\tensor m} \to \CliffKuz_{2(m+1)+1} \tensor
L\dual{}^{\tensor (m+1)}$ defined similarly as above.  Arguing
similarly as in the proof of Proposition~\ref{prop:compare}, there is
an isomorphism of Clifford bimodules $\CliffAlg_1$ and $\mathcal{B}_1$
equivariant for the above isomorphism of even Clifford algebras.

\medskip

As constructed, $\CliffKuz$ is a quadratic $\OO_S$-algebra in the
sense of \cite[Ch.~1,~\S2]{poli-posi}.  There is another quadratic
$\OO_S$-algebra associated to $(E,q,L)$:\ the coordinate
algebra $\CoordKuz = \bigoplus_{l \geq 0} \pi\pushforward
\OO_{Q/S}(l)$.  The cohomology of the exact sequence
\eqref{eq:sequnce-quadric-fibration} tensored by $\OO_{Q/S}(l-1)$,
implies, noting that $R^1\pi\pushforward\OO_{Q/S}(l)=0$, that
$\pi\pushforward \OO_{Q/S}(l) \isom
S^l(E\dual)/\bigl(S^{l-2}(E\dual)\tensor s_q\bigr)$.  Dualizing the
diagram \eqref{eq:B2}, we see that $\CliffKuz_2\dual \subset
T^2(E\dual)$ is the ideal of relations of degree 2 in $\CoordKuz$.  In
fact, the quotient $T(E\dual) \to \CoordKuz$ has ideal of relations
generated in by $\CliffKuz_2\dual$.  Hence $\CoordKuz$ is the
quadratic dual $\OO_S$-algebra to $\CliffKuz$ in the terminology of
\cite[Ch.~1,~\S2,~Def.~1]{poli-posi}.  In fact, both $\CoordKuz$ and
$\CliffKuz$ are Koszul algebras (cf.\
\cite[Ch.~2,~\S1,~Def.1]{poli-posi}), which can be proved as in
\cite[\S6]{poli-posi}.

\medskip

The following list of properties of the even Clifford algebra are all
proved by descent, using the corresponding classical properties of
even Clifford algebra of $\OO_S$-valued forms, see
\cite[\S1.8]{auel:clifford}.

\begin{proposition}
\label{prop:properties}
Let $(E,q,L)$ be a quadratic form of rank $n$ on $S$.
\begin{enumerate} \setlength{\itemsep}{2pt}

\item \label{C_0-center} Assume that $(E,q,L)$ is (semi)regular.  If
$n$ is odd, then $\CliffAlg_0(E,q,L)$ is a central $\OO_S$-algebra.
If $n$ is even, then the center $\CliffZ(E,q,L)$ of
$\CliffAlg_0(E,q,L)$ is an \'etale quadratic $\OO_S$-algebra.

\item \label{C_0-Azumaya} If $(E,q,L)$ is (semi)regular and $n$ is
odd, then $\CliffAlg_0(E,q,L)$ is an Azumaya $\OO_S$-algebra.  If
$(E,q,L)$ is regular and $n$ is even, then $\CliffAlg_0(E,q,L)$ is an
Azumaya algebra over its center.

\item \label{C_0-similarity}
Any similarity $(\vp, \lambda) : (E,q,L) \to (E',q',L')$ induces
an $\OO_S$-algebra isomorphism
$$
\CliffAlg_0(\vp, \lambda) : \CliffAlg_0(E,q,L) \to
\CliffAlg_0(E',q',L').
$$

\item \label{C_0-proj_similarity}
Any $\OO_S$-module isomorphism $\phi : N^{\tensor 2}\tensor L
\to L'$ induces an $\OO_S$-algebra isomorphism
$$
\CliffAlg_0(N\tensor E,\phi \circ (q_{N}\tensor q),L')
\to \CliffAlg_0(E,q,L)
$$
where $q_N : N \to N^{\tensor 2}$ is the canonical quadratic form.

\item \label{C_0-functoriality}
 For any morphism of schemes $g : S' \to S$, there is a canonical
$\OO_S$-module isomorphism
$$
g\pullback\CliffAlg_0(E,q,L) \to \CliffAlg_0(g\pullback(E,q,L)).
$$
\end{enumerate}
\end{proposition}
\begin{proof}
For \eqref{C_0-center} and \eqref{C_0-Azumaya}, see
\cite[IV~Thm.~2.2.3,~Prop.~3.2.4]{knus:quadratic_hermitian_forms} or
\cite[Thm.~3.7]{bichsel_knus:values_line_bundles}.  For
\eqref{C_0-similarity} and \eqref{C_0-proj_similarity}, see
\cite[IV~Props.~7.1.1,~7.1.2]{knus:quadratic_hermitian_forms}.  The
tensorial nature of the construction immediately implies
\eqref{C_0-functoriality}.
\end{proof}

The following list of properties of the Clifford bimodule are all
proved by descent, using the corresponding classical properties of odd
Clifford algebra of $\OO_S$-valued forms, see
\cite[\S4.1]{auel:clifford}.

\begin{proposition}
\label{prop:properties_C_1}
Let $(E,q,L)$ be a quadratic form of rank $n$ on $S$.
\begin{enumerate}  \setlength{\itemsep}{2pt}
\item \label{C_1-C_0-module}
Via multiplication in the tensor algebra, there is a canonical
morphism
$$
m : \kc_1(E,q,L) \otimes_{\kc_0(E,q,L)}
\kc_1(E,q,L) \to \kc_0(E,q,L)\tensor_{\OO_S}L 
$$
of $\kc_0(E,q,L)$-bimodules (with trivial action on $L$).
If $(E,q,L)$ is primitive then $\CliffAlg_1(E,q,L)$ is an invertible
$\CliffAlg_0(E,q,L)$-bimodule and the map $m$ is an isomorphism.

\item \label{C_1-similarity}
Any similarity transformation $(\vp, \lambda) : (E,q,L) \to
(E',q',L')$ induces an $\OO_S$-module isomorphism
$$
\kc_1(\vp, \lambda) : \kc_1(E,q,L) \to
\kc_1(E',q',L').
$$
that is $\kc_0(\vp,\lambda)$-semilinear according to the diagram
$$
\xymatrix{
\kc_1(E,q,L) \otimes_{\kc_0(E,q,L)}
\kc_1(E,q,L)\ar[d]^{\kc_1(\vp,\lambda) \tensor
\kc_1(\vp,\lambda)}\ar[r]^(.6)m&\kc_0(E,q,L)\tensor_{\OO_S}L
\ar[d]^{\kc_0(\vp,\lambda)\tensor\lambda}\\ 
\kc_1(E',q',L') \otimes_{\kc_0(E',q',L')}
\kc_1(E',q',L')\ar[r]^(.6)m&\kc_0(E',q',L')\tensor_{\OO_S}L'
}
$$

\item \label{C_1-projsim} Any $\OO_S$-module isomorphism $\phi : N^{\tensor 2}\tensor L
\to L'$ induces an $\OO_S$-module isomorphism
$$
\CliffAlg_1(N\tensor E,\phi \circ (q_{N}\tensor q),L')
\to N \tensor \CliffAlg_1(E,q,L).
$$

\item \label{C_1-functoriality} For any morphism of schemes $g : S' \to S$, there is a canonical
$\OO_S$-module isomorphism
$$
g\pullback\kc_1(E,q,L) \to \kc_1(g\pullback(E,q,L)).
$$
\end{enumerate}
\end{proposition}
\begin{proof}
For \eqref{C_1-C_0-module}, see \cite[\S2]{bichsel:thesis} or
\cite[Lemma~3.1]{bichsel_knus:values_line_bundles}.  For a proof of
the final statement in \eqref{C_1-C_0-module}, we have the following
local calculation:\ if $(E,q,L)$ is a primitive quadratic form over a
local ring then there exists a line subbundle $N \subset E$ such that
$q|_N$ is regular, and in this case, $N$ generates $\CliffAlg_1$ over
$\CliffAlg_0$.  For \eqref{C_1-similarity}, see
\cite[Prop.~2.6]{bichsel:thesis}.  For \eqref{C_1-projsim}, we can
appeal to \cite[Lemma~3.1]{bichsel_knus:values_line_bundles}.  The
tensorial nature of the construction immediately implies
\eqref{C_1-functoriality}.
\end{proof}

\section{The equality of some Brauer classes related to quadric fibrations}
\label{subsec:grassmannian}

Our perspective has been to use the even Clifford algebra as an
algebraic way of getting at the Brauer class associated to a quadric
fibration.  Other authors, notably in \cite{hassett:rational_cubic},
\cite{hassett_varilly:K3}, \cite{geeman:K3}, prefer to use geometric
manifestations of Brauer classes.  One of these geometric
manifestations involves the Stein factorization of the relative
lagrangian grassmannian.  Over a field (more generally, when the
quadric fibration is regular), it's a classical fact that in
dimensions 2 and 4, the Brauer classes arising from these two methods
coincide, see \cite[XVI~Ex.~85.4]{elman_karpenko_merkurjev}.  This is
proved by appealing to the exceptional isomorphisms of projective
homogeneous varieties associated to algebraic groups of low rank, as
in \cite[\S15]{book_of_involutions}. The aim of this section is
provide a proof for quadric fibrations with simple degeneration.  For
quadric surface bundles, this fact has been noted in
\cite[Lemma~4.2]{kuznetsov:cubic_fourfolds} and
\cite[Thm.~4.5]{ingalls-khalid}.

Let $(E,q,L)$ be a quadratic form of rank $n$ on a scheme $S$.  We
recall the standard moduli theoretic description of the
\linedef{isotropic grassmannian fibration} $\LG_r(q) \to S$.

\begin{theorem}
\label{thm:grass_moduli}
For each integer $0 \leq r \leq \lceil n/2 \rceil$, the 
$S$-scheme $\LG(q)$ represents the 
the moduli subfunctor
$$
u : U \to S \quad \mapsto \quad \left\{ W \mapto{v} u^* E \; : \;
\text{$W$ has rank $r+1$ and $u^*q|_W = 0$} \right\} 
$$
of the grassmannian of rank $r+1$ sub vector bundles of $E$.
\end{theorem}

Note that $\LG_0(q)=Q$ recovers Theorem~\ref{thm:moduli_quadric}.
When $n$ is even, the $\LG(q) = \LG_{n/2}(q)$ is called the
\linedef{lagrangian grassmannian fibration} associated to $(E,q,L)$.

We recall that the \linedef{Stein factorization} of a proper morphism
$p : X \to S$ of schemes is a decomposition $X \mapto{r} Z \mapto{f}
S$ such that $r\pushforward \OO_X \isom \OO_Z$ and $f$ is affine, and
which satisfies the following universal property:\ for any other
factorization $X \to Z' \mapto{f'} S$ with $f'$ affine, there exists a
unique morphism $\alpha : Z \to Z'$ such that $f' \circ \alpha = f$.
The Stein factorization may be constructed by taking $Z = \SSpec
p\pushforward\OO_X$.  Note that the following functoriality property
holds:\ given a proper morphism $p' : X' \to S$ and an $S$-morphism $j
: X' \to X$, there exists a commutative diagram
$$
\xymatrix@C=40pt@R=18pt{
X' \ar[d]_j \ar[r]^{r'} & Z' \ar[d]_g \ar[r]^{f'} & S \ar@{=}[d] \\
X  \ar[r]^{r} & Z \ar[r]^f & S 
}
$$
of Stein factorizations.

Let $f : T \to S$ be the discriminant cover, and $Q \to S$ be the
quadric fibration, associated to $(E,q,L)$.  Let  
\begin{equation}
\label{eq:stein}
\LG(q) \mapto{r'} T' \mapto{f'} S
\end{equation} 
be the Stein factorization of the morphism $\LG(q) \to S$.  In
\cite[\S3.2]{hassett_varilly:K3}, it's proved that if $Q \to S$ has
simple degeneration, then $r'$ is smooth and $f'$ is finite flat of
degree 2.  By the classical theory, if $n=4$ or $n=6$, then $r'$ is an
\'etale locally trivial projective space bundle of relative dimensions
1 and 2, respectively.  

Our aim is to show that in these cases, the projective space bundle
$r'$ and the Azumaya algebra $\CliffB_0$ (the even Clifford algebra
considered over the discriminant cover, see
\S\ref{subsec:clifford_discriminant}) define the same Brauer class.
We must first verify that $f' : T' \to S$ is indeed isomorphic to the
discriminant cover, a result that seems to be often used in the
literature, though no complete proof could be readily found.

\begin{proposition}
\label{prop:isocovers}
Let $S$ be a regular scheme and $(E,q,L)$ a quadratic form of even
rank $n$ on $S$. Assume that $(E,q,L)$ has simple degeneration along a
regular divisor $D$ of $S$.  Then there is a canonical $S$-isomorphism
$T' \to T$.
\end{proposition}
\begin{proof}
We first construct the morphism $T' \to T$, following
\cite[\S85]{elman_karpenko_merkurjev}.  For each $u : U \to S$ and
each lagrangian submodule $v : W \to u\pullback E$, there is a
canonical $\OO_U$-module morphism $\psi : \det W \tensor L^{\vee
\tensor n/4} \to \CliffAlg_0(E,q,L)$ (if $n \equiv 0 \bmod 4$) or
$\psi : \det W \tensor L^{\vee \tensor (n-2)/4} \to
\CliffAlg_1(E,q,L)$ (if $n \equiv 2 \bmod 4$).  But the action of the
center $u\pullback\CliffZ$ stabilizes the image of $\psi$, hence there
is an $\OO_U$-algebra morphism $u\pullback \CliffZ \to
\EEnd_{\OO_U}(\det W) \isom \OO_U$, hence an $\OO_S$-algebra morphism
$\CliffZ \to u\pushforward \OO_U$, hence a morphism $U \to T$.  By the
functoriality of the even Clifford algebra and Clifford bimodule (cf.\
Propositions~\ref{prop:properties}(\ref{C_0-functoriality}) and
\ref{prop:properties_C_1}(\ref{C_1-functoriality})), we've just
defined a morphism of moduli functors $\LG(q) \to T$ over $S$.  Since
$T \to S$ is affine, the universal property of Stein factorization
provides an $S$-morphism $t : T' \to T$.  Note that under our
hypotheses, both $T$ and $T'$ are finite flat of degree 2 over $S$.

Since $T$ and $T'$ are affine over $S$, to argue that $t : T' \to T$
is an isomorphism we may assume that $S$ is the spectrum of a local
ring $(R,\maxideal)$.  Note that $t$ is equivariant for the Galois
action since any reflection of $(E,q,L)$ induces the nontrivial
$S$-isomorphism of both $T$ and $T'$.

First note that if $(E,q,L)$ is regular over $S$, then both $T$ and
$T'$ are \'etale of degree 2 and hence, as both correspond to torsors
for the group $\Z/2\Z$, any equivariant morphism between them is an
isomorphism.  Thus we may assume that $(E,q,L)$ has simple
degeneration along a regular divisor $D$ of $S$ generated by a
nonsquare nonzero divisor $\pi \in \maxideal$ and (modifying up to
units) that $\pi$ is the discriminant of $q$.

Now we proceed by induction on the rank.  If $(E,q,L)$ has rank 2,
then (multiplying by units, which does not change the quadric) it can
be diagonalized as $q = <\! 1, -\pi\!>$.  Now $\LG(q) = \Proj
R[x,y]/(x^2-\pi y^2)$ and note that the standard affine patch $U_y$
(where $y$ is invertible) is canonically isomorphic to $\Spec
R[x]/(x^2-\pi)$ which is precisely $T$.  Finally, note that $U_x
\subset U_y$ since $\pi$ is a nonzero divisor, so that $\LG(q) \isom
T$.

Now assume that the proposition holds in rank $n \geq 2$.  Let
$(E,q,L)$ have rank $n+2$.  We first reduce to the case that $q$ is
isotropic.  In general, there exists an finite \'etale extension $S'
\to S$ over which $q$ becomes isotropic since $q$ has simple
degeneration and rank $> 2$, it has a (semi)regular subform of rank
$>1$, which acquires a zero over a finite \'etale extension.  Then if
$t : T' \to T$ becomes an isomorphism over $S'$, it was an
isomorphism.

Now assume that $(E,q,L)$ is isotropic.  Then since $D$ and $S$ are
regular, the associated quadric fibration is regular, and any
isotropic line $N \subset E$ is regular (see
Lemma~\ref{regsect}). Hence we have a decomposition $(E,q,L) = H_L(N)
\perp (E',q',L)$ corresponding to quadric reduction (see
\S\ref{subsec:hyperbolic_algebraic}), with $q'$ having the same
degeneration as $q$ by Corollary~\ref{cor:degeneration_reduction}.
There is a closed embedding $j : \LG(q') \to \LG(q)$ defined by $W
\mapsto l\inv(W)$ where $l : N^{\perp} \to E'$ is the quotient map.
Thus there is a commutative diagram
$$
\xymatrix@C=40pt@R=18pt{
\LG(q') \ar[d]^j\ar[r]^(.55){r''}& T'' \ar[d]_g \ar[r]^{f''} & S \ar@{=}[d]\\
\LG(q)  \ar[r]^(.55){r'} & T' \ar[r]^{f'} & S 
}
$$
of Stein factorizations.  

We claim that $g : T'' \to T'$ is an isomorphism.  Pushing forward
the ideal sheaf $\ki$ of the embedding $j$, we arrive at an exact
sequence
$$
0 \to r\pushforward' \ki \to \OO_{T'} \to
g\pushforward \OO_{T''} \to R^1 r\pushforward' \ki \to 0
$$
which is exact at right since $R^1 r\pushforward' \OO_{\LG(q)}$
vanishes (equivalently, $f\pushforward' R^1r\pushforward' \OO_{\LG(q)}
\isom R^1 p\pushforward' \OO_{\LG(q)}$ vanishes, since $f'$ is
affine).  Indeed, the fibers of $r'$ are projective homogeneous
varieties whose structure sheaves have no higher cohomology by Kempf's
vanishing theorem \cite{kempf:vanishing} (cf.\
\cite[Proof~of~Prop.~3.3]{hassett_varilly:K3}).  As the generic fiber
of $R^ir\pushforward' \ki$ is trivial for $i=0,1$ (again, since the
generic fiber of $T'' \to T'$ is an equivariant morphism of \'etale
quadratic covers, it is an isomorphism), these are torsion sheaves.
In particular, $r\pushforward' \ki = 0$ since it is torsion-free.  By
\'etale localization, we are reduced to the following:\ if $q$ is a
quadratic form with simple degeneration (i.e., radical of rank 1) over
a separably closed field $k$, then the affine part of the Stein
factorization of $\LG(q)$ is isomorphic to $\Spec
k[\varepsilon]/(\varepsilon^2)$.  This follows by a geometric
argument:\ considering a flat family with smooth generic fiber and
special fiber $q$, the two connected components of the lagrangian
grassmannian over the generic fiber come together in the special
fiber.  Then, since any injective algebra endomorphism between rings
of dual numbers is an isomorphism, the special fiber of $\OO_{T'} \to
g\pushforward\OO_{T''}$ is an isomorphism. Thus $R^1 r\pushforward'\ki
= 0$ and we have proved the claim.

By Proposition~\ref{prop:split_hyperbolic_case}, $\CliffAlg_0(E,q,L)$
and $\CliffAlg_0(E',q',L)$ have isomorphic centers, hence the
discriminant cover of $q'$ is isomorphic to $T \to S$.  But now by the
induction hypothesis, the induced morphism $T'' \to T$ is an
isomorphism, hence by the above, $T' \to T$ is an isomorphism.
\end{proof}

\begin{proposition}\label{brauerclasses}
Let $S$ be a regular integral scheme.  Let $Q\to S$ be a quadric
fibration of relative dimension 2 or 4 with simple degeneration along
a regular divisor and discriminant double cover $T \to S$.  Then
there is an equality classes in $\Br(T) \isom \Br(T')$:
\begin{itemize}
\item The class $\lambda$ of the Severi--Brauer scheme appearing as
the connected part of the Stein factorization $\LG(Q) \to T' \to S$.

\item The class $\beta$ arising from the even Clifford algebra $\CliffB_0$
on the discriminant cover $T\to S$.
\end{itemize}
\end{proposition}
\begin{proof}
Since $S$ and $D$ are regular, $T$ will be regular. Since $S$ is
integral and $D$ is nonempty, $T$ will be regular.  Denote by $L$ the
function field of $T$.  Thus $\Br(T) \to \Br(L)$ is injective (see
\cite{auslander_goldman:brauer_group_commutative_ring} or
\cite[II~Cor.~1.10]{grothendieck:Brauer}).  Identifying $\Br(T') \isom
\Br(T)$ by Proposition \ref{prop:isocovers}, we need only compare the
two Brauer classes $\lambda$ and $\beta$ at the generic point, where
this statement is classical (see
\cite[Exer.~85.4]{elman_karpenko_merkurjev} for example).
\end{proof}

It would be interesting to give a direct explicit isomorphism between
the Severi--Brauer scheme of $\CliffB_0$ and the connected part of the
Stein factorization of $\LG(Q)$.  See \cite[Thm.~4.5]{ingalls-khalid}
for an approach in the case of surface bundles.  This would seem most
naturally accomplished via the moduli space interpretations of the two
objects.


\begin{thebibliography}{10}

\bibitem{SGA3-3}
\emph{Sch\'emas en groupes. {III}: {S}tructure des sch\'emas en groupes
  r\'eductifs}, S\'eminaire de G\'eom\'etrie Alg\'ebrique du Bois Marie 1962/64
  (SGA 3). Dirig\'e par M. Demazure et A. Grothendieck. Lecture Notes in
  Mathematics, Vol. 153, Springer-Verlag, Berlin, 1962/1964.

\bibitem{alekseev:dP4}
V.A.\ Alekseev, \emph{On conditions for the rationality of three-folds with a
  pencil of del {P}ezzo surfaces of degree {$4$}}, Mat. Zametki \textbf{41}
  (1987), no.~5, 724--730, 766.

\bibitem{amer:theorem}
M.\ Amer, \emph{Quadratische Formen uber Funktionenkorpern}, unpublished dissertation, Johannes Gutenberg Universitat, Mainz (1976).

\bibitem{an-au-ga-za}
A.Ananyevskiy, A. Auel, S.Garibaldi, and K.Zainoulline
\emph{Exceptional colletcions of ilne bundles on projective
homogeneous varieties}, Adv.\ Math.\ \textbf{236} (2013), 111--130. 


\bibitem{auel:clifford}
A.\ Auel, \emph{Clifford invariants of line bundle-valued quadratic
forms}, MPIM preprint series \textbf{33}, 2011.

\bibitem{auel:Brdim}
\bysame, \emph{Surjectivity of the total Clifford invariant and Brauer
dimension}, preprint arXiv:1108.5728.
 
\bibitem{auel_parimala_suresh}
A.\ Auel, R.\ Parimala, V.\ Suresh,
\emph{Quadric surface bundles over surfaces}, preprint arXiv:1207.4105.

\bibitem{auslander_goldman:brauer_group_commutative_ring}
M.\ Auslander and O.\ Goldman,
\emph{The {B}rauer group of a commutative ring},
Trans.\ Amer.\ Math.\ Soc.\ \textbf{97} (1960), 367--409.

\bibitem{baeza:semilocal_rings}
R.\ Baeza, \emph{Quadratic forms over semilocal rings}, Lecture Notes in
  Mathematics, Vol. 655, Springer-Verlag, Berlin, 1978.

\bibitem{balmer_calmes:lax}
P.\ Balmer and B.\ Calm{\`{e}}s, \emph{Bases of total {W}itt groups and
  lax-similitude}, J.\ Algebra Appl.\ \textbf{11}, 1250045 (2012), no.\ 3, 24 pp.

\bibitem{bayer_fainsilber}
E.\ Bayer-Fluckiger and L.\ Fainsilber, 
\emph{Non-unimodular Hermitian forms},
Invent.\ Math.\ \textbf{123} (1996), no.\ 2, 233--240. 


\bibitem{beauvillejaco}
A. Beauville,
{\em Vari\'et\'es de Prym et jacobiennes interm\'ediaires},
Ann. Scient. ENS {\bf 10} (1977), 309--391.\bibitem{berna_macri_mehro_stella}

M.\ Bernardara, E.\ Macr{\`\i}, S.\ Mehrotra, and P.\ Stellari,
\emph{A categorical invariant for cubic threefolds}, 
Adv.\ Math.\ \textbf{229} (2012), no.\ 2, 770--803.

\bibitem{bolognesi_bernardara:conic_bundles}
M.\ Bernardara and M.\ Bolognesi, \emph{Derived categories and
  rationality of conic bundles}, preprint arXiv:1010.2417v3, 2010,
  Compositio Math., to appear.

\bibitem{bolognesi_bernardara:representability}
\bysame, \emph{Categorical representability
 and intermediate Jacobians of Fano threefolds}, in \emph{Derived categories in algebraic geometry},
EMS Series of Congress Reports, 1--25 (2012).

\bibitem{bertin:trisecants}
M.-A.\ Bertin, \emph{On the singularities of the trisecant surface to
a space curve}, 
Le Matematiche (Catania) \textbf{53} (1999), 15--22.   

\bibitem{bichsel:thesis}
W.\ Bichsel, \emph{Quadratische R\"aume mit Werten in invertierbaren Moduln},
  Ph.D. thesis, ETH Z{\"{u}}rich, 1985.

\bibitem{bichsel_knus:values_line_bundles}
W.\ Bichsel and M.-A.\ Knus, \emph{Quadratic forms with values in line bundles},
  Contemp.\ Math. \textbf{155} (1994), 293--306.

\bibitem{bhargava:ICM}
M.\ Bhargava, \emph{Higher composition laws and applications},
International {C}ongress of {M}athematicians.\ {V}ol.\ {II},
Eur.\ Math.\ Soc., Z\"urich, 2006, 271--194.

\bibitem{bondal:representations}
A.I.\ Bondal, \emph{Representations of associative algebras and coherent
  sheaves}, Izv.\ Akad.\ Nauk SSSR Ser.\ Mat.\ \textbf{53} (1989),
  no.\ 1, 25--44.

\bibitem{bondal_kapranov:reconstructions}
A.I.\ Bondal and M.M.\ Kapranov, \emph{Representable functors, {S}erre
  functors, and reconstructions}, Izv.\ Akad.\ Nauk SSSR Ser.\ Mat.\ \textbf{53}
  (1989), no.\ 6, 1183--1205, 1337.

\bibitem{bondal_orlov:semiorthogonal}
A.I.\ Bondal and D.O.\ Orlov, \emph{Semiorthogonal decomposition for
  algebraic varieties}, MPIM preprint 1995-15, 1995.

\bibitem{bondal_orlov:ICM2002}
\bysame,
\emph{Derived categories of coherent sheaves},
Proceedings of the International Congress of Mathematicians, Vol. II
(Beijing, 2002), 47--56, Higher Ed.\ Press, Beijing, 2002. 

\bibitem{bondal_vdB}
A.I.\ Bondal, and M.\ van den Bergh,
\emph{Generators and representability of functors in commutative and noncommutative geometry},
Mosc.\ Math.\ J.\ \textbf{3} (2003), no.\ 1, 1--36, 258.
 
\bibitem{bridgeland:triangulated}
T.\ Bridgeland, \emph{Equivalences of triangulated categories and
  {F}ourier-{M}ukai transforms}, Bull.\ London Math.\ Soc.\ \textbf{31} (1999),
  no.\ 1, 25--34.


\bibitem{cadman}
C.\ Cadman,
\emph{Using stacks to impose tangency conditions on curves},
Amer.\ J.\ Math.\ \textbf{129} (2007), no.\ 2, 405--427.

\bibitem{caldararu:thesis}
A.\ C{\u{a}}ld{\u{a}}raru, \emph{Derived categories of twisted sheaves on
  {C}alabi-{Y}au manifolds}, Ph.D. thesis, Cornell University, Ithica, NY, May
  2000.

\bibitem{campana_peternell_pukhlikov}
F.\ Campana, T.\ Peternell, and A.V.\ Pukhlikov,
\emph{The generalized {T}sen theorem and rationally connected {F}ano
              fibrations}, Mat.\ Sb.\ \textbf{193} (2002), no.\ 10, 49--74.

\bibitem{caenepeel_van_oystaeyen}
S.\ Caenepeel and F.\ van Oystaeyen, \emph{Quadratic forms with values
  in invertible modules}, {K}-{T}heory \textbf{7} (1993), 23--40.

\bibitem{canonaco-stellari}
A.\ Canonaco and P.\ Stellari,
\emph{Twisted Fourier-Mukai functors}, {A}dv.\ {M}ath.\ 212 (2007), 484--503.

\bibitem{colliot-thelene_levine}
J.-L.\ Colliot-Th{\'e}l{\`e}ne and M.\ Levine, \emph{Une version du
th{\'e}or{\`e}me d'Amer et Brumer pour les z{\'e}ro-cycles},
Quadratic forms, linear algebraic groups, and cohomology, Dev.\ Math.,
vol.\ 18, Springer, New York, 2010, pp.\ 215--223.

\bibitem{colliot-thelene_sansuc_swinnerton-dyer:quadrics_I}
J.-L.\ Colliot-Th{\'e}l{\`e}ne, J.-J.\ Sansuc, and P.\
  Swinnerton-Dyer, \emph{Intersections of two quadrics and {C}h{\^a}telet
  surfaces. {I}}, J.\ Reine Angew.\ Math.\ \textbf{373} (1987), 37--107.


\bibitem{delign_mumford:irreducibility}
P.\ Deligne and D.\ Mumford,
\emph{The irreducibility of the space of curves of given genus},
Inst.\ Hautes \'Etudes Sci.\ Publ.\ Math.\ \textbf{36} (1969), 75--109. 

\bibitem{demazure_gabriel}
M.\ Demazure and P.\ Gabriel, \emph{Groupes alg{\'{e}}briques. {T}ome
  {I}: G{\'{e}}om{\'{e}}trie alg{\'{e}}brique, g{\'{e}}n{\'{e}}ralit{\'{e}}s,
  groupes commutatifs. {A}vec un appendice, \emph{{C}orps de classes local},
  par {M}ichiel {H}azewinkel}, Masson {\&} Cie, {\'{E}}diteur, Paris,
  North--Holland Publishing Company, Amsterdam, 1970.

\bibitem{desale-ramanan}
U.V.\ Desale and S.\ Ramanan,
{\em Classification of vector bundles of rank 2 on hyperelliptic curves},
Invent.\ Math.\ {\bf 38} (1976), 161--185.

\bibitem{delone_faddeev}
B.N.\ Delone and D.K.\ Faddeev,
\emph{The theory of irrationalities of the third degree},
Translations of Mathematical Monographs \textbf{10}, American Mathematical Society, Providence, R.I, 1964. 

\bibitem{elman_karpenko_merkurjev}
R.\ Elman, N.\ Karpenko, and A.\ Merkurjev, \emph{The algebraic
  and geometric theory of quadratic forms}, American Mathematical Society
  Colloquium Publications, vol.\ 56, American Mathematical Society, Providence,
  RI, 2008.


\bibitem{fulton:flag_bundles}
\bysame, \emph{Schubert varieties in flag bundles for the classical
  groups}, Proceedings of Conference in Honor of {H}irzebruch's 65th Birthday,
  {B}ar {I}lan, 1993, vol.\ 9, Israel Mathematical Conference Proceedings, 1995.

\bibitem{geeman:K3}
B.\ van Geemen, \emph{Some remarks on Brauer groups of K3 surfaces},
Adv.\ in Math., \textbf{197} (2005), 222-247. 

\bibitem{giraud}
J.\ Giraud, \emph{{C}ohomologie non ab{\'e}lienne},
{S}pringer-{V}erlag, Berlin, 1971.

\bibitem{graber_harris_starr}
T.\ Graber, J.\ Harris, and J.\ Starr, \emph{Families of rationally
  connected varieties}, J.\ Amer.\ Math.\ Soc.\ \textbf{16} (2003),
  no.\ 1, 57--67.

\bibitem{gross_lucianovic}
B.\ Gross and M.\ Lucianovic,
\emph{On cubic rings and quaternion rings},
J.\ Number Theory \textbf{129} (2009), no.\ 6, 1468--1478.

\bibitem{EGA4}
A.\ Grothendieck, \emph{\'{E}l\'ements de g\'eom\'etrie alg\'ebrique. {IV}.
  \'{E}tude locale des sch\'emas et des morphismes de sch\'emas. {I}, {II},
  {III}, {IV}}, Inst. Hautes \'Etudes Sci. Publ. Math. (1964, 1965, 1966,
  1967), no.~20, 24, 28, 32.


\bibitem{grothendieck:Brauer}
A.\ Grothendieck, \emph{Le groupe de {B}rauer. {I}, {II}, {III}. {T}h\'eorie
  cohomologique}, Dix {E}xpos\'es sur la {C}ohomologie des {S}ch\'emas,
  North--Holland, Amsterdam, 1968, pp.\ 67--87.


\bibitem{hartshorne:algebraic_geometry}
R.\ Hartshorne, \emph{{A}lgebraic geometry}, {G}raduate {T}exts {M}ath.,
  vol.~52, {S}pringer-{V}erlag, {N}ew {Y}ork, 1977.

\bibitem{hassett:rational_cubic}
B.\ Hassett, \emph{Some rational cubic fourfolds}, J.\ Algebraic
Geometry {\bf 8} (1999), no.\ 1, 103--114. 

\bibitem{hassett_varilly:K3}
B.\ Hassett, P.\ Varilly, and A.\ V{\'a}rilly-Alvarado,
  \emph{Transcendental obstructions to weak approximation on general {$K3$}
  surfaces}, Adv.\ in Math., \textbf{228} (2011), 1377-1404.

\bibitem{huybrechts-book}
D. Huybrechts,
{\em Fourier-{M}ukai transforms in {A}lgebraic {G}eometry}.
Oxford Math. Monongraphs (2006).

\bibitem{huybrechts-lehn}
D.\ Huybrechts and M.\ Lehn,
{\em The geometry of moduli spaces of sheaves},
Aspects of Mathematics, E31, Braunschweig 1997.

\bibitem{huybrechts-stellari}
D.\ Huybrechts and P.\ Stellari, 
\emph{Equivalences of twisted K3 surfaces}, 
Math.\ Ann.\ \textbf{332} (2005), no.\ 4, 901--936.

\bibitem{ingalls-khalid}
C.\ Ingalls and M.\ Khalid,
{\em Derived equivalences of Azumaya algebras on K3 surfaces}, 
preprint arXiv:1104.4333, 2011.

\bibitem{iskovconicduke}
V.A. Iskovskikh,
{\em On the rationality problem for conic bundles},
Duke Math. J. {\bf 54} (1987), 271--294.

\bibitem{isko_fano}
V.A.\ Iskovskikh and Y.\ Prokhorov,\emph{Algebraic geometry. V. 
Fano varieties}, A translation of Algebraic geometry. 5 (Russian), Ross.\ Akad.\ Nauk, Vseross.\ Inst.\ Nauchn.\ i Tekhn.\ Inform., Moscow. Translation edited by A.N.\ Parshin and I.R.\ Shafarevich. Encyclopaedia of Mathematical Sciences, \textbf{47}. Springer-Verlag, Berlin, 1999. 

\bibitem{kapranov:quadric}
M.M.\ Kapranov, \emph{The derived category of coherent sheaves on a
quadric}, Funkcionalniy analiz i ego pril., \textbf{20} (1986), 67.

\bibitem{kapranov:intersection_quadrics}
\bysame, \emph{On the derived category and $K$-functor of coherent sheaves on intersections of quadrics}, Math.\ USSR Izv.\ \textbf{32} (1989), 191--204. 

\bibitem{kapranov:derived}
\bysame, \emph{On the derived categories of coherent sheaves on
some homogeneous spaces}, Invent.\ Math., \textbf{92} (1988), 479--508.

\bibitem{kko}
A.\ Kapustin, A.\ Kuznetsov, and D.\ Orlov, \emph{Noncommutative instantons and twistor transform}, Comm. Math. Phys. \textbf{221}
(2001), no. 2, 385--432.

\bibitem{kashiwara_schapira:categories_sheaves}
M.\ Kashiwara and P.\ Schapira, \emph{Categories and sheaves},
  Grundlehren Math.\ Wiss., vol.\ 332, Springer-Verlag, Berlin, 2006.

\bibitem{kempf:vanishing}
G.R.\ Kempf,
\emph{Linear systems on homogeneous spaces},
Ann.\ of Math.\ (2) \textbf{103}, no.\ 3, 557--591.


\bibitem{knus:quadratic_hermitian_forms}
M.-A.\ Knus, \emph{Quadratic and hermitian forms over rings},
  {S}pringer-{V}erlag, {B}erlin, 1991.

\bibitem{book_of_involutions}
M.-A.\ Knus, A.\ Merkurjev, M.\ Rost, and J.-P.\ Tignol,
  \emph{The book of involutions}, Colloquium Publications, vol.\ 44, AMS, 1998.

\bibitem{knus_ojanguren:metabolic}
M.-A.\ Knus and M.\ Ojanguren, \emph{The {C}lifford algebra of a
  metabolic space}, Arch.\ Math.\ (Basel) \textbf{56} (1991), no.\ 5, 440--445.

\bibitem{knus_parimala_sridharan:rank_4}
M.-A.\ Knus, R.\ Parimala, and R.\ Sridharan,
\emph{On rank 4 quadratic spaces with given {A}rf and {W}itt invariants},
Math.\ Ann.\ \textbf{274} (1986), no.\ 2, 181--198. 



\bibitem{kuznetsov:v14}
A.\ Kuznetsov,
{\em Derived categories of cubic and $V_{14}$ threefolds},
Proc.\ Steklov Inst.\ Math.\ {\bf 246}, 171--194 (2004); translation from Tr.\ Mat.\ Inst.\ Steklova {\bf 246}, 183--207 (2004). 

\bibitem{kuznetsov:v12}
\bysame,
{\em Derived categories of {F}ano threefolds $V_{12}$},
Math.\ Notes {\bf 78} (2005), no.\ 4, 537--550; translation from Mat.\ Zametki {\bf 78} (2005), no.\ 4, 579--594.

\bibitem{kuznetsov:hyp-sections}
\bysame,
{\em Hyperplane sections and derived categories},
Izv.\ Math.\ {\bf 70} (2006), no. 3, 447--547; translation from
Izv.\ Ross.\ Akad.\ Nauk, Ser.\ Mat.\ {\bf 70} (2006), no.\ 3, 23--128.

\bibitem{kuznetsov:hpd}
\bysame, \emph{Homological projective duality}, Publ.\ Math.\ Inst.\
  Hautes {\'E}tudes Sci.\ (2007), no.\ 105, 157--220.

\bibitem{kuznetsov:quadrics}
\bysame, \emph{Derived categories of quadric fibrations and intersections of
  quadrics}, Adv.\ Math.\ \textbf{218} (2008), no.\ 5, 1340--1369.

\bibitem{kuznetsov:cubic_fourfolds}
\bysame,
{\em Derived categories of cubic fourfolds},
in {\em Cohomological and geometric approaches to rationality problems},
Progr.\ Math.\ {\bf 282}, Birkh{\"a}user Boston, 163--208, 2010.

\bibitem{kuznetbasechange} \bysame, \emph{Base change for semiorthogonal decompositions}.
Compositio Math. {\bf 147} (2011), no.~3,
852--876.

\bibitem{lam:algebraic_theory_quadratic_forms}
T.Y.\ Lam, \emph{The algebraic theory of quadratic forms}, Benjamin/Cummings
  Publishing Co.\ Inc.\ Advanced Book Program, Reading, Mass., revised
  second printing, Mathematics Lecture Note Series, 1980.

\bibitem{champs-algebriques}
G.\ Laumon and L.\ Moret-Bailly, \emph{Champs algébriques}, Ergebnisse der Mathematik und ihrer Grenzgebiete.\ 3.\ Folge.\ Springer-Verlag, Berlin, 2000. 

\bibitem{leep:amer_brumer_arbitrary}
D.\ Leep, \emph{The Amer-Brumer Theorem over arbitrary fields},
preprint, 2007.

\bibitem{lieblich:thesis}
M.\ Lieblich, \emph{Moduli of twisted sheaves and generalized {A}zumaya
  algebras}, Ph.D.\ thesis, Massachusetts Institute of Technology, Cambridge,
  MA, June 2004.

\bibitem{lieblich:moduli_twisted_sheaves}
M.\ Lieblich, 
\emph{Moduli of twisted sheaves}, 
Duke Math.\ J.\ \textbf{138} (2007), no.\ 1, 23--118. 

\bibitem{lieblich:period-index}
\bysame,
\emph{Twisted sheaves and the period-index problem},
Compos.\ Math.\ \textbf{144} (2008), 1--31.


\bibitem{matsumura:commutative_ring_theory} 
H.~Matsumura,
\emph{Commutative ring theory}, 
Cambridge University Press, Cambridge,
1986, Translated from the Japanese by M.\ Reid.

\bibitem{nitsure:Quot_Hilbert}
N.\ Nitsure, 
\emph{Construction of {H}ilbert and {Q}uot schemes}, 
{Fundamental algebraic geometry},
{Math. Surveys Monogr.} \textbf{123}, 105--137, 
{Amer. Math. Soc.}, {Providence, RI}, 2005.6

\bibitem{ottaviani:spinor_bundles}
G.\ Ottaviani, \emph{Spinor bundles on quadrics}, Trans.\ Amer.\ Math.\ Soc.\
  \textbf{307} (1988), no.\ 1, 301--316.Cambridge
  Studies in Advanced Mathematics, vol.~8,

\bibitem{orlov:blowup}
D.O.\ Orlov, \emph{Projective bundles, monoidal transformations, and derived categories of coherent sheaves},
Russian Acad.\ Sci.\ Izv.\ Math. \textbf{41}, (1993), no.\ 1, 133--141 

\bibitem{orlovequivabel}
\bysame, {\em Derived categories of coherent sheaves on abelian varieties and equivalences between them}
(Russian) Izv. Ross. Akad. Nauk Ser. Mat. {\bf 66} (2002), no. 3, 131--158; translation in Izv. Math. {\bf 66} (2002), no. 3, 569--594.


\bibitem{parimala_sridharan:norms_and_pfaffians}
R.\ Parimala and R.\ Sridharan, \emph{Reduced norms and pfaffians via
  {B}rauer-{S}everi schemes}, Contemp.\ Math.\ \textbf{155} (1994), 351--363.

\bibitem{pfister:amer}
A.\ Pfister, \emph{Quadratic forms with applications to algebraic geometry and
topology}, Cambridge University Press, (1995).

\bibitem{poli-posi}
A.\ Polishchuk and L.\ Positselski, \emph{Quadratic algebras}, University Lecture Series, \textbf{37}. American Mathematical Society, Providence, RI, 2005. xii+159 pp.

\bibitem{reid:thesis}
M.\ Reid, \emph{The intersection of two quadrics}, Ph.D.\
  thesis, Trinity College, Cambridge, June 1972.

\bibitem{sarkisov-bira}
V.G.\ Sarkisov,
{\em Birational automorphisms of conic bundles},
Math.\ USSR, Izv.\ {\bf 17} (1981), no.\ 1, 177--202.

\bibitem{scharlau}
W.\ Scharlau,
\emph{Quadratic and Hermitian forms},
Grundlehren Math.\ Wiss., vol.\ 270. Springer-Verlag, Berlin, 1985.

\bibitem{shokuprym}
V.V. Shokurov,
{\em Prym varieties: theory and applications},
Math. USSR-Izv. {\bf 23} (1984), 83--147.

\bibitem{shramov:dp4}
K.A.\ Shramov, \emph{On the rationality of non-singular threefolds with a pencil of Del Pezzo
    surfaces of degree 4}, Sb.\ Math., \textbf{197} (2006), no.\ 1, 127--137.


\bibitem{swan:quadric_hypersurfaces}
R.G.\ Swan, \emph{{$K$}-theory of quadric hypersurfaces}, Ann.\ of Math.\
  (2) \textbf{122} (1985), no.\ 1, 113--153.

\bibitem{balaji_ternary}
T.E.\ Venkata Balaji,
\emph{Line-bundle valued ternary quadratic forms over schemes},
J.\ Pure Appl.\ Algebra \textbf{208} (2007), 237--259.

\bibitem{voight:quaternion_rings}
J.\ Voight,
\emph{Characterizing quaternion rings over an arbitrary base},
J.\ reine angew.\ Math.\ \textbf{657} (2011), 113--134.

\bibitem{wood:binary}
M.M.\ Wood,
\emph{Gauss composition over an arbitrary base},
Adv.\ Math.\ \textbf{226} (2011), no.\ 2, 1756--1771.

\end{thebibliography}
\end{document}